\documentclass{amsart}
\usepackage[margin=1.5in]{geometry}
\allowdisplaybreaks[1]

\usepackage{breakurl}

\usepackage{color}
\usepackage{tikz-cd}
\usetikzlibrary{arrows}
\usepackage{amsmath,amssymb}
\usepackage{graphicx}
\usepackage{mathrsfs}
\usepackage{mathabx}
\usepackage[all]{xy}
\usepackage{adjustbox}
\usepackage[normalem]{ulem}
\usepackage{spectralsequences}
\usepackage{tcolorbox}
\usepackage{xfrac}
\usepackage{soul}
\usepackage{varwidth}

\usepackage{hyperref}
\usepackage[capitalise, noabbrev]{cleveref}
\usepackage[T1]{fontenc}
\usepackage[utf8]{inputenc}
\usepackage{caption}
\usepackage{subcaption}
\usepackage{multicol}

\newtheorem{theorem}[equation]{Theorem}
\newtheorem{lemma}[equation]{Lemma}
\newtheorem{proposition}[equation]{Proposition}
\newtheorem{cor}[equation]{Corollary}
\newtheorem{definition}[equation]{Definition}
\newtheorem{remark}[equation]{Remark}

\newtheorem*{thm*}{Theorem}
\newtheorem*{cor*}{Corollary}
\newtheorem*{lem*}{Lemma}
\newtheorem*{prop*}{Proposition}

\theoremstyle{plain}

\numberwithin{equation}{section}
\numberwithin{figure}{section}
\DeclareMathOperator{\Ext}{Ext}
\DeclareMathOperator{\Tor}{Tor}
\DeclareMathOperator{\Hom}{Hom}
\DeclareMathOperator{\Prim}{Prim}

\DeclareMathOperator*{\colim}{colim}

\DeclareMathOperator{\Sq}{Sq}
\DeclareMathOperator{\gr}{gr}

\newcommand{\mC}{{\mathbb C}}

\newcommand{\mF}{{\mathbb F}}

\newcommand{\mM}{{\mathbb M}}
\newcommand{\mN}{{\mathbb N}}

\newcommand{\mR}{{\mathbb R}}

\newcommand{\mZ}{{\mathbb Z}}

\newcommand{\cA}{{\mathcal A}}
\newcommand{\cB}{{\mathcal B}}

\newcommand{\cE}{{\mathcal E}}

\newcommand{\cM}{{\mathcal M}}

\newcommand{\umF}{\underline{\mF}}

\newcommand{\umZ}{\underline{\mZ}}

\newcommand{\0}{\langle 0 \rangle }
\newcommand{\one}{\langle 1 \rangle }

\newcommand{\n}{\langle n \rangle }

\definecolor{green'}{HTML}{009E73}
\definecolor{pink'}{HTML}{CC79A7}
\definecolor{orange'}{HTML}{D55E00}
\definecolor{darkblue'}{HTML}{0072B2}
\definecolor{yellow'}{HTML}{F0E442}
\definecolor{lightblue'}{HTML}{56B4E9}
\definecolor{gold'}{HTML}{E69F00}

\usepackage[textwidth=25mm, textsize=tiny]{todonotes}

\usepackage[style=alphabetic]{biblatex}

\title{A spectrum-level splitting of the $ku_\mathbb{R}$-cooperations algebra}
\author{Guchuan Li, Sarah Petersen, Elizabeth Tatum }

\definecolor{blue}{HTML}{0072B2}
\definecolor{green}{HTML}{009E73}
\definecolor{red}{HTML}{D55E00}

\begin{document}

\begin{abstract}
    In the 1980's, Mahowald and Kane used integral Brown--Gitler spectra to decompose $ku \wedge ku$ as a sum of finitely generated $ku$-module spectra. This splitting, along with an analogous decomposition of $ko \wedge ko,$ led to a great deal of progress in stable homotopy computations and understanding of $v_1$-periodicity in the stable homotopy groups of spheres. In this paper, we construct a $C_2$-equivariant lift of Mahowald and Kane's splitting of $ku \wedge ku$. We also describe the resulting $C_2$-equivariant splitting in terms of $C_2$-equivariant Adams covers and record an analogous splitting for $H\umZ \wedge H \umZ$. Along the way, we give complete computations of the $ku_{\mR}$ and $H \mZ$ operations and cooperations algebras.
\end{abstract}

\maketitle

\tableofcontents

\section{Introduction}\label{sec:introduction}
\subsection{Motivation} In the 1970's, Brown--Gitler and Cohen constructed Brown--Gitler spectra, a family of spectra realizing subcomodules of the dual Steenrod algebra $A_{*}$ \cite{BrownGitler73,cohen1979geometry}. Mahowald subsequently observed that these spectra could be used to decompose the cooperations algebra for integral homology, $H\mathbb{Z} \wedge H\mathbb{Z}$, as a sum of finite $H\mathbb{Z}$-modules. Mahowald and Kane later used integral Brown--Gitler spectra, a family of spectra realizing subcomodules of $H\mF_{2_{*}}H\mathbb{Z}$, to decompose $ku \wedge ku$, the cooperations algebra for connective complex $K$-theory, and $ko \wedge ko$, the cooperations algebra for connective real $K$-theory, in an analogous way \cite{mahowald1981bo, kane1981operations}. 

These splittings of $ku \wedge ku$ and $ko \wedge ko$ proved particularly useful in $ku$- or $ko$-based Adams spectral sequences. As $ku_*ku$ or $ko_*ko$ is not flat over $ku_*$ or $ko_*$, the corresponding $E_2$-page cannot be computed as an $\Ext$ group. One has to start the computation with the $E_1$-page. 
However, Mahowald and Kane's decompositions of $ku \wedge ku$ and $ko \wedge ko$ in terms of finitely generated $ku$- and $ko$-modules, respectively, made computation of the $E_1$-pages of these spectral sequences tractable by splitting the computation of the $E_1$-page into smaller pieces \cite{mahowald1981bo, Davis87, DavisMahowald89, BeaudryBehrensBattacharyaCulverXu20}. With a little more work, these splittings also make computation of the $E_2$-page accessible. This is incredibly useful because even the $E_2$-pages of these spectral sequences provide a wealth of homotopical information. Both the $ku$- and $ko$-based Adams spectral sequences have vanishing lines which result in collapse at the $E_{2}$-page in a large range.

The $ko$-based Adams spectral sequence in particular is a very effective tool for studying $v_{1}$-periodicity in the stable homotopy groups of the sphere, $\pi_{*}S_{(2)}.$ For example, Mahowald used $ko$-based Adams spectral sequences to give a complete identification of the $v_1$-periodic elements in the stable homotopy groups of spheres \cite{Mahowald82}.
He also used the $ko$-based Adams spectral sequence to prove the Telescope Conjecture for height one at the prime $2$ \cite{mahowald1981bo}. Moreover, the $ko$-based Adams spectral sequence is a highly efficient way to compute $\pi_{*}S_{(2)}$ through the $40$-stem \cite{Davis87, DavisMahowald89, BeaudryBehrensBattacharyaCulverXu20}. Additionally, at odd primes, Gonzalez used the $ku$-based Adams spectral sequence to study $v_{1}$-periodicity, as well as to classify stunted lens spaces \cite{Gonzalez00}. 

Given the range of successful applications of Mahowald and Kane's splittings in nonequivariant homotopy theory, one may ask whether analogous splittings exist in equivariant homotopy theory. In this paper, we answer this question affirmatively by constructing a $C_2$-equivariant splitting of $ku_\mR \wedge ku_\mR$ in terms of finitely generated $ku_\mR$-modules. We also show how this splitting leads to a complete computation of the $E_1$-page of the $ku_\mR$-based Adams spectral sequence and give partial information towards computing the $E_2$-page. In future work, we plan to investigate the $E_2$-page and beyond. Here, and throughout the paper, $ku_\mR$ is the equivariant connective cover of $KU_\mathbb{R},$ the spectrum representing Atiyah's real $K$-theory. We also work in the category of $2$-complete spectra, often omitting the $2$-complete notation.

There are several reasons for working with the group of equivariance $G = C_2,$ the cyclic group of order two. First, Mahowald and Kane's nonequivariant splittings involve (integral) Brown--Gitler spectra which realize subcomodules of the dual Steenrod algebra $A_{*},$ or in the case of integral Brown--Gitler spectra, realize subcomodules of $H {\mF_2}_{*}H\mathbb{Z}$. Thus it is reasonable to begin with a group of equivariance where the dual Steenrod algebra has been computed and closely resembles the classical nonequivariant dual Steenrod algebra. This is the case when $G = C_2$ \cite[Theorem~6.41]{HuKriz2001}, and, in an earlier paper, we constructed $C_2$-equivariant integral Brown--Gitler spectra $\cB_0(k)$ realizing subcomodules of $H {\umF_2}_\star H \umZ$ \cite{LiPetersenTatum23}. These $C_2$-equivariant integral Brown--Gitler spectra play an integral role in our construction of a $C_2$-equivariant splitting of $ku_\mR \wedge ku_\mR$.

For cyclic groups of prime order $p \neq 2,$ the dual Steenrod algebra has been computed and is known not to be flat over the coefficients $H {\umF_p}_\star$ \cite{SanW2il022,KuKrizSomberZou23}. This suggests the construction of $C_p$-equivariant (integral) Brown--Gitler spectra, where $p$ is an odd prime, may be more complicated, requiring techniques beyond those developed in \cite{LiPetersenTatum23}. Additionally, odd primary analogue of $BP_\mR$ has not been constructed yet \cite{HHR2011,HahnSengerWilson23}.

Even in the case where $G = C_2,$ constructing a splitting of $ku_\mR \wedge ku_\mR$ is significantly more complicated than in the nonequivariant case. This is largely due to the fact that the coefficients $H {\umF_2}_\star$ form a bigraded ring (described in \cref{sec:C2point}) and that the $C_2$-equivariant Steenrod algebra has the structure of a Hopf algebroid rather than a Hopf algebra. Many of the results in \cref{sec:C2homology}, including our Margolis homology results, are designed to deal with this more complicated algebraic context. Further, much of the work in Sections \ref{sec: height zero} and \ref{sec:height1split}, where we construct the $H \umZ \wedge H \umZ$ and $ku_\mR \wedge ku_\mR$ splittings, lies in keeping careful track of all the bigraded elements. 

We now describe our main results. This is followed by a discussion of problems which are newly accessible with the methods developed in this paper. 

\subsection{Main results} 

The main result of this paper is a lift of Mahowald and Kane's splitting of $ku \wedge ku$ to $C_{2}$-equivariant spectra. In order to state this result, let $\rho$ be the regular representation and $\cB_{0}(k)$ be the $C_2$-equivariant integral Brown–Gitler spectra constructed in \cite{LiPetersenTatum23} and discussed in \cref{remark:intBG}.

\begin{thm*}[\cref{thm:mainTheorem}] 
Up to $2$-completion, there is a splitting of $ku_\mR$-modules
    \[ku_\mR \wedge ku_\mR  \simeq \overset{\infty}{\underset{k=0}{\bigvee}} ku_\mR \wedge \Sigma^{\rho k}\cB_{0}(k).\]
\end{thm*} 

Our proof requires the development of a number of $C_2$-equivariant homology results. Specifically, we give a Whitehead Theorem for Margolis homology in the $C_2$-equivariant setting (\cref{prop:EquivWhitehead}), compute the homology of $C_2$-equivariant mod $2$ and integral Brown--Gitler spectra (Propositions \ref{equiv:BG:free} and \ref{bg lightning split}), and record a computation of $H_\star BP_\mR \langle n \rangle$ which we first learned from Christian Carrick. Here, and throughout the remainder of the paper, we let $H$ denote $H\umF_{2},$ the Eilenberg-MacLane spectrum for the $C_2$-constant Mackey functor $\umF_2.$ 

Just as in the nonequivariant setting, we can also describe $ku_\mR \wedge ku_\mR$ in terms of equivariant Adams covers  $ku_{\mR}^{\langle \nu_{2}(n!) \rangle}$ of $ku_{\mR}$. 

\begin{cor*}[\cref{thm:C2AdamsCovers}]
Up to $2$-completion,    \[ ku_{\mR} \wedge ku_{\mR} \simeq \bigvee\limits_{k=0}^{\infty} \Sigma^{\rho k} ku_{\mR}^{\langle \nu_{2}(n!) \rangle} \vee V,\]
where $V$ is a sum of suspensions of $H\umF_{2}$.
\end{cor*}
Given these splittings, we describe how to obtain the $E_1$-page of the $ku_\mR$-based Adams spectral sequence in \cref{sec:E1kuR}.

Incidentally, in our proof of the main theorem, we describe the cooperations algebra $ku_{{\mR}_\star}ku_\mR$. 
\begin{thm*}[\cref{thm:cooperationsAlg}]
The $ku_\mR$-cooperations algebra ${ku_\mR}_\star ku_\mR$ splits as
    \[
    {ku_\mR}_\star ku_\mR \cong \overset{\infty}{\underset{k=0}{\bigoplus}}  {ku_\mR}_{\star - \rho k} \cB_0(k), 
    \]
    where as a ${ku_\mR}_\star$-module 
    \begin{align*}
        {ku_\mR}_{\star}\cB_{0}(k) 
        & \cong  H \umZ_{\star}\{x_{0}, x_{1}, \ldots, x_{\nu_2(2k!)-1} \} \oplus {ku_\mathbb{R}}_{\star}\{x_{\nu_2(2k!)} \} \oplus V_{k},
    \end{align*}
    with extensions $v_1 x_{i - 1} = \rho x_i$, and where $| x_i | = \rho i$ and $V_k$ is a sum of suspensions of $H_\star$.
\end{thm*}

Here, and throughout this paper, we overload the symbol $\rho$, using it to denote both the $C_2$-regular representation and the nontrivial element in ${H \umF_2}_{\{-\sigma\}}$. This choice is made in order to mirror the notation commonly used in the motivic homotopy theory methods which are also used in this paper. Whether $\rho$ denotes the $C_2$-regular representation or the element in ${H \umF_2}_{\{-\sigma\}}$ will always be clear from context.

The methods we use to compute the $ku_\mR$-cooperations algebra also yield an analogous $\mR$-motivic calculation. Specifically, we have the following corollary.

\begin{cor*}[\cref{cor:kglcooperations}]
     The $kgl$-cooperations algebra $kgl_{*,*} kgl$ is given by 
    \[
        kgl_{*,*} kgl \cong \bigoplus\limits_{k=0}^\infty BPGL \langle 0 \rangle_{*,*} \{ x_0, x_1, \cdots, x_{\nu_2 (2k!) - 1} \} \oplus kgl_{*, *} \{ x_{ \nu_2 (2k!)} \} \oplus V_k,
    \]
    with extensions $v_1 x_{i - 1} = \rho x_i,$ and where $\vert x_i \vert = (2i, i)$ and $V_k$ is a sum of suspensions of $\mM_2^\mR.$
\end{cor*}

While we do not explicitly identify the locations of the summands $H\umF_2$ in $V$, it follows from our computations that there will be exactly one copy of  $H\umF_2$ in degree $(2k,k)$ for each copy of $H\mF_2$ in degree $k$ of $ku \wedge ku$. These can be identified using the techniques of \cite{carlsson1980stable}, which covers the more complicated case of identifying those summands in $ko \wedge ko$ and the $E_1$-page of the $ko$-Adams spectral sequence. 

Here, $kgl$ is the lift of $ku_\mR$ along the Betti realization functor (see for example \cite[Section 4.4]{HellerOrmsby16}), which is a symmetric monoidal left adjoint from the $\mR$-motivic stable homotopy category to the category of genuine $C_2$-spectra. Similarly, $BPGL \langle 0 \rangle$ is the lift along Betti realization of $BP_\mR \langle 0 \rangle.$ Our methods do not yield a splitting of $kgl \wedge kgl$ in terms of Brown--Gitler spectra solely because we do not know of a construction yielding $\mR$-motivic (integral) Brown--Gitler spectra. This is due to the fact that the known $C_2$-equivariant construction relies on a Thom spectrum model, which is not known to exist in the $\mR$-motivic setting. See \cite{MorrisPetersenTatum2025} for a splitting of $kgl \wedge kgl$ constructed in terms of Adams covers. 

Returning to the $C_2$-equivariant setting, we also describe the operations algebra $[ku_\mR, ku_\mR]$. 
\begin{thm*}[\cref{thm:operationsAlgebra}]
     The operations algebra $[ku_\mR, ku_\mR]$ splits as
     \[ [ku_\mR, ku_\mR] \cong \bigoplus\limits_{k=0}^{\infty}[\Sigma^{\rho k}\cB_{0}(k), ku_\mR] .\]
     The Adams spectral sequence \[\Ext_{\cE(1)_{\star}}(H_{\star}\cB_{0}(k), H_{\star}ku_\mR) \Longrightarrow [\Sigma^{\rho k}\cB_{0}(k), ku_\mR]\] collapses at the $E_{2}$-page, and its $E_{2}$-page is described in \cref{lem:k>m}.
\end{thm*}

Similarly to the $ku_\mR$-cooperations algebra, the methods we use to compute the $ku_\mR$-\newline operations algebra also yield an analogous $\mR$-motivic calculation. Specifically, the operations algebra $[kgl, kgl]$ can be obtained from \cref{thm:operationsAlgebra} by omitting the negative cone summands. 

Returning again to the $C_2$-equivariant setting, we also record an analogous splitting of $H\umZ \wedge H\umZ$. 
\begin{thm*}[\cref{thm: height zero splitting}]
     Up to $2$-completion there is a splitting
\[ H \umZ \wedge H \umZ \simeq \bigvee\limits_{k=0}^{\infty} H \umZ \wedge \Sigma^{\rho k}\cB_{-1}(k) \]
of $H \umZ$-modules.
\end{thm*}
We also obtain the following corollary, which can be compared with \cref{thm:C2AdamsCovers}.
\begin{cor*}[\cref{cor: height zero simples}] There is a splitting
    \[  H \umZ \wedge H \umZ \simeq H \umZ \vee Z,\]
where $W$ is a sum of suspensions of $H\umF_{2}$.
\end{cor*}

Just as in the nonequivariant setting, $H\umZ$ and $ku_{\mR}$ fit into a larger context of truncated Brown--Peterson spectra.
At the prime $2$, $ku_\mR \simeq BP_\mR \langle 1 \rangle$, while at odd primes, even analogues of $BP_\mathbb{R}$ have not yet been defined. We show that, similarly to the nonequivariant setting, there is a splitting in the homology of $BP_\mR \langle n \rangle \wedge BP_\mR \langle n \rangle$ at all heights. Specifically, we have the following theorem.
\begin{thm*}[\cref{thm:homology:decomp}]
For $n \geq 0$, there exists a family of maps
$$\{\theta_k: H_\star\Sigma^{\rho k} B_{n - 1}(k) \to H_\star BP_\mR \langle n \rangle \, | \, k \in \mN \} $$
such that their sum
$$  \theta : \bigoplus_{k = 0}^\infty  H_{\star - \rho k} B_{n - 1}(k) \to H_\star BP_\mR \langle n \rangle $$
is an isomorphism of $\mathcal{E}(n)_\star$-comodules. 
\end{thm*}

\subsection{New directions}
Similarly to the nonequivariant setting, we expect the splitting of $ku_\mR \wedge ku_\mR$ (\cref{thm:mainTheorem}) will make $ku_\mR$-based Adams spectral sequence computations newly accessible. We give some indication of this in \cref{sec:E1kuR} where we describe how the splitting allows us to easily compute the $E_1$-page of the $ku_{\mR}$-based Adams spectral sequence for the sphere. In future work, we plan to investigate whether the $ku_\mR$-based Adams spectral sequence allows us to extend Isaksen--Guillou's computations of $C_2$-equivariant stable stems \cite{GuillouIsaksen24}. Also, similarly to the nonequivariant setting, the $ku_\mR$-based Adams spectral sequence should have a vanishing line allowing one to identify $v_1$-periodicity in the $C_2$-equivariant stable stems. 

One may further wonder whether $ko_{C_2}$-based Adams spectral sequences are even more efficient for computing $C_2$-equivariant stable stems and detecting $v_1$-periodicity. Here, $ko_{C_2}$ denotes the equivariant connective cover of $KO_{C_{2}},$ the spectrum representing the $K$-theory of $C_{2}$-equivariant real vector bundles. Thus one may also be interested in constructing a $C_2$-equivariant splitting for $ko_{C_2} \wedge ko_{C_2}$ in terms of finitely generated $ko_{C_2}$-modules. The techniques and methods of this paper can be used as a stepping stone towards such a construction. 

Additionally, computations with $ku_\mR$-based Adams spectral sequences are closely related to those of $ko_{C_2}$-Adams spectral sequences via the Wood cofiber sequence 
\[
\Sigma^\sigma ko_{C_2} \xrightarrow{\eta} ko_{C_2} \to ku_{\mR},
\]
where $\eta$ is the first $C_2$-equivariant Hopf map. Thus in a precise sense $ku_\mR$-based computations already contain $ko_{C_2}$ information.

Another consequence of our Adams spectral sequence approach is partial progress towards constructing higher height $C_2$-equivariant Brown--Gitler spectra. Specifically, constructing $ku_\mR$ or $ko_{C_2}$-Brown--Gitler spectra and $BP_\mR \langle 2 \rangle$-Brown--Gitler spectra via an obstruction theoretic approach (analogous to the nonequivariant constructions of \cite{goerss1986some,klippenstein1988brown}) requires constructing certain maps $ku_\mR \to H \umZ$ and $BP_\mR \langle 2 \rangle \to ku_\mR,$ respectively. This reduces to studying certain summands in the $H \umZ$ and $ku_\mR$ relative Adams spectral sequences which we describe completely in \cref{sec: height zero} and \cref{sec:height1split}. Finishing the obstruction theoretic construction requires some additional analysis with the $C_2$-equivariant lambda algebra, which may prove difficult due to the complicated nature of the $C_2$-equivariant lambda algebra. However, nonequivariantly, such Brown--Gitler spectra have been useful in height $2$ homotopy calculations \cite{BOSS19} and thus if constructed could prove useful in $C_2$-equivariant homotopy calculations. 

In addition to computational applications of the splitting, the cooperations algebra ${ku_\mR}_\star ku_\mR$ (\cref{thm:cooperationsAlg}) is of interest due to connections with number theory. In the underlying nonequivariant setting, $ku_{*}ku$ can be described in terms of numerical polynomials. In particular, $KU_{*}KU$ can be identified with the ring of finite Laurent series satisfying certain conditions \cite[Theorem~13.4]{Adams74}. This extends to a description of $KU_{*}ku$ in terms of numerical polynomials \cite[Theorem~17.4]{Adams74}, as well as to the torsion-free component of $ku_*ku$ \cite[p.~358]{Adams74}. Using this comparison, one can identify the summands of the splitting of $ku_{*}ku$ in $KU_{*}KU$.  A detailed discussion of this correspondence (and the analogous story for $ko$) can be found in Sections 3.2 and 3.3 of \cite{BOSS19}. Our computation of ${ku_\mR}_\star ku_\mR$ gives a starting point for developing a similar description in the $C_2$-equivariant setting. 

Similarly to the splitting of $ku_\mR \wedge ku_\mR$, the decomposition of $H \umZ \wedge H \umZ$ into a sum of finite $H \umZ$-modules given by \cref{thm: height zero splitting} and \cref{cor: height zero simples} provides a starting point for $H\umZ$-based Adams spectral sequence computations. For example, Burklund-Pstr\k{a}gowski use the nonequivariant analogue of \cref{cor: height zero simples} to give a cleaner and more concise rephrasing of Toda's original obstruction-theoretic approach to the construction of $BP$ in terms of the $H\mZ$-Adams spectral sequence \cite[Theorem~4.41]{burklund2023quiversadamsspectralsequence}. It would be interesting to investigate such an argument in the $C_2$-equivariant case. 

\subsection{Outline of the Paper}

In \cref{sec:UnderlyingSplittings}, we discuss classical nonequivariant splittings and outline the nonequivariant relative Adams spectral sequence argument underlying our equivariant arguments. In \cref{sec: equivariant preliminaries}, we recall equivariant foundations necessary for our main computations. In \cref{sec:C2homology} we introduce $C_2$-equivariant homology results which will later be used in our constructions of splittings. In particular, we describe the homology of $BP_\mR \langle n \rangle,$ state some $C_2$-equivariant Margolis homology results, and compute the homology of $C_2$-equivariant (integral) Brown--Gitler spectra in terms of $C_2$-equivariant lightning flash modules. We also prove a splitting in homology at all heights $n.$ In \cref{sec: height zero}, we construct a $C_2$-equivariant spectrum-level splitting for $H \umZ \wedge H \umZ.$ In \cref{sec:height1split}, we construct the splitting for $ku_\mR \wedge ku_\mR,$ deduce the $ku_\mR$-cooperations and operations algebras, and describe how to compute the $E_1$-page of the $ku_\mR$-Adams spectral sequence for the sphere using the splitting of $ku_\mR \wedge ku_\mR$.

\subsection{Notation}
We will make use of the following notation:
\begin{enumerate}
    \item $H = H \umF_2$ denotes the Eilenberg--MacLane spectrum associated to the $C_2$-constant Mackey functor $\umF_2.$
    \item $\mM_2^\mC = \mF_2[\tau]$ is the motivic homology of $\mC$ with $\mF_2$ coefficients, where $\tau$ has bidegree $(0,-1).$
    \item $\mM_2^\mR = \mF_2 [\tau, \rho]$ is the motivic homology of $\mR$ with $\mF_2$ coefficients, where $\tau$ and $\rho$ have bidegrees $(0,-1)$ and $(-1,-1)$ respectively. 
    \item $\mM_2$ is the bigraded equivariant homology of a point with coefficients in the constant Mackey functor $\umF_2.$ See \cref{sec:C2point} for a description.
    \item $A,$ $\cA^\mC,$ $\cA^\mR$ and $\cA$ are the classical, $\mC$-motivic, $\mR$-motivic, and $C_2$-equivariant mod $2$ Steenrod algebras. 
    \item $E(n)$ is the subalgebra of $A$ generated by $Q_0 = \Sq^1, \, Q_1=\Sq^1 \Sq^2 + \Sq^2 \Sq^1,$ $\cdots, Q_n.$ The analogously defined subalgebras of $\cA^\mC,$ $\cA^\mR,$ and $\cA$ are denoted $\cE^\mC(n),$ $\cE^\mR(n),$ and $\cE(n),$ respectively. 
    \item $E(n)_*,$ $\cE^\mC_\star (n),$ $\cE^\mR_\star (n),$ and $\cE(n)_\star$ (see \cite[Definition~6.5]{Ricka15}) denote the duals of the above subalgebras, respectively.
    \item The square-zero extension $\mM_{2} \cong \mM_{2}^{\mR} \oplus NC$ induces a decomposition \cite[p.8]{GuillouHillIsaksenRavenel2020} 
\[ \Ext_{\cE(n)_\star}(\mM_{2}^{C_2}, \mM_{2}^{C_2}) \cong \Ext_{\cE_{\star}^{\mR}(n)}(\mM_{2}^{\mR}, \mM_{2}^{\mR}) \oplus \Ext_{\cE_{\star}^{\mR}(n)}(NC, \mM_{2}^{\mR})  .\]
We will abuse notation by writing 
\[ \Ext_{\cE_{\star}^{\mR}(n)}(\mM_{2}, \mM_{2}):= \Ext_{\cE_{\star}^{\mR}(n)}(\mM_{2}^{\mR}, \mM_{2}^{\mR}) \]
and 
\[ \Ext_{\cE_{\star}^{NC}(n)}(\mM_{2}, \mM_{2}):= \Ext_{\cE_{\star}^{\mR}(n)}(NC, \mM_{2}^{\mR}) .\]

Furthermore, we will inductively compute $\Ext_{\cE(n)_\star}(M, N)$ for various modules $M$ and $N$, which will split into a ``positive cone summand'' coming from \newline $\Ext_{\cE_\star^{\mR}(n)}(\mM_{2}, \mM_{2})$ and a ``negative cone summand'' coming from \newline $\Ext_{\cE^{NC}_\star(n)}(\mM_{2}, \mM_{2})$. We will denote these summands respectively as \newline $\Ext_{\cE(n)_{\star}^+}(M, N)$ and $\Ext_{\cE(n)_{\star}^-}(M, N)$.

    \item $\frac{\mF_2 [\tau]}{\tau^\infty}$ is the $\mF_2 [\tau]$-module $\colim_n \mF_2 [\tau] / \tau^n.$ 
    We write $\frac{\mF_2 [\tau]}{\tau^\infty} \{x\}$ for the infinitely divisible $\mF_2 [\tau]$-module consisting of elements of the form $\frac{x}{\tau^k}$ for $k\geq 1.$ 
    \item $MU_\mR$ is the spectrum $MU$ with complex conjugation action.
    \item $BP_\mR$ is the $C_2$-equivariant analogue of the nonequivariant Brown--Peterson spectrum $BP$ with complex conjugation action.
    \item $BP_\mR \langle n \rangle$ models the classical truncated Brown–Peterson spectrum $BP \langle n \rangle$ with $C_2$-action via complex conjugation.
    \item $ku_\mR$ is the equivariant connective cover of $KU_\mathbb{R},$ the spectrum representing Atiyah real $K$-theory.
    \item $ko_{C_{2}}$ is the equivariant connective cover of $KO_{C_{2}},$ the spectrum representing the $K$-theory of $C_{2}$-equivariant real vector bundles.
    \item Given a $C_2$-spectrum $X,$ we use $X^e$ to denote the underlying nonequivariant spectrum.
    \item Given a class $x \in H_\star X$ and a representation $V,$ we let $\Sigma^V x$ denote the image of $x$ under the isomorphism $H_{\star}X\cong H_{\star+V}\Sigma^V X$.
 \end{enumerate}

 Grading conventions: Consider the real representation ring of $C_2,$ $RO(C_2) \cong \mZ [\sigma] / (\sigma^2 -1),$ where $\sigma$ is the one-dimensional real sign representation. We express the equivariant degree $i + j\sigma$ according to the motivic convention as $(i + j, j)$ where $i + j$ is the total degree and $j$ is the weight. We will also at times use representation spheres to denote the appropriate suspension. For example, instead of writing $\Sigma^{2,1},$ we will write $\Sigma^\rho$ where $\rho$ is the $C_2$-regular representation. Whether $\rho$ is an element in the homology of the point or the $C_2$-regular representation will be clear from context.  

We grade $\Ext$-groups in the form $(s,f,w)$, where $s$ is the stem, that is, the total degree minus the homological degree, $f$ is the Adams filtration, that is the homological degree, and $w$ is the weight. We will also refer to the Milnor--Witt degree, which is $s-w.$

\subsection*{Acknowledgments}

The authors would like to thank Mark Behrens, Christian Carrick, Mike Hill, and Vesna Stojanoska for enlightening conversations. The first and second authors would also like to thank the Max Planck Institute for Mathematics in Bonn for its hospitality and financial support. The second and third authors would like to thank the Isaac Newton Institute for Mathematical Sciences, Cambridge, for support and hospitality during the program Equivariant homotopy theory in context, where work on this paper was undertaken. This work was supported by EPSRC grant EP/Z000580/1. The third author thanks the Hausdorff Institute in Bonn for its hospitality and financial support, as well as the Knut and Alice Wallenberg Foundation for financial support. All three authors are grateful to the Hausdorff Institute in Bonn for its hospitality during the Spectral Methods in Algebra, Geometry, and Topology trimester program in Fall of 2022, funded by the Deutsche Forschungsgemeinschaft (DFG, German Research Foundation) under Germany's Excellence Strategy- EXC-2047/1-390685813. This material is based upon work supported by the National Science Foundation under Grant No. DMS 2135884.

\section{Underlying nonequivariant splittings} \label{sec:UnderlyingSplittings}
We begin by summarizing nonequivariant splittings involving $ku$ and $ko.$ We then outline a construction of the splitting of $ku \wedge ku$ via an Adams spectral sequence. This argument underlies our approach in Sections \ref{sec: height zero} and \ref{sec:height1split} where we construct $C_2$-equivariant splittings of $H \umZ \wedge H \umZ$ and $ku_\mR \wedge ku_\mR$. 

Before we begin, it is helpful to note some context for splittings involving $ku$. At all primes, $H\mZ \simeq BP\langle 0 \rangle$. At the prime $2$, $ku \simeq BP\one$, and at odd primes, $ku$ splits as a sum of suspensions of $BP\one$. Thus the splitting of $ku \wedge ku$ at the prime $2$ fits into a broader story of splittings of $BP \langle 1 \rangle \wedge BP \langle 1 \rangle$ at all primes, which we describe briefly below.

In order to state these nonequivariant splittings, we first introduce nonequivariant integral Brown--Gitler spectra. 

\subsection{Nonequivariant Brown--Gitler spectra}
Let $A$ denote the classical Steenrod algebra, and let $A_{*} := H\mF_{p_{*}}H\mF_{p}$ denote its dual. When $p=2$, 
\[A_{*} \cong \mF_{2}[\xi_{1}, \xi_{2},\ldots ],
\]
where $| \xi_i | = 2^i - 1$. When $p$ is odd, 
\[A_{*} \cong \mF_{2}[\xi_{1}, \xi_{2},\ldots ] \otimes E(\tau_{0}, \tau_{1}, \ldots),
\]
where $|\xi_i| = 2(p^i -1)$ and $| \tau_i | = 2 p^i - 1$. 

The weight filtration on $A_{*}$ is defined by setting 
\[
wt({\xi}_{k}) = wt({\tau}_{k}) = p^{k}
\]
and 
\[
wt(xy) = wt(x) + wt(y),
\]
where $\bar{\xi}_i$ and $\bar{\tau}_i$ denote the conjugates of the usual generators of the dual Steenrod algebra.
Let $E(n)$ denote the subalgebra $E(n) = E(Q_{0}, Q_{1}, \ldots Q_{n})$. Then
\[H_{*}BP\langle n \rangle \cong A//E(n)_{*},\]
where the right-hand side is the dual of the Hopf algebra quotient $A//E(n)$ \cite[Prop~1.7]{wilson1975omega}. 
In particular, $H_{*}H\mZ \cong A//E(0)_{*}$. 

At each prime $p$, there is a family of integral Brown--Gitler spectra $\{B_{0}(k)|\ k \in \mathbb{N} \}$ such that $H_{*}B_{0}(k) \cong \mF_{p}\{x \in H_{*}H\mathbb{Z} \mid  wt(x) \le k\}$ \cite{goerss1986some}. Later we will use a similar notation, $\mathcal{B}_{0}(k),$ to denote the $C_2$-equivariant integral Brown--Gitler spectra constructed in \cite{LiPetersenTatum23}.

\subsection{Nonequivariant Splittings}
\subsubsection{Literature on nonequivariant splittings}
In \cite{mahowald1981bo}, Mahowald constructed a splitting \[ko \wedge ko \simeq \overset{\infty}{\underset{k=0}{\bigvee}} ko \wedge \Sigma^{4k} B_{0}(k).\]
Subsequently, Kane \cite{kane1981operations} constructed a splitting 
\[ 
BP\one \wedge BP\one \simeq \overset{\infty}{\underset{k=0}{\bigvee}}BP\one \wedge \Sigma^{2k} B_{0}(k) 
\] 
at odd primes.

While a construction of the splitting for $ku$ at the prime $2$ does not appear in the literature, it can be deduced either from Mahowald's construction for $ko$, or by using the same arguments as Kane's odd primary splitting. In \cite{klippenstein_splitting}, Klippenstein also constructed a decomposition of $BP\one \wedge BP\langle n \rangle$ in terms of $BP\one$-module spectra at odd primes, and notes that the same statements hold at the prime $2$.

\subsubsection{Strategies for the nonequivariant splitting} For the purpose of comparison with the equivariant splitting of $ku_\mR \wedge ku_\mR$ in \cref{sec:height1split}, we focus on the nonequivariant splitting of $BP \langle 1 \rangle \wedge BP \langle 1 \rangle.$ We begin with a brief discussion of Kane's approach \cite{kane1981operations}. 

Following Mahowald's strategy for the construction of the $ko$-splitting, Kane used pairings $B_{0}(k) \wedge B_{0}(m) \to B_{0}(k+m)$ to inductively construct maps
\[\tilde{\theta}_{k}: \Sigma^{2k} B_{0}(k) \to BP\one \wedge BP\one \]
such that the sum of the composites
\[ 
\bigvee\limits_{k=0}^{\infty} BP\one \wedge \Sigma^{2k} B_{0}(k) \xrightarrow{1 \wedge \tilde{\theta}_{k}}  BP\one^{\wedge 3} \xrightarrow{\mu \wedge 1} BP\one^{\wedge 2} 
\] is a homotopy equivalence. 

To construct this splitting, assume inductively that appropriate maps 
\[
\tilde{\theta}_{k}: \Sigma^{2k}B_{0}(k) \to BP\one \wedge BP\one
\]
have been constructed. If $k$ is not a power of $p$, take the $p$-adic decomposition $k = k_{0} + k_{1}p + \cdots + k_{r}p^{r}$, and consider the map
\[\Sigma^{2k} B_{0}(1)^{k_{0}} \wedge B_{0}(p)^{k_{1}} \wedge \cdots \wedge B_{0}(p^{r})^{k_{r}} \to B_{0}(k). \]
 (If $k = p ^{s}$ for some $s$, take the same approach but with the decomposition $B_0(p^{s-1})^{\wedge p}$.) Consider the composite
\begin{align*}
    \Sigma^{2k} B_{0}(1)^{k_{0}} \wedge B_{0}(p)^{k_{1}} \wedge \cdots \wedge B_{0}(p^{r})^{k_{r}}  \xrightarrow{\tilde{\theta}_{0}^{k_{0}} \wedge \tilde{\theta}_{1}^{k_{1}} \wedge \tilde{\theta}_{r}^{k_{r}}}  BP\one & \wedge BP\one \wedge  \cdots \wedge BP\one  \\ \xrightarrow{\qquad \qquad} BP\one & \wedge BP\one,
\end{align*}
where the second map is the standard multiplication map. The key step in the construction is to show that this composite factors through $B_{0}(k)$. One can then show that this factorization is the appropriate map $\tilde{\theta}_k: \Sigma^{2k}  B_{0}(k) \to BP\one \wedge BP\one$.

The strategy for showing that the map factors through $B_{0}(k)$ does not lend itself well to the equivariant setting for two reasons: first, it relies in part on connectivity arguments that the negative cone prevents us from replicating. Second, extending this argument to the equivariant setting would require us to deduce that in homology, various complicated cofibers of free $\mM_{2}$-modules are also free. While these obstacles may be surmountable, we have found that the alternative strategy of using an Adams spectral sequence to lift the isomorphism on homology is most readily adaptable to the equivariant setting. Since this argument does not appear in the literature, we present it here. This Adams spectral sequence argument is most similar to Klippenstein's splitting construction \cite{klippenstein_splitting}. However, we use a $ku$-relative Adams spectral sequence as it makes the construction slightly cleaner. We also do not need his comparison with the Universal Coefficient Spectral Sequence as we only address the case $BP \one \wedge BP\one$, not $BP \one \wedge BP\langle n \rangle $. For ease of comparison with the rest of this paper, we present the $2$-primary version and use the notation $ku$. However, the same argument works for $BP\one$ at all primes.  

\subsubsection{Constructing the nonequivariant splitting}
In order to construct the splitting, we start with the following isomorphism of $E(1)_{*}$-comodules. 
\begin{proposition}\cite[Section~12]{kane1981operations}\cite[Section~4.1] {CulverOddp20}\label{nonequiv: homology splitting} There is an isomorphism 
   \[H_{*}ku \cong \bigoplus_{k=0}^{\infty}H_{*}\Sigma^{2k}B_{0}(k)\]
   of $E(1)_{*}$-comodules.
\end{proposition}
Our goal is to lift this isomorphism of $E(1)_*$-comodules to the level of spectra. To that end, we make use of some $ku$-relative homology and $ku$-relative Adams spectral sequence computations. 

In general, if $E$ is an $R$-algebra and $M$ is an $R$-module in spectra, then $E$-homology in the category of $R$-modules is 
\begin{equation*}
    E_*^R (M) := \pi_{*} (E \underset{R}{\wedge} M).
\end{equation*}
Note that
\begin{align}\label{nonequiv: R's cancel}
\begin{split}
     E^R_*(R \wedge M) & = \pi_*(E \underset{R}{\wedge} R \wedge M) \\
    & \cong E_{*}M.
\end{split}
\end{align}

\begin{proposition}\cite[Prop~2.1]{baker2001adams}\label{nonequiv:relative ASS exists}
    Let $E$ be an $R$-algebra spectrum, and let $X,Y$ be $R$-modules. If $E_{*}^{R}X$ is projective as an $E_{*}$-module, then there exists an $E$-based Adams spectral sequence in the category of $R$-modules
\[
E_{2}^{s,f} \cong \Ext_{E_{\star}^{R} E}^{s,f} (E_\star^R X, E_\star^R Y) \Longrightarrow [X,Y]_{\widehat{E}}^{R \, \, s, f}
\]
 where $[X,Y]^{R}_{\widehat{E}}$ denotes the $E$-nilpotent completion of $R$-module maps from $X$ to $Y$, $s$ denotes the stem, and $f$ denotes the homological degree. 
\end{proposition}
In particular, we use the $ku$-relative Adams spectral sequence to lift the isomorphism of \cref{nonequiv: homology splitting} 
to the splitting 
\[ku \wedge ku \simeq ku \wedge \bigvee\limits_{k=0}^{\infty}\Sigma^{2k} B_{0}(k).\] 
It is helpful to observe that \begin{equation}\label{nonequiv:rel steenrod}H\mF_{2_{*}}^{ku}H\mF_{2} \cong E(1)_{*}.\end{equation}
This isomorphism follows from the same argument as its $C_2$-equivariant analogue, \cref{prop:relHomology}.

Furthermore, we can use (\ref{nonequiv: R's cancel}) to rewrite the isomorphism of \cref{nonequiv: homology splitting} in terms of $ku$-relative homology
\begin{equation}\label{nonequiv: spelled out iso} H_{*}^{ku}\left(ku \wedge \bigvee\limits_{k=0}^{\infty}\Sigma^{2k}B_{0}(k)\right)\to H_{*}^{ku}\left(ku \wedge ku \right) .\end{equation}
Then, just as one can use the classical Adams spectral sequence to lift an isomorphism of $A_{*}$-comodules $H_{*}X \to H_{*}Y$ to a map of spectra $X \to Y$, we can use the $ku$-relative Adams spectral sequence to lift the $E(1)_{*}$-comodule  isomorphism of \cref{nonequiv: homology splitting} to a $ku$-module isomorphism.

Specifically, consider the relative Adams spectral sequence 
\begin{equation}\label{nonequiv:ass}
     \Ext_{E(1)_{*}}^{s,f} \left(H_{*} \Sigma^{2k} B_{0}(k), H_{*}ku \right) \Longrightarrow [ku \wedge \Sigma^{2k}B_{0}(k), ku \wedge ku]^{ku},
\end{equation}

and consider $\theta_{k}: H_{*}\Sigma^{2k}B_{0}(k) \to H_{*}ku$ as a class in filtration $f=0$. If we can show that this class survives the spectral sequence for each $k$, then we have constructed a family of maps realizing the $ku$-relative homology equivalence of (\ref{nonequiv: spelled out iso}), and thus will have a $ku$-module splitting
\[ ku \wedge ku \simeq \bigvee\limits_{k=0}^{\infty}ku \wedge \Sigma^{2k}B_{0}(k) .\]

To compute the $E_{2}$-page of (\ref{nonequiv:ass}), it suffices to compute $\Ext_{E(1)_{*}}(H_{*}B_{0}(k), B_{0}(m))$ for all $m$. We make use of the following equivalence of categories (see for example \cite[Theorem 4.7]{BrzezinskiWisbauer03}).

\begin{proposition}\label{prop:classical equiv of cat}
    There is an equivalence of categories between left $E(n)_{*}$-comodules and right $E(n)$-modules.
\end{proposition}

This equivalence of categories is helpful because it is often easier to describe modules rather than comodules. In particular, it is helpful to use the following decomposition of the homology of Brown--Gitler spectra in terms of homological lightning flash modules.

\begin{remark}
    \rm{Since $E(n)$ is a Hopf algebra, it comes with a conjugation map that yields a further equivalence of categories, from right $E(n)$-modules to left $E(n)$-modules. The left $E(n)$-module action is the one used in other references (see for example \cite{Culverp219}). However, in the $C_2$-equivariant setting, no formula for a conjugation map on $\cE(n)$ is known, so we use the right $E(n)$-module structure here for consistency with later sections. See \cite[Section~6]{Boardman82}, as well as \cite{bhattacharya2024steenrod} and \cite[Thm~4.7]{BrzezinskiWisbauer03} for further discussion of these structures and their relationships.}
\end{remark}

\begin{definition}\label{def:classicalLightning}
    The homological lightning flash module is given by 
    $$L(k) = E(1) \{x_1, x_2, \cdots, x_k \, |\,  x_{i + 1} Q_1 =  x_i Q_0,\,  1 \leq i \leq k-1 \}$$
    where $| x_i| = 2i + 1.$ Further define $L(0) = \mF_2.$
\end{definition}

\begin{proposition}\label{nonequiv: homology}
    As an $E(1)_{*}$-comodule, \[H_{*}B_{0}(k) \cong L(\nu_{2}(k!)) \oplus F_{k}\] where $F_{k}$ is a sum of suspensions of copies of $E(1)_{*}$.
\end{proposition}
Using \cref{nonequiv: homology} and the identification of $H_* ku$ in \cref{nonequiv: homology splitting},  we can rewrite the $E_{2}$-page of the Adams spectral sequence (\ref{nonequiv:ass}) as 
\begin{align} \label{nonequiv:E2}
\begin{split}
    \Ext_{E(1)_{*}}\left(H_{*}\Sigma^{2k}B_{0}(k), H_{*}ku\right) & \cong \bigoplus\limits_{m=0}^{\infty}\Ext_{E(1)_{*}}\left(\Sigma^{2k}L(\nu_{2}(k!)), \Sigma^{2m}L(\nu_{2}(m!)) \right)  \oplus W \\
\end{split}
\end{align}
where $W$ is a sum of suspensions of $\mF_{2}$ in filtration $f=0$. 

Note that $\Ext_{E(0)_{*}}(\mF_{2}, \mF_{2}) \cong \mF_{2}[v_{0},v_{1}]$. The $\Ext_{E(1)_*} (L(k), L(m))$ terms have a nice description as $\mF_{2}[v_{0},v_{1}]$-modules.

\begin{proposition}\label{nonequiv: Ext computation}
    When $k \le m$, 
    \[\Ext_{E(1)_*}^{s,f}(L(k), L(m)) \cong \mF_{2}[v_{0}, v_{1}]\{x_{0}, x_{1}, \ldots x_{m-k}|\ v_{1}x_{i} = v_{0}x_{i+1} \} \oplus \Sigma^{-1-2j}\mF_2\{z_1,\dots,z_k\},\] 
    where $|x_{i}| = (2i, 0)$ and $|z_i| = (-1-2i,0)$.
    
    When $k > m$, \[\Ext_{E(1)_{*}}^{s,f}(L(k), L(m)) \cong \mF_{2}[v_{0}, v_{1}]\left\{x, y_{0}, y_{1}, \ldots, y_{k-m}  \Bigg\vert \begin{array}{l}
         v_{1}y_{i} = v_{0}y_{i+1},\\  v_{0}y_{0} = 0,\\
         v_{1}y_{k-m} = 0 
    \end{array}\right\} \oplus \mF_2\{z_{k-m+1},\dots,z_k\},\]
    where $|x| = (0, k-m)$, $|y_{i}| = (-3 - 2(k-m-i), 0)$, and $|z_i| = (-1-2i,0)$.
\end{proposition} 

\begin{figure}
\centering
\begin{minipage}{.5\textwidth}
\centering
\includegraphics[width=2in, height=2in]{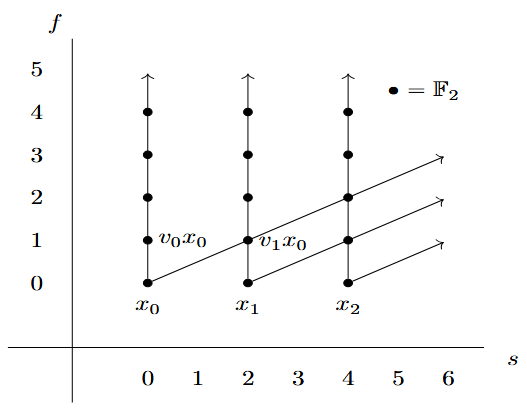}
 \captionof{figure}{\\$\Ext_{E(1)_*} (\mF_2, L(3))$}
  \label{fig:clExtFpL3}
\end{minipage}%
\begin{minipage}{.5\textwidth}
\begin{center}
\includegraphics[width=2.75in,height=2in]{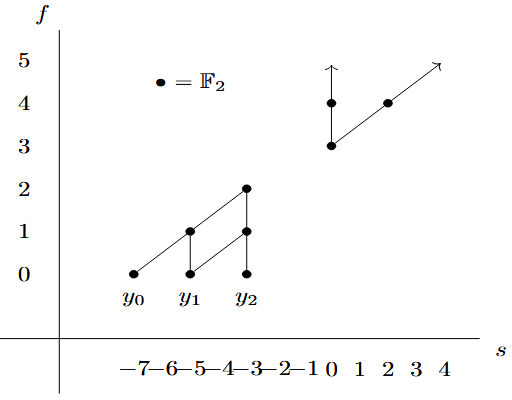}
 \end{center}
  \captionof{figure}{\\ $\Ext_{E(1)_*} (L(3), \mF_2)$}
  \label{fig:clExtL3Fp}
\end{minipage}
\end{figure}

These $\Ext$-terms can be visualized as in the charts given in \cref{fig:clExtFpL3} (where $k \leq m$) and \cref{fig:clExtL3Fp} (where $k > m$).

Furthermore, one can lift the $E(1)_{*}$-comodule isomorphism of \cref{nonequiv: homology} to a splitting of $ku$-module spectra, using a $ku$-relative Adams spectral sequence. Towards this end, it will be helpful to first record a simpler splitting of $ku \wedge B_0(k)$ and $ku \wedge ku$. 

\begin{proposition}\label{nonequiv: other splitting}
There are $ku$-module splittings
    \[ ku \wedge B_{0}(k) \simeq C_{k} \vee V_{k}\]
     \[ ku \wedge ku \simeq C \vee V,\]
     where $V_{k}$ and $ V$ are sums of suspensions of $H\mF_{2}$, and $C_k$ and $C$ are $v_{1}$-torsion-free $ku$-modules.
\end{proposition}

For details of the proof, see the $C_2$-equivariant analogue, \cref{prop: other splitting}.

\begin{proposition}\label{nonequiv: collapse}
    For all $k$, the Adams spectral sequence of (\ref{nonequiv:ass}), 
    \[
    \Ext_{E(1)_{*}}^{s,f} \left(H_{*} \Sigma^{2k} B_{0}(k), H_{*}ku \right) \Longrightarrow [ku \wedge \Sigma^{2k}B_{0}(k), ku \wedge ku]^{ku},
    \]
    collapses.
\end{proposition}
\begin{proof}
\cref{nonequiv: other splitting} implies that the splitting of the $E_{2}$-page (\ref{nonequiv:E2}) is in fact a splitting of spectral sequences. Recall from \cref{nonequiv:rel steenrod} that $H\mF_{2_{*}}^{ku}H\mF_{2} \cong E(1)_{*}$. Since $E(1)_{*}$ is both injective and projective, all summands of the Adams spectral sequence (\ref{nonequiv:ass}) which contain $V_k$ or $V$ in either term are concentrated in filtration $0$ and thus have no differentials. So all potential differentials must occur within the summand
\[
\bigoplus\limits_{m=0}^{\infty} \Sigma^{2(m-k)} \Ext_{E(1)_{*}}\left(L(\nu_{2}(k!)), L(\nu_{2}(m!)) \right).
\]

We will show that no such differentials occur. First, recall that no nontrivial differentials can go from $v_{1}$-torsion classes to $v_{1}$-torsion free classes. Next, note that all $v_{1}$-torsion classes are concentrated in odd stems, and the Adams differential $d_{r}$ has degree $(s,f) = (-1,r)$, so there are no differentials between them. Likewise, all of the $v_{1}$-torsion free classes are concentrated in even stem $s$, so there are no differentials between $v_{1}$-torsion free classes.  Finally, we must consider the possibility of differentials from the $v_{1}$-torsion free classes to $v_{1}$-torsion classes. This is impossible for the following degree reasons. 

 We will now show there are no differentials with source a $v_1$-torsion free class and target a $v_{1}$-torsion class. The $v_1$-torsion classes in filtration $f > 0$ are precisely the classes $\Sigma^{2(m-k)}v_0^iy_{j}$, which only occur in stem $s \le -3$. Since 
\[
\bigoplus\limits_{0 \le k \le m}\Sigma^{2(m-k)}\Ext_{E(1)_*}( L(\nu_{2}(k!) ), L(\nu_{2}(m!) ) )
\]
is contained entirely in stem $s \ge 0$, it only remains to show that there are no differentials originating in $v_1$-torsion free classes of
\[
\bigoplus\limits_{k > m}\Sigma^{2(m-k)}\Ext_{E(1)_*}( L(\nu_{2}(k!) ), L(\nu_{2}(m!) ) )
.\]

%Note that the $v_{1}$-torsion free classes are entirely concentrated in even stem. Fix $n \ge 0$. We will put bounds on the Adams filtration of the classes in stem $-2n$, and bounds on the Adams filtration of classes in stem $-2n-1$. We will use these bounds to conclude that no nontrivial differentials are possible. 

%The only $v_{1}$-torsion free generator of $\Sigma^{2(m-k)} \Ext_{E(1)_{*}}\left(L(\nu_{2}(k!), L(\nu_{2}(m!)\right)$ is $\Sigma^{2(m-k)}x$, in Adams filtration $\nu_2(k!)-\nu_2(m!).$ It follows that in \[
%\bigoplus\limits_{k > m}\Sigma^{2(m-k)}\Ext_{E(1)}( L(\nu_{2}(k!) ), L(\nu_{2}(m!) ) )
%,\]
%the only generator in stem $-2n$ is $\Sigma^{-2n}x$, which has Adams filtration $\nu_{2}(k!) - \nu_{2}((k-n)!)$.

These are all generated by $\Sigma^{2(m-k)}x$. Any target of a differential exiting $\Sigma^{-2n}x$ will be in stem $-2n-1$. One can check that the highest filtration class in stem $-2n-1$ has Adams filtration $\nu_2(k!) - \nu_2((k-n+1)!) - 1$. Thus no $v_1$-torsion classes can be the target of differentials with source a $v_1$-torsion free class, and indeed the Adams spectral sequence collapses. 

\begin{comment}
Now observe that in $ \Ext_{E(1)_{*}}\left(L(\nu_{2}(k!), L(\nu_{2}(m!)\right)$, the highest filtration of any $v_{1}$-torsion class in stem $-3-2i$ is $\nu_{2}(k!)-\nu_{2}(m!)-i-1$. So in $\Sigma^{2(k-m)}\Ext_{\cE(1)_{\star}}\left(L(\nu_{2}(k!), L(\nu_{2}(m!)\right)$, the highest filtration of any $v_{1}$-torsion class in stem $-3-2i - 2(k-m)$ is $\nu_{2}(k!)-\nu_{2}(m!)-i-1$. Therefore, for any $0 \le i \le m < k$ such that 
\[-3 -2i -2(k-m) = -2n -1,\]
the highest filtration class will be at $\nu_{2}(k!) - \nu_{2}\left((k-n+1+i)!\right) - i-1$. So at any stem $-2n-1$, the highest filtration class will be in filtration $\nu_{2}(k!) - \nu_{2}((k-n+1)!) - 1$. 
Thus no $v_1$-torsion classes can be the target of differentials with source a $v_1$-torsion free class, and indeed the Adams spectral sequence collapses.
\end{comment}
\end{proof}

So indeed the isomorphism of \cref{nonequiv: homology splitting} lifts, and the splitting is an immediate consequence.
\begin{theorem}
    Up to $2$-completion, there is a splitting of $ku$-modules
    \[ku \wedge ku \simeq \bigvee\limits_{k=0}^{\infty}ku\wedge \Sigma^{2 k}B_{0}(k).\]
\end{theorem}

This splitting can also be interpreted in terms of Adams covers. Let $ku^{\langle n \rangle}$ denote the $n$-th Adams cover of $ku$, that is, the $n$-th term in a minimal Adams resolution of $ku$ over $H\mF_{2}$ (See \cite[p.2-3]{lellmann1984operations} and Proposition 6.18 of \cite{kane1981operations} for discussion of Adams covers and the proof of the next proposition. Note that their phrasing is slightly different as they discuss cohomology and do not reference relative homology). 

\begin{proposition} \label{nonequiv:kuRelHomology}
The $ku$-relative homology of the $n$-th Adams cover of $ku$ is
    \[
    H_{*}^{ku}ku^{\langle n \rangle} \cong L(\nu_{2}(n)).
    \]
\end{proposition}

The Adams cover $ku^{\langle n \rangle}$ is uniquely determined up to homotopy by its homology, and the fact that it is a $ku$-module. 

\begin{theorem} Up to $2$-completion, 
    \[ ku \wedge B_{0}(k) \simeq ku^{\langle \nu_{2}(k!) \rangle} \vee V_{k},\]
    where $V_{k}$ is a sum of suspensions of $H\mF_{2}$.
\end{theorem}

\begin{proof}
Consider the $ku$-relative Adams spectral sequence
\begin{equation}
     E_{2}^{s,f} \cong \Ext_{E(1)_{*}}^{s,f} \left(H_{*}^{ku}C_{k}, H_{*}^{ku}ku^{\langle \nu_{2}(k!) \rangle} \right) \Longrightarrow [ku \wedge \Sigma^{2k}B_{0}(k), ku \wedge ku]^{ku},
\end{equation}
where the $C_k$ are as defined in \cref{nonequiv: other splitting}. 
Note that 
\[
\Ext_{E(1)_*}^{s,f} \left( H_*^{ku} C_k, H_*^{ku} ku^{\langle \nu_2 (k!) \rangle} \right) \cong \Ext_{E(1)_*}^{s,f} \left( L(\nu_2(k!) ), L(\nu_2(k!) ) \right).
\]
It follows from the arguments of \cref{nonequiv: collapse} that this spectral sequence collapses at the $E_{2}$-page, and thus we can lift the isomorphism $H_{*}^{ku}C_{k} \to H_{*}^{ku}ku^{\langle \nu_{2}(k!) \rangle}$ of \cref{nonequiv:kuRelHomology} to the level of spectra.
\end{proof}

\begin{cor} Up to $2$-completion, 
    \[ ku \wedge ku \simeq \bigvee_{k=0}^{\infty} \Sigma^{2k} ku^{\langle \nu_{2}(k!) \rangle} \vee V,  \]
    where $V$ is a sum of suspensions of $H\mF_{2}$.
\end{cor}

\section{Equivariant preliminaries}\label{sec: equivariant preliminaries}

\subsection{The $C_2$-equivariant homology of a point} \label{sec:C2point} A detailed description of the coefficients $H_\star$ will be useful for our calculations. Our description closely follows that of \cite[Section~2.1]{GuillouHillIsaksenRavenel2020}, which in turn is a reinterpretation of results in \cite[Proposition~6.2]{HuKriz2001}. Throughout, we frequently denote  the coefficients $H_\star$ as $\mM_2.$

Additively, $\mM_2$ is
\begin{enumerate}
    \item $\mF_2$ in degree $(s,w)$ if $s \leq 0$ and $w \leq s.$
    \item $\mF_2$ in degree $(s,w)$ if $s \geq 0$ and $w \geq s + 2.$
    \item $0$ otherwise.
\end{enumerate}

This additive structure is represented by the dots in \cref{fig:M2}. The non-zero element in degree $(0,-1)$ is called $\tau,$ and the non-zero element in degree $(-1,-1)$ is called $\rho.$ Sometimes in the equivariant literature, the element $\tau$ is called $u$ or $u_\sigma,$ and $\rho$ is called $a$ or $a_\sigma.$ We choose to use the names $\tau$ and $\rho$ common in motivic literature so that we can easily write $\mM_2$ as a square-zero extension of $\mM_2^\mR,$ the motivic homology of $\mR$ with $\mF_2$-coefficients. 

\begin{remark}
    \rm{We also use $\rho$ to denote the $C_2$-regular representation, writing $\Sigma^\rho$ for suspension by the regular representation in the $RO(C_2)$ grading. Whether $\rho$ is an element in the homology of the point or the $C_2$-regular representation will be clear from context.  }
\end{remark} 

\begin{figure}[ht]
    \centering
    \includegraphics[width=3in,height=3in]{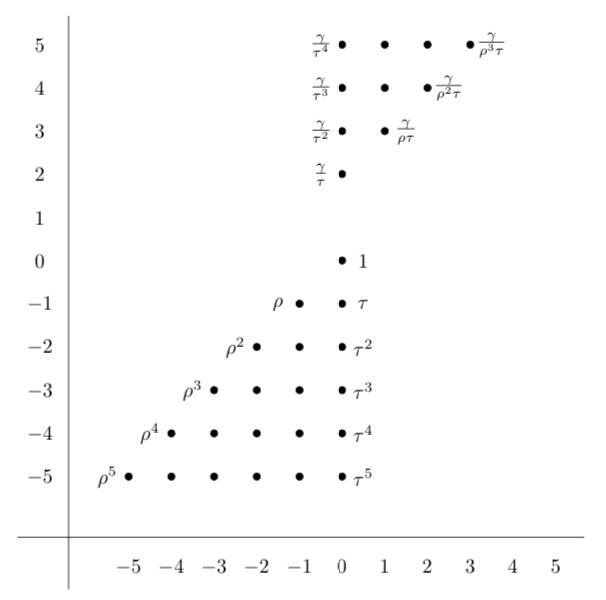}
    \caption{$\mM_2$}
    \label{fig:M2}
\end{figure}

The ``positive cone'' refers to the part of $\mM_2$ in degrees $(s,w)$ with $w \leq 0$. (This may seem backwards, but the reason for this is that in the cohomology of a point, which is isomorphic to the homology of a point and is also often denoted $\mM_{2}$ in the literature, the ``positive cone'' portion of $\mM_{2}$ is in positive degrees). The positive cone is isomorphic to the $\mR$-motivic homology ring $\mM^\mR_2$ of a point. Multiplicatively, the positive cone is a polynomial ring on two variables, $\rho$ and $\tau.$

The ``negative cone'' $NC$ refers to the part of $\mM_2$ in degrees $(s,w)$ with $w \ge 2.$ Multiplicatively, the product of any two elements of $NC$ is zero, so $\mM_2$ is a square-zero extension of $\mM^\mR_2.$ Multiplications by $\rho$ and $\tau$ are non-zero in $NC$ whenever they make sense. Thus elements of $NC$ are infinitely divisible by both $\rho$ and $\tau.$ We use the notation $\frac{\gamma}{\rho^j \tau^k}$ for the non-zero element in degree $(j, 1+j+k),$ which is consistent with the described multiplicative properties. The symbol $\gamma,$ which does not correspond to an actual element of $\mM_2$,  has degree $(0,1).$

 The $\mF_2[\tau]$-module structure on $\mM_2$ is essential for calculations filtered by powers of $\rho.$ Thus we describe the $\mF_2[\tau]$-module structure on $NC$ in further detail. We define $\mF_2 [\tau]/\tau^\infty$ to be the $\mF_2[\tau]$-module $\colim_n \mF_2[\tau] / \tau^k$. Following \cite[Remark~2.1]{GuillouHillIsaksenRavenel2020}, we write $\frac{\mF_2 [\tau^{k}]}{\tau^\infty} \{x\}$ for the infinitely divisible $\mF_2[\tau]$-module consisting of elements of the form $\frac{x}{\tau^k}$ for $k \geq 1.$ Note that $x$ itself is not an element of $\frac{\mF_2 [\tau]}{\tau^\infty} \{x\}.$ The idea is that $x$ represents the infinitely many relations $\tau^k \cdot \frac{x}{\tau^k} = 0$ that define $\frac{\mF_2 [\tau]}{\tau^\infty} \{x\}.$

 With this notation in place, $\mM_2$ is equal to 
 \[
 \mM^{\mR}_2 \oplus NC = \mM^\mR_2 \oplus \underset{s \geq 0}{\bigoplus} \frac{\mF_2[\tau]}{\tau^\infty} \left\{ \frac{\gamma}{\rho^s} \right\}
 \]
 as an $\mF_2[\tau]$-module.

 \subsection{Comparison with nonequivariant homology} 

 Suppose $X$ is a $C_2$-spectrum. The cofiber of $\rho$,
 \[
 S^{-\sigma} \hookrightarrow S^0 \to C\rho,
 \] 
is stably given by $C\rho \simeq \Sigma^{1-\sigma} {C_2}_+$. Smashing $H \wedge X $ with the transfer map
\[
S^{1-\sigma} \xrightarrow[]{\mathrm{tr}} \Sigma^{1-\sigma}{C_2}_+\simeq C\rho
\]
gives a map
 \[
 H \wedge X \xrightarrow{\mathrm{tr} \wedge \mathrm{id}_X} \Sigma^{1-\sigma}H \wedge X \to H \wedge X \wedge C\rho.
 \]

 Applying $\pi_V^{C_2}$ to this map gives a homomorphism 
 \[
 \Phi^e: H_V(X) \to H_{\vert V \vert} (X^e),\]
 which can be used to compare the $C_2$-equivariant homology with the underlying nonequivariant homology of the nonequivariant spectrum underlying $X.$

 We can use the same technique to compare the ${ku_\mR}_{\star}$-module structure of ${ku_\mR}_{\star}X$ to the $ku_{*}$-module structure of $ku_{*}X^{e}$. Specifically, smash $ku_{\mR} \wedge X$ with the transfer map, and apply $\pi_{V}^{C_2}$ to construct a homomorphism
 \[ \Phi^{e}: ku_{{\mR}_{\star}}X \to ku_{*}X^{e} .\]
 This will allow us to deduce $ku_{{\mR}_{\star}}$-module extensions in \cref{sec:height1split}.

 \subsection{Equivariant connective covers}
Throughout this paper, we make use of equivariant connective covers so we give their definition here. Suppose $X$ is a $C_2$-spectrum. The equivariant connective cover $X \langle 0 \rangle \xrightarrow{q} X$ is a $C_2$-spectrum such that:
\begin{enumerate}
    \item the restriction of $q$ is the connective cover of the underlying spectrum $X^e$, and
    \item the categorical fixed points of $q$ is the connective cover of the categorical fixed points of $X.$ 
\end{enumerate}
See \cite[p. 25]{GuillouHillIsaksenRavenel2020} for a more detailed description of the restriction and categorical fixed points functors (cf. \cite[Proposition~3.3]{Lewis95}).

\subsection{The $C_2$-equivariant Dual Steenrod Algebra} In \cite{HuKriz2001}, Hu--Kriz computed the $C_2$-equivariant dual Steenrod algebra
$$\cA_\star \cong H_\star [{\xi}_1, {\xi}_2, \cdots, {\tau}_{0}, {\tau}_{1}\, \cdots ] \slash {({\tau}_i}^2 = a {\tau}_{i + 1} + (u + a {\tau}_0) \xi_{i + 1} ),$$
where $\vert \xi_{i} \vert = (2^i - 1)\rho,$ $\vert \tau_i \vert = 2^i \rho - \sigma,$ and 
$$\psi(\xi_i) = \sum_{0 \leq j \leq i} \xi_{i - j}^{2^j} \otimes \xi_j \qquad \qquad \psi(\tau_i) = \tau_i \otimes 1 + \sum_{0 \leq j \leq i} \xi_{i - j}^{2^j} \otimes \tau_j.  $$

We denote the images of the generators $\xi_i$ and $\tau_i$ under the conjugation map 
$$c: \cA_\star^{C_2} \to \cA_\star^{C_2}$$ 
by $\bar{\xi}_i$ and $\bar{\tau}_i$ respectively. 

The coproduct formulas on the conjugates are 
$$\psi(\bar{\xi}_i) = \sum_{0 \leq j \leq i} \bar{\xi}_j \otimes \bar{\xi}_{i - j}^{2^j} \qquad \qquad \psi (\bar{\tau}_i ) = 1 \otimes \bar{\tau}_i  + \sum_{0 \leq j \leq i} \bar{\tau_j} \otimes \bar{\xi}_{i - j}^{2^j}.$$

\subsection{Relative homology and Adams spectral sequences} \label{sec:Relative}
We will make use of $RO(C_2)$-graded relative homology and relative Adams spectral sequences in our computations. Analogously to the nonequivariant setting, if $E$ is an $R$-algebra and $M$ is an $R$-module in spectra, then $E$-homology in the category of $R$-modules is 
$$E_\star^R (M) := \pi_\star (E \underset{R}{\wedge} M).$$
Note that
$$E^R_\star(R \wedge M) = \pi_\star(E \underset{R}{\wedge} R \wedge M) = E_{\star}M.$$

In \cite[Prop~2.1]{baker2001adams}, Baker--Lazarev introduce a relative Adams spectral sequence in the category of $R$-modules. The $RO(C_2)$-graded version exists by the same construction.

\begin{proposition}\label{prop:relative ASS exists}
    Let $E$ be a $C_{2}$-equivariant $R$-algebra spectrum, and let $X,Y$ be $R$-modules. If $E_{\star}^{R}X$ is projective as an $E_{\star}$-module, then there exists an $E$-based Adams spectral sequence in the category of $R$-modules
\[
E_{2}^{s,f,w} \cong \Ext_{E_{\star}^{R} E}^{s,f,w} (E_\star^R X, E_\star^R Y) \Longrightarrow [X,Y]_{\widehat{E}}^{R \, \, s, w}
\]
 where $[X,Y]^{R}_{\widehat{E}}$ denotes the $E$-nilpotent completion of $R$-module maps from $X$ to $Y$, $s$ denotes the stem, $f$ denotes the homological degree, and $w$ denotes the weight.
\end{proposition}

\begin{remark} 
    \rm{We use several types of Adams spectral sequences in this paper, so let us take a moment to clarify what they are and how we will refer to them. In \cref{sec: height zero}, we will take $E = H\umF_{2}$, and $R = H\umZ$, and we will refer to the $H\umF_{2}$-based Adams spectral sequence in the category of $H\umZ$-modules simply as ``the relative Adams spectral sequence''. In \cref{sec:height1split}, we will take $E = H\umF_{2}$, and $R = ku_{\mR}$, and we will refer to the $H\umF_{2}$-based Adams spectral sequence in the category of $ku_{\mR}$-modules simply as ``the relative Adams spectral sequence''. Finally, the primary motivation for constructing the splitting of $ku_{\mR} \wedge ku_\mR$ is to facilitate computations with the $ku_{\mR}$-based Adams spectral sequence for the sphere: this means taking $E = ku_{\mR}$, and $R = \mathbb{S}$. We will refer to this spectral sequence as ``the $ku_{\mR}$-based Adams spectral sequence". This spectral sequence only appears in \cref{sec:introduction} and in \cref{sec:E1kuR}. }
\end{remark}

\section{{$C_2$}-equivariant homology results} \label{sec:C2homology}

We will make use of the following relative homology computation. 

\begin{proposition}\label{prop:relHomology} The $BP_\mR\0$-relative dual Steenrod algebra is
\begin{align*}
    H_{\star}^{BP_\mR\0}H & \cong \cE(0)_{\star} 
\end{align*}
and the $BP_\mR\one$-relative dual Steenrod algebra is 
    \[
    H_{\star}^{BP_\mR\one}H \cong \cE(1)_{\star}.
    \]
\end{proposition}

To prove this, we need to fix some notation. Recall there is a classifying map $L \to {MU_\mR}_\star$ that is an isomorphism in degrees $\rho k,$ where $L \cong \mZ[a_1, a_2, \cdots]$ is the Lazard ring, and $\vert a_i \vert = \rho i$ \cite[Theorem~2.28]{HuKriz2001}. Working $2$-locally, $MU_\mR$ splits as a sum of shifts of the spectrum
\[
BP_\mR := MU_\mR / (a_i | i \neq 2^k -1).
\]
As is standard, let $v_i := a_{2^i - 1}$ and
\[
BP_\mR \langle m \rangle = BP_\mR / (v_{m + 1}, v_{m + 2}, \cdots ).
\]

\begin{proof}[Proof of \cref{prop:relHomology}]
We give an argument for the second statement and leave the first to the reader. Consider the cofiber sequence 
\[
 BP_\mR \one \xrightarrow{{v}_1} BP_\mR \one \xrightarrow{} H \umZ.
 \]
Smashing $H  \underset{BP_\mR \one}{\wedge} \text{---} \,$ with the cofiber sequence yields a splitting
 \[
 H \underset{BP_\mathbb{R}\one}{\wedge} H \umZ \simeq H  \wedge (S^0 \vee S^{\rho + 1}),
 \]
 which suffices to determine the additive structure of $H_{\star}^{BP_\mR \one}H\umZ$. 
 
 The multiplicative structure can be deduced from the algebra map 
 \[
 \pi_{\star}\big(H \wedge H \big) \to \pi_{\star}\big(H \underset{BP_\mR \one}{\wedge} H \big).
 \]
\end{proof}

\subsection{The homology of $\mathbf{BP_\mR \langle n \rangle}$}  The authors first learned of a computation of $H_\star BP_\mathbb{R} \langle n \rangle,$ 
from Christian Carrick in the context of Carrick, Hill, and Ravenel's work on the homological slice spectral sequence in motivic and Real bordism \cite{CarrickHillRavenel24}. While Carrick--Hill--Ravenel did indeed compute the equivariant homology groups $H_\star BP_\mathbb{R} \langle n \rangle,$ the result does not appear in the final version of their paper, so we include an argument that we learned from  Carrick here.

Recall that $\pi_{*\rho}BP_\mR \cong \mathbb{Z}[{v}_1, {v}_2, \cdots]$ \cite[Theorem~6.18]{HuKriz2001}. 
Consider the cofiber sequence
\[
H \wedge \Sigma^{(2^i-1)\rho} BP_\mR \xrightarrow{\cdot{v}_i} H \wedge BP_\mR \xrightarrow{} H \wedge BP_\mR / {v}_i.
\]
Note that ${v}_i = 0 \in H_\star BP_\mR$, so applying homotopy yields a short exact sequence
\[
0 \to H_\star BP_\mR \to H_\star (BP_\mR / {v}_i) \to H_{\star - ((2^i-1)\rho + 1)} BP_\mR \to 0
.\]
This sequence splits as a sum of $H_\star BP_\mR$-modules, so
\[
H_\star (BP_\mR / {v}_i) \cong H_\star BP_\mR \oplus H_{\star - ((2^i - 1)\rho + 1)} BP_\mR.
\]
In particular, $H_\star (BP_\mR / {v}_i)$ is a free $H_\star BP_\mR$-module. 

\begin{proposition}
   Let $g: BP_\mR \to H$ be the real orientation. The map 
    \[
    g_{\star}: H_\star (BP_\mR / {v}_i) \to H_\star H
    \]
    is injective, with image the free $H_\star BP_\mR$-submodule generated by $1$ and $\bar{\tau}_{i}$. 
\end{proposition}

\begin{proof}
We can write the $H_{\star}BP_{\mR}$-module splitting above as \[H_{\star}BP_{\mR}/{v}_{i}\cong H_{\star}BP_{\mR}\{x_{0}\} \oplus  H_{\star}BP_{\mR}\{x_{1}\}, \] where $|x_{0}| = 0$ and $|x_{1}| = (2^{i}-1)\rho + 1$. Note that everything in the first summand has degree a multiple of $\rho$, while $\vert \tau_i \vert = (2^i)\rho - \sigma$ $=(2^i-1)\rho + 1$. So in order to show that the map $g_{\star}: H_\star (BP_\mR / {v}_i) \to \cA_{\star} $ is injective, it will suffice to show that $g_{\star}(x_{1})$ is nonzero.

Note that since $g_{\star}$ is a map of $\cA_{\star}$-comodules, $ker\ g_{\star}$ is an $A_{\star}$-subcomodule of $H_{\star}BP_{\mR}/{v}_{i}$. Suppose towards a contradiction that $g_{\star}(x_{1}) = 0$. This implies that $H_{\star}BP_{\mR}\{x_{1}\}$ is a sub-$A_{\star}$-comodule of $H_{\star}BP_{\mR}/{v}_{i}$, and so the $H_{\star}BP_{\mR}$-splitting 
\[
H_{\star}BP_{\mR}/{v}_{i}\cong H_{\star}BP_{\mR}\{x_{0}\} \oplus H_{\star}BP_{\mR}\{x_{1}\} 
\]
would furthermore be an $\cA_{\star}$-comodule splitting. If we then run the Adams spectral sequence
\[
\Ext_{\cA_{\star}}\left(\mM_{2_{\star}}, H_{\star}BP_{\mR}/v_{i} \right) \Longrightarrow \pi_{\star}BP_{\mR}/{v}_{i},
\] 
we get 
\[
\pi_{\star}BP_{\mR}/{v}_{i} \cong \pi_{\star}BP_{\mR} \oplus \Sigma^{(2^i-1)\rho + 1}\pi_{\star}BP_{\mR}.
\]
In particular, this tells us that ${v}_{i}$ is nonzero in $\pi_{\star}BP_{\mR}/{v}_{i}$. This is a contradiction, so indeed $g_{\star}x_{1} \neq 0$.

Now we will show that $g_{\star}(x_{1}) \equiv \tau_{i}$ mod $H_{\star}BP_{\mR}$. The generator $x_{1}$ is primitive, so $g_{\star}(x_{1})$ must be a nonzero $A_{\star}$-comodule primitive. Note that $\bar{\tau}_{i}$ is a primitive in ${\cA}^{C_2}_\star / H_\star BP_\mR$ and has the right degree. So to show that $g_{\star}(x_{1}) \equiv \bar{\tau}_{i}$, it will suffice to show that ${\Prim}({\cA}^{C_2}_\star / H_\star BP_\mR)$ contains only one element in degree $(2^{i}-1)\rho + 1$. We will accomplish this by first showing that 
\[ 
{\Prim}^{2^{i+1}\rho + 1}({\cA}^{C_2}_\star / H_\star BP_\mR) \cong \Ext_{\cE_{\star}}^{2^{i+1}-2,1,2^{i}-1}(\mathbb{M}_{2},\mathbb{M}_{2}) ,
\]
and then analyzing this $\Ext$ group.
Recall that 
\[
{\Prim}({\cA}_\star / H_\star BP_\mR) \cong \Hom_{\cA_{\star}}\left(\mM_{2}, {\cA}_\star / H_\star BP_\mR\right) .\]
To analyze this $\Hom$ group, consider the short exact sequence \[ 0 \to H_{\star}BP_{\mR} \to \cA_{\star} \to \cA_{\star}/H_{\star}BP_{\mR} \to 0 \]
and induced long exact sequence 
\[ 0 \to \Ext_{\cA_{\star}}^{s,0,w}(\mathbb{M}_{2},H_{*}BP_{\mR}) \to  \Ext_{\cA_{\star}}^{s,0,w}(\mathbb{M}_{2}, \cA_{\star}) \to   \Ext_{\cA_{\star}}^{s,0,w}(\mathbb{M}_{2}, \cA_{\star}/H_{\star}BP_{\mR})\]
\[ \to \Ext_{\cA_{\star}}^{s-1,1,w}(\mathbb{M}_{2},H_{\star}BP_{\mR}) \to \Ext_{\cA_{\star}}^{s-1,1,w}(\mathbb{M}_{2}, \cA_{\star}) \to \cdots.\]

First note that 
$\Ext_{\cA_{\star}}^{*,1,*}(\mathbb{M}_{2}, \cA_{\star}) = 0$ because $\cA_{\star}$ is free. Let $\cE = \cE(Q_{0}, Q_{1}, \ldots)$ and note that $H_{\star}BP_{\mR} \cong A//\cE_{\star}$. By change-of-rings,
\[ \Ext^{s,f,w}_{{\cA}^{C_2}_\star} (\mM_{2}, H_\star BP_\mR ) \cong \Ext^{s,f,w}_{\cE_{\star}} (\mM_{2}, \mM_{2}).\] 
 Observe that $\Ext^{2^{i+1}-1,0,2^{i}-1}_{\cE_\star}(\mathbb{M}_{2}, \mathbb{M}_{2}) = 0$, so indeed
 \[ \Ext^{2^{i+1}-1,0,2^{i}-1}_{\cA_{\star}}(\mathbb{M}_{2}, \cA_{\star}/H_{\star}BP_{\mR}) \cong \Ext_{\cE_{\star}}^{2^{i+1}-2,1,2^{i}-1}(\mathbb{M}_{2},\mathbb{M}_{2}) .\]
 
 Now we will show that $\Ext_{\cE_{\star}}^{2^{i+1}-2,1,2^{i}-1}(\mathbb{M}_{2},\mathbb{M}_{2})$ contains at most one element.

Recall from \cite[Prop~3.1,Rmk~3.3]{GuillouHillIsaksenRavenel2020} that there exists a $\rho$-Bockstein spectral sequence
\[ E_{1} \cong \Ext_{gr_{\rho}\cE}(gr_{\rho}\mM_{2}, gr_{\rho}\mM_{2})[\rho] \Longrightarrow \Ext_{\cE}(\mM_{2},\mM_{2}) .\]
Furthermore, the spectral sequence decomposes into a sum of two spectral sequences: a ``positive cone summand''
\[ E_{1}^{+} \cong \Ext_{\cE^{\mC}}(\mM_{2}^{\mC},\mM_{2}^{\mC})[\rho]\Longrightarrow  \Ext_{\cE^{\mR}}(\mM_{2}^{\mR},\mM_{2}^{\mR}),\]
and a ``negative cone summand''
\[ E_{1}^{-} \cong \Ext_{gr_{\rho}\cE^{\mR}}(gr_{\rho}NC,gr_{\rho}\mM_{2}^{\mR})[\rho]\Longrightarrow  \Ext_{\cE^{\mR}}(NC,\mM_{2}^{\mR}),\]
where $NC$ is the negative cone summand in $\mM_{2}$. 
In \cite{Hill11}, Hill computes 
\[
\Ext_{\cE^{\mC}}(\mM_{2}^{\mC},\mM_{2}^{\mC}) \cong \mF_{2}[\rho, \tau, v_{0}, v_{1}, \ldots],
\]
where $|\rho| = (-1,0,-1)$, $|\tau| = (0,0,-1)$, and $|v_{i}| = (2^{i+1} - 2, 1, 2^{i}-1)$ for all $i$ \cite[Thm~3.1]{Hill11}. So the only element in $E_{1}^{+}$ in degree $(2^{i+1}-2,1,2^{i}-1)$ is $v_{i}$. Combining the fact that $\Ext_{\cE^{\mC}}(\mM_{2}^{\mC},\mM_{2}^{\mC})$ is free over $\mM_2$ with \cite[Prop~3.1,Rmk~3.3]{GuillouHillIsaksenRavenel2020} tells us that
\[ 
E_{1} \cong \Ext_{gr_{\rho}\cE^{\mR}}(gr_{\rho}NC,gr_{\rho}\mM_{2}^{\mR}) \cong \bigoplus\limits_{s=0}^{\infty}\left\{\frac{\gamma}{\rho^{s}}\right\} \underset{\mM_{2}^{\mathbb{C}}}{\otimes}\Ext_{\cE^{\mC}}(\mM_{2}^{\mC},\mM_{2}^{\mC})[\rho]  .\]
Since $|\frac{\gamma}{\rho^{m}\tau^{n}}| = (m+n,0,n)$, there are no elements in degree $(2^{i+1} - 2, 1, 2^{i}-1)$ in $E_{1}^{-}$. Therefore there is at most one element in $\Ext_{\cE}^{2^{i+1} - 2, 1, 2^{i}-1}(\mM_{2},\mM_{2})$ and thus $\bar{\tau}_{i}$ is the only primitive element of degree $2^{i-1}\rho + 1$ in $A_{\star}/H_{\star}BP_{\mR}$. So indeed $g_{\star}(x_{1}) \equiv \bar{\tau}_{i}$ mod $H_{\star}BP_{\mR}$ and therefore the image of $g_{\star}$ is generated as an $H_{\star}BP_{\mR}$-submodule by $1$ and $\bar{\tau}_{i}$.

\end{proof}

By freeness, for any $i, \, j,$ the K\"{u}nneth spectral sequence 
    \[
    \Tor_{*, \star}^{H_\star BP_\mR} (H_\star (BP_\mR / \bar{v}_i), H_\star (BP_\mR / \bar{v}_j)) \Longrightarrow H_\star (BP_\mR / (\bar{v}_i, \bar{v}_j))
    \]
    collapses to give an isomorphism
    \[
    H_\star (BP_\mR / (\bar{v}_i, \bar{v}_j)) \cong H_\star (BP_\mR / \bar{v}_i) \otimes_{H_\star BP_\mR} H_\star (BP_\mR / \bar{v}_j). 
    \]
We can iterate this process to compute  $H_\star BP_\mR / (\bar{v}_{n + 1}, \cdots, \bar{v}_{n + m})$. Taking the colimit of $H_\star BP_\mR / (\bar{v}_{n + 1}, \cdots, \bar{v}_{n + m})$ over $m$ yields the following proposition.
\begin{proposition} \label{prop:BPRnHomology}
    The map $H_\star BP_\mR \langle n \rangle \to H_\star H$ is injective with image the sub-$H_\star$-algebra
    \[
    H_\star BP_\mR \langle n \rangle \cong H_\star [ \bar{\xi}_1, \bar{\xi}_2, \bar{\xi}_3, \cdots , \bar{\tau}_{n + 1}, \bar{\tau}_{n + 2}, \bar{\tau}_{n + 3}, \cdots] / (\bar{\tau}_i^2 = a \bar{\tau}_{i + 1} + u \bar{\xi}_{i + 1})
    \]
    In particular,
    \begin{equation}\label{eqn: homlogy of truncated}
    H_\star BP_\mR \langle n \rangle \cong \cA_\star \underset{\mathcal{E} (n)_{\star}}{\square} H_\star.
    \end{equation}
\end{proposition}

\begin{remark}
    \rm{The description of $H_{\star}BP_{\mR}\n$ given in (\ref{eqn: homlogy of truncated}) is analogous to the nonequivariant isomorphism $H_{*}BP\n \cong A//E(n)_{*}$. Likewise, the motivic homology of the $2$-complete algebraic Johnson-Wilson spectra $BPGL \langle n \rangle$ over $p$-adic fields \cite[Theorem~3.9]{Ormsby2011} has the same form, with the appropriate motivic analogues substituted for $A$ and $E(n)$.}
\end{remark}
The left $\mathcal{E}(i)_\star$-coaction on $\cA_\star \underset{\mathcal{E} (n)_\star}{\square} H_\star$ is given by
\begin{align} \alpha: \cA_\star \underset{\mathcal{E} (n)_\star}{\square} H_\star \overset{\psi \otimes 1}\longrightarrow \cA_\star \otimes \left(\cA_\star \underset{\mathcal{E} (n)_\star}{\square} H_\star\right) \overset{\pi \otimes 1}\longrightarrow \mathcal{E} (n)_\star \otimes \left(\cA_\star \underset{\mathcal{E} (n)_\star}{\square} H_\star \right), \label{eq:Ecoaction}
\end{align}
which on generators $\bar{\xi}_k$ and $\bar{\tau}_k$ is 
\begin{align*} 
    \alpha (\bar{\xi}_k) & = 1 \otimes \bar{\xi}_k \\
    \alpha(\bar{\tau}_{n + k}) & = 1 \otimes \bar{\tau}_{n + k} + \sum_{0 \leq i \leq n} \bar{\tau}_{i} \otimes \bar{\xi}_{n + k -i}^{2^i}.
\end{align*}

Let $M$ be a left $\cE(i)_{\star}$-comodule. Then there is an induced right $\cE(i)$-module action $\lambda: M \otimes \cE(i)  \to M$, defined by 
\begin{equation} \label{eq:inducedRightAction}
    \lambda(x,\theta) =  (\theta\otimes \mathrm{Id}_{M}) \circ \alpha(x)
\end{equation}
where $\alpha(x) = \Sigma_{i}\theta_{i} \otimes x_{i}$ (see \cite[\textsection~6]{Boardman82}).

Since $\bar{\tau}_{i}$ is the dual of $Q_{i}$, the right $\mathcal{E} (1)$-module action on $H_\star BP_\mR \langle 1\rangle$ is
\begin{align*}
    \bar{\xi}_{j}Q_i  & = 0 \\
     \bar{\tau}_{k}Q_0  & = \bar{\xi}_{k} \\
     \bar{\tau }_{k}Q_1 &  = \bar{\xi}_{k-1}^2.
\end{align*}

Further, since $\cE(i)_{\star}$ is finitely generated and projective over $\mM_{2}$ \cite{Ricka15}, we obtain the following proposition through standard algebra (see for example \cite[Theorem 4.7]{BrzezinskiWisbauer03}). 
\begin{proposition}\label{prop:equiv of cat}
    There is an equivalence of categories between left $\cE(n)_{\star}$-comodules and right $\cE(n)$-modules.
\end{proposition}

\subsection{A Splitting in Homology}\label{sec:homology splitting}
Define a weight filtration on $\cA_\star$ by 
$$wt(\bar{\xi}_k) = 2^k = wt(\bar{\tau}_k) \qquad \text{ and } \qquad wt(xy) = wt(x) + wt(y).$$ 
For all $i \ge -1$, the $j^{th}$ Brown--Gitler comodule $N_i ( j)$ is the subspace of $\cA//\cE(i)_{\star}$ spanned by monomials of weight less than or equal to $j$. (Here, we abuse notation by defining $\cA//\cE(-1)_{\star} := \cA_{\star}$)  

Define $M_i (j)$ to be the $\mM_{2}$-submodule of $\cA_\star \square_{\mathcal{E} (i)_{\star}} H_\star$ spanned by monomials of weight exactly $j.$ Observe from the coaction on the generators $\bar{\tau}_{j}$ and $\bar{\xi}_{j}$ (\ref{eq:Ecoaction}) that the $\cE(i)_{\star}$-coaction on $\cA//\cE(n)_{\star}$ preserves Mahowald weight, so $M_{i}(j)$ is an $\cE(i)_{\star}$-subcomodule of $\cA//\cE(n)_{\star}$. The next proposition follows immediately, and is analogous to the underlying nonequivariant statement \cite[Proposition~3.3]{Culverp219}. 

\begin{proposition}\label{prop:wt-decomp}
There is a natural isomorphism
$$\cA_\star \underset{\mathcal{E} (i)_\star}{\square} H_\star \cong \bigoplus_{k \geq 0} M_i(k)$$
of left $\mathcal{E}(i)_\star$-comodules. 
\end{proposition}

Just as in the nonequivariant case \cite[Lemma~3.4]{Culverp219},\cite[Lemma~4.10]{CulverOddp20}, we also have the following $\cE(i)_{\star}$-comodule isomorphism. 
\begin{proposition}
    For $n \ge 0$, there is an isomorphism of $\cE(n)_{\star}$-comodules
    \[\Sigma^{\rho k} N_{n-1}(k) \xrightarrow{\cong} M_{n}(2k)\]
for all $k$, where 
$$ \bar{\xi}_1^{k_1} \bar{\xi}_2^{k_2} \cdots \bar{\xi}_{n+1}^{k_{n+1}} \bar{\tau}_{n + 1}^{\epsilon_{n +1}} \bar{\xi}_{n + 2}^{k_{n + 2}} \bar{\tau}_{n + 2}^{\epsilon_{n + 2}} \cdots \quad  \mapsto \quad \bar{\xi}_1^{a} \bar{\xi}_2^{k_1} \cdots \bar{\xi}_{n + 2}^{k_{n+1}} \bar{\tau}_{n + 2}^{\epsilon_{n +1}} \bar{\xi}_{n + 3}^{k_{n + 2}} \bar{\tau}_{n + 3}^{\epsilon_{n + 2}}\cdots $$
and
$a = \frac{1}{2}\left[k - wt\left(\bar{\xi}_2^{k_1} \cdots \bar{\xi}_{n + 2}^{k_{n+1}} \bar{\tau}_{n + 2}^{\epsilon_{n +1}} \bar{\xi}_{n + 3}^{k_{n + 2}} \bar{\tau}_{n + 3}^{\epsilon_{n + 2}}\cdots\right)\right].$
\end{proposition}

\begin{theorem}\label{thm:homology:decomp}
For $n \geq 0$, there exists a family of maps
$$\{\theta_k: H_\star\Sigma^{\rho k} B_{n - 1}(k) \to H_\star BP_\mR \langle n \rangle \, | \, k \in \mN \} $$
such that their sum
$$  \theta : \bigoplus_{k = 0}^\infty  H_{\star - \rho k} B_{n - 1}(k) \to H_\star BP_\mR \langle n \rangle $$
is an isomorphism of $\mathcal{E}(n)_\star$-comodules. 
\end{theorem}
The following analogue of \cite[Rmk~4.12]{CulverOddp20} follows formally from the fact that the natural projection $\cA_{\star} \to \cE(n)_{\star}$ is a map of Hopf algebroids, and $H_{\star}BP_{\mR}\langle n \rangle \cong A//\cE(n)_{\star}$.

\begin{cor}
    There are $\cA_{\star}$-comodule isomorphisms
    \[H_{\star}( BP_{\mR}\n \wedge BP_{\mR}\n ) \cong \bigoplus \limits_{k=0}^{\infty}H_{\star}BP_{\mR}\n \otimes \Sigma^{\rho k}N_{n - 1}(k) .\]
\end{cor}
This naturally leads one to ask for which heights $n$ these isomorphisms can be realized. The first hurdle is to construct spectra realizing the Brown--Gitler comodules $N_{i}(k)$.

\subsection{{$C_2$-equivariant} (integral) Brown--Gitler spectra}
In \cite{BehrensWilson18}, Behrens--Wilson give an equivalence
\[
(\Omega^\rho S^{\rho + 1})^\mu \simeq H 
\]
of $C_2$-spectra. In \cite{HW2020}, Hahn--Wilson observe that the left hand side of this equivalence carries a natural filtration, which produces a filtration of $H$ by spectra. This filtration is analogous to the May-Milgram filtration of $H\mF_{2}$, which Mahowald observed and Cohen proved could be used to construct Brown--Gitler spectra \cite{cohen1985immersion, Mahowald1977}. In \cite{HW2020}, Hahn--Wilson point out that one can simply define equivariant Brown--Gitler spectra using this filtration. In \cite[Proposition 3.4]{LiPetersenTatum23}, we confirm that these spectra, which we will denote $\cB_{-1}(k)$, indeed have the property that
\[ H_{\star}\cB_{-1}(k) \cong \mM_{2}\{x \in \cA_{\star}\, |\, wt(x) \le k \}.\]

In \cite{CDGM1988}, Cohen-Davis-Goerss-Mahowald give a construction of nonequivariant integral Brown--Gitler spectra. 
In \cite[Thm~6.1]{LiPetersenTatum23}, we construct $C_{2}$-equivariant lifts of these spectra, that is, finite spectra $\{\cB_{0}(k)|\ k \in \mN\}$ such that 
\[
H_{\star}\cB_{0}(k) \cong \mM_{2}\{ x \in H_{\star}H\mZ|\ wt(x) \le k \}.
\]

\begin{remark} \label{remark:intBG}
    \rm{There is a small difference between our notation for the integral Brown--Gitler spectra here and in \cite{LiPetersenTatum23}. Here we use  $\cB_{0}(k)$ to denote the Brown--Gitler spectrum realizing the subcomodule $H_{\star}H\mZ$ having weight at most $\mathbf{k}$. In \cite{LiPetersenTatum23}, we used $\cB_{0}(k)$ to denote the Brown--Gitler spectrum realizing the subcomodule $H_{\star}H\mZ$ having weight at most $\mathbf{2k}$. Since every element of $H_{\star}H\mZ$ has weight divisible by $2$, these are indeed the same family of spectra. The convention in this paper makes the results of \cref{sec:homology splitting} easier to state, and generally makes the formulas in \cref{sec:height1split} cleaner and more readable.}
\end{remark}

\subsection{Margolis homology and free {$\cE(n)$}-modules}\label{sec: Margolis homology}
The goal of this subsection is to give a Whitehead Theorem for Margolis homology in the $C_{2}$-equivariant setting. To this end, we first establish some $C_2$-equivariant freeness criteria. These criteria are $C_2$-equivariant analogues of the $\mR$-motivic freeness criteria studied in \cite[Section~2]{BGL2022}, and we follow the techniques developed there closely. The motivating idea is that by requiring the homology be free over $\mM_2,$ we can reduce to simpler $\mR$- and $\mC$-motivic calculations. Our first aim will be to prove the following proposition. 

\begin{proposition} \label{prop:mFree}
A finitely generated $\cE(n)$-module $M$ is free if and only if
\begin{enumerate}
    \item $M$ is free as an $\mM_2$-module and
    \item $\mF_2 {\otimes}_{\mM_2} M$ is free as an $E(n)$-module.
\end{enumerate}
\end{proposition}

To prove \cref{prop:mFree}, we will make use of the following lemma. 

\begin{lemma}[\cite{BGL2022} Lemma 2.1] \label{lem:mfree}
A finitely generated $\cE(n)$-module $M$ is free if and only if 
\begin{enumerate}
    \item $M$ is a free $\mM_2$-module and
    \item $M / (\rho)$ is a free $\cE(n) / (\rho)$-module.
\end{enumerate}
\end{lemma}

Since condition (1) makes $M$ a free $\mM_2$-module, the proof of \cref{lem:mfree} is the same as the $\mR$-motivic case described in \cite[Lemma~2.1]{BGL2022}. In particular, no additional complications arise from the negative cone. 

We will also make use of the following $\mC$-motivic lemma which is an immediate consequence of \cite[Theorem~B~(i)]{HK2018}.

\begin{lemma}\cite[Theorem~B]{HK2018}\label{lem:mfreeoverC}
A finitely generated $\cE^{\mC}(n)$-module is free if and only if
\begin{enumerate}
    \item $M$ is a free $\mM^{\mC}_{2}$-module and
    \item $M/(\tau)$ is a free $\cE^{\mC}(n)/(\tau)$-module.
\end{enumerate}
\end{lemma}

\begin{proof}[Proof of \cref{prop:mFree}]
If a finitely generated $\cE(n)$-module $M$ is free, the two conditions are immediately satisfied. Thus we consider $M$ where the two conditions are satisfied and show $M$ is a free $\cE(n)$-module. First, note that condition (2) of \cref{prop:mFree} is equivalent to saying that $(M/(\rho))/(\tau)$ is free over $\cE^{\mC}(n)/(\tau)$. \cref{lem:mfreeoverC} then implies that $M/(\rho)$ is a free $\cE^{\mC}(n)$-module, which is equivalent to saying that $M/(\rho)$ is a free $\cE(n)/(\rho)$-module. So by \cref{lem:mfree}, $M$ is indeed a free $\cE(n)$-module.
\end{proof}

Recall the following freeness criteria for modules over $E(n)$.
\begin{proposition}\cite[Thm~18.8]{margolis2011spectra}\label{classical freeness}
Let $M$ be a bounded-below free $E(n)$-module. Then $M$ is free if and only if $\cM_{*}(M, Q_{i}) = 0$ for all $0 \le i \le n$. 
\end{proposition}
Note that a map of $\cE(n)$-modules $f:M \to N$ is said to be a stable equivalence if there exist free $\cE(n)$-modules $P$, $Q$ such that the map $M \oplus P \to N \oplus Q$ induced by $f$ is an isomorphism \cite[p.205]{margolis2011spectra}. 

Combining \cref{classical freeness} with \cref{prop:mFree} yields a $C_{2}$-equivariant analogue of the Whitehead Theorem for Margolis homology. The proof proceeds in exactly the same way as its classical analogue, (see \cite[Thm~18.8~ii]{margolis2011spectra}, or the $\mathbb{C}$-motivic analogue in \cite[Cor~4.10]{GheorgheIsaksenRicka18} for a version with more modern writing). 
\begin{proposition}[Whitehead Theorem]\label{prop:EquivWhitehead} 
Let $M$ and $N$ be finitely generated $\cE(1)$-modules that are $\mM_2$-free, and let $f: M \to N$ be an $\cE(1)$-module map. Then $f$ is a stable equivalence if and only if $f/(\rho,\tau):M/(\rho,\tau) \to N/(\rho,\tau)$ induces an isomorphism in Margolis homologies with respect to $Q_0$ and $Q_1.$
\end{proposition}

\subsection{The homology of mod-$2$ Brown--Gitler spectra}
Recall that $H_{\star}B_{-1}(0)\cong \mM_{2}$. We now show that $H_{\star}B_{-1}(k)$ is a free and injective $\cE_{\star}(0)$-comodule for all $k>0$. To do so, we need a few propositions.

\begin{proposition}\label{equiv:free:inj}
If $M$ is a free $\cE(n)_{\star}$-comodule, then $M$ is an injective $\cE(n)_{\star}$-comodule. 
\end{proposition}
\begin{proof}
    First, recall from \cref{prop:equiv of cat} that we can instead work with $\cE(n)$-modules. Since every free $\cE(n)$-module is a sum of copies of $\cE(n)$, it will suffice to show that $\cE(n)$ is self-injective. This follows from modifying May's proof that $\mM_{2}$ is self-injective \cite[Appendix]{May20}. In particular, the graded ideals of $\cE(n)$ are just those of $\mM_{2}$, with the addition of various $Q_{i}$. 
\end{proof}

\begin{proposition}\label{equiv:BG:free}
For all $k>0$, $H_{\star}B_{-1}(k)$ is a free and injective $\cE(0)_{\star}$-comodule.
\end{proposition}

\begin{proof} Recall from \cref{prop:equiv of cat} that we can regard left $E(n)_*$-comodules as right $E(n)$-modules via the induced right action of \cref{eq:inducedRightAction}. We use \cref{prop:mFree} to show that $H_\star B_{-1}(k) $ is a free $\cE(0)$-module for all $k>0$. First, note that $H_\star B_{-1}(k)$ is a free $\mM_{2}$-module. Then observe that  \[  \mF_2 \otimes_{\mM_2} H_{\star}B_{-1}(k) \cong \mF_2 \{x =  \bar{\xi}_{1}^{i_{1}} \bar{\xi}_{2}^{i_{2}}\cdots \bar{\xi}_{r}^{i_{r}} \bar{\tau}_{0}^{\epsilon_{0}} \bar{\tau}_{1}^{\epsilon_{1}} \cdots \bar{\tau}_{s}^{\epsilon_{s}}\mid wt(x) \le n\}, \]
which is exactly the homology of the classical Brown--Gitler spectrum of weight $k$. The homology of the classical Brown--Gitler spectrum of weight $k$ is a free $E(0)$-module. Thus $H_\star B_{-1}(k) $ is a free $\cE(0)$-module. So by \cref{equiv:free:inj}, we know that $H_\star B_{-1}(k) $ is a free and injective $\cE(0)_{\star}$-comodule.  \end{proof}

\subsection{The homology of integral Brown--Gitler spectra}
\begin{definition}\label{def: lightning}
    The homological lightning flash module is given by 
    $$L(k) = \mathcal{E}(1) \{x_1, x_2, \cdots, x_k \, |\,  x_{i + 1} Q_1 =  x_i Q_0,\, 1 \leq i \leq k-1 \}$$
    where $| x_i| = i \rho + 1.$ Further define $L(0) = \mM_2.$
\end{definition}

These lightning flash modules can be easily visualized as in \cref{fig:HomLightFlash}, which depicts $L(4)$ using the motivic grading convention that the total topological degree (number of sign plus trivial real $C_2$-representations) is plotted along the horizontal axis and the weight (number of sign representations) is plotted along the vertical axis.  Here, a point denotes a copy of $\mM_2$, straight arrows indicate the
(non-trivial) operation of $Q_0,$ and curved arrows indicate the (non-trivial) operation of $Q_1.$  

\begin{figure}[ht]\label{fig:lightning flash}
    \centering
    \resizebox{2in}{1in}{
    \begin{tikzpicture}
    \draw (3,1) -- (2,1); % horizontal line
    \node at (2.5,1) {\textup{$<$}}; % arrow
    \draw (5,2) -- (4,2); % horizontal line
    \node at (4.5,2) {\textup{$<$}}; % arrow
    \draw (7,3) -- (6,3); % horizontal line
    \node at (6.5,3) {\textup{$<$}}; % arrow
    \draw (9,4) -- (8,4); % horizontal line
    \node at (8.5,4) {\textup{$<$}}; % arrow

    \draw (0,0)  arc (180:90:2); % first lightning flash
    \draw (2,2) arc (90:0:1); % first lightning flash 
    \node at (2,2) {\textup{$<$}}; % arrow
    \node at (3.25,0.85) {\textup{$x_1$}}; % label 
    
    \draw (2,1) arc (180:270:1); % second lightning flash
    \draw (3 , 0 ) arc (270:360:2); % second lightning flash
    \node at (3,0) {\textup{$<$}}; % arrow
    \node at (5.33,2) {\textup{$x_2$}}; % label

    \draw (4,2)  arc (180:90:2); % third lightning flash
    \draw (6,4) arc (90:0:1); % third lightning flash 
    \node at (6,4) {\textup{$<$}}; % arrow
    \node at (7.24,2.85) {\textup{$x_3$}}; % label
    
    \draw (6,3) arc (180:270:1); % fourth lightning flash
    \draw (7 , 2 ) arc (270:360:2); % fourth lightning flash 
    \node at (7,2) {\textup{$<$}}; % arrow
    \node at (9.25, 4.1) {\textup{$x_4$}}; %label

    \node at (0,0) {\textup{$\bullet$}}; % point
    \node at (2,1) {\textup{$\bullet$}}; % point
    \node at (3,1) {\textup{$\bullet$}}; % point
    \node at (5,2) {\textup{$\bullet$}}; % point
    \node at (4,2) {\textup{$\bullet$}}; % point
    \node at (7,3) {\textup{$\bullet$}}; % point
    \node at (6,3) {\textup{$\bullet$}}; % point
    \node at (9,4) {\textup{$\bullet$}}; % point
    \node at (8,4) {\textup{$\bullet$}}; % point

    \end{tikzpicture}}
    \caption{Homological lightning flash module: $L(4)$}
    \label{fig:HomLightFlash}
\end{figure}

\begin{proposition}\label{bg lightning split}
    For $k \ge 0$, there is an isomorphism 
    \[H_{\star}B_{0}\left( k \right) \cong L(\nu_{2}(k!)) \oplus W_{k}\]
    of $\cE(1)_\star$-comodules where $W_{k}$ is a sum of suspensions of $\cE(1)_{\star}$.
\end{proposition}
\begin{proof}
We will use \cref{prop:EquivWhitehead} to prove this isomorphism, so we start by computing the relevant Margolis homologies. 
First, observe from looking at the figure above that
\[\cM_{*}\big(L(\nu_{2}(k!))/(\rho,\tau), Q_{0}\big) \cong \mF_{2}\{1\}\] \[\cM_{*}\big(L(\nu_{2}(k!))/(\rho,\tau), Q_{1}\big) \cong \mF_{2}\{x_{\nu_{2}(k!)}Q_{0} \}.\]  Next we must compute the Margolis homology of $H_{\star}\cB_{0}(k)$. It is easier to first compute the Margolis homology for $H_{\star}ku_{\mR}$,
\begin{align*}
    \cM_{*}\left(H_{\star}ku_{\mR}/(\rho,\tau), \, Q_{0}\right) & \cong \mF_{2}\{1\} \\
    \cM_{*}\left(H_{\star}ku_{\mR}/(\rho,\tau),  \,Q_{1} \right) & \cong \mF_{2}\{\bar{\xi}_{1}^{\epsilon_{1}}\bar{\xi}_{2}^{\epsilon_{2}}\ldots \bar{\xi}_{n}^{\epsilon_{n}}|\ 0 \le \epsilon_{i} \le 1 .\}
\end{align*} (This is exactly the computation of the Margolis homology of $H_{*}ku$ in \cite[Lemma~16.9]{Adams74}, although the notation differs slightly from ours). 
 By \cref{thm:homology:decomp},
 \[H_{\star}ku_{\mR} \cong \bigoplus\limits_{k=0}^{\infty}\Sigma^{\rho k}H_{\star}\cB_{0}(k),\]
 Since $H_{\star}ku_{\mR} $ is of finite-type, we can apply this decomposition to the corresponding Margolis homology and get
\begin{align*}
    \cM_{*}\left(H_{\star}\cB_{0}\left( k \right)/(\rho,\tau), \, Q_{0}\right) & \cong \mF_{2}\{1\} \\
    \cM_{*}\left(H_{\star}\cB_{0}\left( k \right)/(\rho,\tau),  \,Q_{1} \right) & \cong \mF_{2}\{\bar{\xi}_{i_{1}}\bar{\xi}_{i_{2}}\ldots \bar{\xi}_{i_{n}}\},
\end{align*}
where $i_{1} < i_{2} < \cdots < i_{n}$ and $2^{i_{1}} + 2^{i_{2}}+\cdots + 2^{i_{n}}$ is the $2$-adic expansion of $k$. 

Note that $|x_{\nu_{2}(k!)}| = |\bar{\xi}_{i_{1}}\bar{\xi}_{i_{2}}\cdots \bar{\xi}_{i_{n}}|$. So if we construct an $\cE(1)$-module map 
\[
L(\nu_{p}(k!)) \rightarrow H_{\star}\cB_{0}(k)
\]
realizing the isomorphism, then we can apply \cref{prop:EquivWhitehead} to conclude the proof. 
Let $S(k)$ denote the submodule of $L(k)$
\[S(k) = \cE(1)\{ \bar{\xi}_{i_{1}}\bar{\xi}_{i_{2}}\ldots \bar{\xi}_{i_{n}}\bar{\tau}_{j}|j \le i_{1} < i_{2} < \cdots < i_{n} \}.\]
Observe that $S(k)$ is naturally isomorphic to $L(\nu_{2}(k!))$, and so the inclusion $S(k) \to H_{\star}B_{0}(k)$ is exactly the map we are looking for.
\end{proof}

\section{Splitting {$H \umZ \wedge H \umZ$}}\label{sec: height zero}

In this section, we construct a family of $H \umZ$-module maps
\[ \Tilde{\theta}_k: \Sigma^{\rho k} H \umZ \wedge \cB_{-1}(k) \rightarrow H \umZ \wedge H\umZ \] 

such that
    \[ \bigvee\limits_{k=0}^{\infty}\Tilde{\theta}_k: \Sigma^{\rho k} H \umZ \wedge \cB_{-1}(k) \xrightarrow{\simeq} H \umZ \wedge H \umZ \]
is an equivalence (up to $2$-completion).

\subsection{Strategy} \label{sec:heightzerostrategy} \cref{thm:homology:decomp} gives a family of maps
$$\{\theta_k: H_\star \Sigma^{\rho k} \cB_{ - 1} (k) \to H_\star H \umZ \, | \, k \in \mN \} $$
such that their sum
$$ \bigoplus_{k = 0}^\infty \theta_k : H_\star \Sigma^{\rho k} \cB_{ - 1} (k) \to H_\star H \umZ $$
is an isomorphism of $\mathcal{E}(0)_\star$-comodules. 

Using $H \umZ$-relative homology (discussed in \cref{sec:Relative}), we can think of the family of $\mathcal{E}(0)_{\star}$-comodule maps
$$ \bigoplus_{k = 0}^\infty \theta_k : H_{\star} \Sigma^{\rho k} B_{-1} (k) \to H_{\star} H \umZ $$
as a family of maps
$$ \bigoplus_{k = 0}^\infty \theta_k : H_{\star}^{H\umZ} (H \umZ_{\star} \wedge \Sigma^{\rho k}  B_{-1} (k)) \to H_{\star}^{H \umZ} (H \umZ \wedge H \umZ)$$
and consider the $H\umZ$-relative Adams spectral sequence 
\begin{align} \label{eq:heightZeroAss}
\begin{split}
E_2^{s,f,w}  = \Ext_{\cE(0)_{\star}}^{s,f,w} (H_{\star} \Sigma^{\rho k} B_{-1}(k), H_{\star} H \umZ ) 
\implies [H \umZ \wedge \Sigma^{\rho k}B_{-1} (k), H\umZ \wedge H \umZ]^{H \umZ}. 
\end{split}
\end{align}
We grade $\Ext$-groups in the form $(s,f,w)$, where $s$ is the stem, that is, the total degree minus the homological degree, $f$ is the Adams filtration, that is the homological degree, and $w$ is the weight. Consider $\theta_k$ as a class in filtration $f = 0$ so $|\theta_k| = (0,0,0)$. The Adams differential $d_{r}$ decreases stem by $1$, increases filtration by $r$, and preserves motivic weight.

Since $E_{2}^{s,f,w}$ is finite in each degree, the spectral sequence converges \cite[Thm~15.6, Thm~7.1]{boardman1999conditionally}. Thus constructing the desired maps $\Tilde{\theta}_k : H \umZ \wedge \Sigma^{\rho k} B_{-1} (k) \to H \umZ \wedge H \umZ$ for all $k \in \mN$ is the same as showing the $\theta_k$ survive the spectral sequence. In fact, we will show that the $H \umZ$-relative Adams spectral sequence of (\ref{eq:heightZeroAss}) collapses at the $E_{2}$-page.

\subsection{Analyzing the $E_{2}$-page}
\subsubsection{Starting the analysis}

\begin{proposition}\label{E2:zero:split}
The $E_2$-page of the Adams spectral sequence (\ref{eq:heightZeroAss}) has the form
\[ E_2^{s, f, w} = \Ext_{\cE(0)_{\star}}^{s,f, w} (H_{\star} \Sigma^{\rho k} B_{-1}(k), H_{\star} H \umZ ) \cong \Ext_{E(0)_{\star}}^{s,f,w} (\mM_2, \mM_2) \oplus V, \]
where $V$ is an $\mM_{2}$-vector space concentrated in Adams filtration $f=0$. 
\end{proposition}
\begin{proof}
Using the isomorphism 
\[
H_\star H \umZ \cong \bigoplus_{m = 0}^\infty H_\star \Sigma^{\rho m}  B_{- 1}(m) 
\]
of $\cE(0)_\star$-comodules given by \cref{thm:homology:decomp} yields
\begin{align*}
E_2^{s, f, w} & \cong \Ext_{\cE(0)_{\star}}^{s, f, w} (H_{\star} \Sigma^{\rho k} B_{-1}(k), H_{\star} H \umZ )  \\
&  \cong \bigoplus\limits_{m=0}^{\infty} \Ext_{\cE(0)_{\star}}^{s,f,w} \big( H_{\star} \Sigma^{\rho k} B_{-1}(k), H_{\star} \Sigma^{\rho m} B_{-1}(m) \big).
\end{align*}

Since $H_{\star}B_{-1}(k)$ is a free and injective $\cE(0)_{\star}$-comodule when $k > 0$ by \cref{equiv:BG:free}, any summand $\Ext_{\cE(0)_{\star}}^{s,f,w} \big(H_{\star} \Sigma^{\rho k} B_{-1}(k), H_{\star} \Sigma^{\rho m} B_{-1}(m) \big)$ with $k$ or $m$ nonzero must be concentrated on the ($f=0$)-line. Then since $H_{\star}B_{-1}(0) \cong \mM_{2}$ by definition,
\[ E_2^{s, f,w} \cong \Ext_{E(0)_{\star}}^{s,f,w} (\mM_2, \mM_2) \oplus V, \]
where $V$ is an $\mM_{2}$-free module concentrated in filtration $f=0$. 
\end{proof}

We now compute $\Ext_{\cE(0)_{\star}}^{s,f,w} (\mM_2, \mM_2).$ Our computation closely follows that of \newline $\Ext_{\cE(1)_{\star}}^{s,f,w} (\mM_2, \mM_2) $ in \cite{GuillouHillIsaksenRavenel2020}, which uses simpler $\mC$- and $\mR$-motivic calculations as stepping stones. In that vein, we view $C_2$-equivariant coefficients as tensored up from $\mR$-motivic coefficients. In particular, 
\[ \cE(n)_\star \cong \mM_{2} \otimes_{{\mM}_{2}^{\mR} }\cE_{\star}^{\mR}(n)\]
(see \cite[Equation~2.4]{GuillouHillIsaksenRavenel2020} for more details).

Let $NC$ denote the negative cone, so that $\mM_2 \cong \mM_{2}^{\mR} \oplus NC$ as an $\mF_{2}[\tau]$-module. Then the square-zero extension $\mM_{2} \cong \mM_{2}^{\mR} \oplus NC$ induces a decomposition \cite[page~8]{GuillouHillIsaksenRavenel2020} 
\[ \Ext_{\cE(n)_\star}(\mM_{2}, \mM_{2}) \cong \Ext_{\cE_{\star}^{\mR}(n)}(\mM_{2}^{\mR}, \mM_{2}^{\mR}) \oplus \Ext_{\cE_{\star}^{\mR}(n)}(NC, \mM_{2}^{\mR})  \]
and one can use the $\rho$-Bockstein spectral sequence to analyze each of these summands. 

\begin{proposition}[\cite{GuillouHillIsaksenRavenel2020} Proposition~3.1] \label{prop:rhoBockstein}
  There is a $\rho$-Bockstein spectral sequence
  \[ E_{1} = \Ext_{\gr_{\rho} \cE(n)_\star}(\gr_{\rho} \mM_2, \gr_{\rho} \mM_2) \Longrightarrow \Ext_{\cE(n)_\star}(\mM_{2}, \mM_{2}) ,\]
  such that the differential $d_{r}$ sends a class in degree $(s,f,w)$ to a class of degree $(s-1, f+1, w)$. Furthermore, the spectral sequence decomposes into the following two pieces:
  \[ E_{1}^{+} = \Ext_{\mC}[\rho] \Longrightarrow \Ext_{\cE_{\star}^{\mR}(n)}(\mM_{2}^{\mR}, \mM_{2}^{\mR}) \]
  and 
  \[ E_{1}^{-} \cong \bigoplus\limits_{s=0}^{\infty} \frac{\mM_{2}^{\mC}}{\tau^{\infty}} \Bigg{\{ } \frac{\gamma}{\rho^{s}} \Bigg{ \} } \underset{\mM_{2}^{\mC}}{\otimes} \Ext_{\cE_\star^{\mC}(n)} (\mM_{2}^{\mC}, \mM_{2}^{\mC})   \Longrightarrow \Ext_{\cE_{\star}^{\mR}(n)}(NC, \mM_{2}^{\mR}). \]
\end{proposition}

\subsubsection{Analyzing $E_{1}^{+}$}\label{sec: positive cone HZ} In \cite{Hill11}, Hill gives a complete calculation of $\Ext_{\cE^{\mR}(n)}(\mM_{2}^\mR, \mM_{2}^\mR).$ We are only working over $\cE(0)$ right now, so we state that portion of the result here. 
\begin{proposition}[\cite{Hill11} Thm~3.2]
    There exists a $\rho$-Bockstein spectral sequence 
\[ E_{1} = \Ext_{\cE^{\mC}(0)}(\mM_{2}^{\mC}, \mM_{2}^{\mC})[\rho] = \mF_{2}[\tau, v_{0}, \rho] \Longrightarrow Ext_{\cE^{\mR}(0)}(\mM_{2}^{\mR}, \mM_{2}^{\mR}),\]
with differential $d_{1}(\tau) = \rho v_{0}$.
\end{proposition}

The values of $\Ext_{\cE^{\mR}(0)}(\mM_{2}^{\mR}, \mM_{2}^{\mR})$ follow immediately.

\begin{proposition}[\cite{Hill11} Thm~3.1]\label{ext:0:pos} There is an isomorphism
    \[ \Ext_{\cE^{\mR}(0)}(\mM_{2}^{\mR}, \mM_{2}^{\mR}) \cong \mF_{2}[\rho, \tau^{2}, v_{0}]/(\rho v_{0}) ,\]

    where $|v_{0}| = (0,1,0)$, $|\tau^{2}| = (0,0,-2)$ and $|\rho| = (-1,0,-1)$.
\end{proposition}

\subsubsection{Analyzing $E_{1}^{-}$}\label{sec: negative cone HZ}
It remains to calculate the $E_{1}^{-}$ summand in the $\rho$-Bockstein spectral sequence. By \cref{prop:rhoBockstein} (\cite[Proposition~3.1]{GuillouHillIsaksenRavenel2020}), 
\[
E_{1}^{-} \cong \bigoplus\limits_{s=0}^{\infty} \frac{\mM_{2}^{\mC}}{\tau^{\infty}} \Bigg{\{ } \frac{\gamma}{\rho^{s}} \Bigg{ \} } \underset{\mM_{2}^{\mC}}{\otimes} Ext_{E^{\mC}(0)} (\mM_{2}^{\mC}, \mM_{2}^{\mC})   \Longrightarrow Ext_{\cE^{\mR}(0)}(NC, \mM_{2}^{\mR}).
\]
To determine the differentials, we use the strategy described in \cite[\textsection~7.2]{GuillouHillIsaksenRavenel2020}. The argument relies heavily on the $\rho$-Bockstein differential for $E^+$ and uses that $E^{-}$ is an $E^{+}$-module. In particular, $E_{1}^{-}$ is generated over $E_{1}^{+}$ by the elements $\frac{\gamma}{\rho^{j}\tau^{k}}$.

The differentials in $E^{-}$ are infinitely divisible by $\rho$, meaning that if $d_{r}(x) = y$, then $d_{r}(\frac{x}{\rho^{j}}) = \frac{y}{\rho^{j}} $ for all $j \ge 0$. So all differentials in the $E^{-}$-summand of the $\rho$-Bockstein spectral sequence are determined by the differential 
\[
d_{1}\left(\frac{\gamma}{\rho\tau^{2k+1}}\right) = \frac{\gamma}{\tau^{2k+2}}v_{0},
\]
in combination with the Leibnitz rule, $E^{+}$-module structure, and infinite-$\rho$-divisibility \newline \cite[\textsection 7.2, Proposition 7.7]{GuillouHillIsaksenRavenel2020}.

Thus using the $\rho$-Bockstein spectral sequence, we have the following computation. 

 We will use the notation
\[ \Ext_{\cE^{+}_{\star}(0)}(\mM_{2}, \mM_{2}) := \Ext_{\cE^{\mR}_{\star}(0)}(\mM_{2}^{\mR}, \mM_{2}^{\mR})\]
\[ \Ext_{\cE^{-}_{\star}(0)}(\mM_{2}, \mM_{2}) := \Ext_{\cE^{\mR}_{\star}(0)}(NC, \mM_{2}^{\mR}).\]

\begin{proposition}\label{ext:0:neg} 
The summand $\Ext_{\cE^{-}_{\star}(0)}(\mM_2, \mM_{2})$ consists of two components:

\begin{enumerate}
    \item elements of the form $\frac{\gamma}{\tau^{2n}}\left[\frac{1}{\rho}\right]$, which are $v_{0}$-torsion and concentrated on the $(f=0)$-line, and 

    \item $v_{0}$-towers, of the form $\mF_{2}[v_{0}]\big{ \{ }\frac{\gamma}{\tau^{2n + 1}} \big{ \} }$. 
\end{enumerate}
\end{proposition}  

The summand $\Ext_{\cE^{-}_{\star}(0)}(\mM_2, \mM_{2})$ is illustrated by the chart appearing in \cref{fig:Ext_M2NC}. Throughout we use the color orange to depict classes which belong to representation-graded shifts of $\Ext_{\cE(0)_\star}(\mM_2, \mM_{2})$. This is particularly helpful in \cref{sec:height1split} where we use a long exact sequence involving both $\Ext_{\cE(0)_\star}(\mM_2, \mM_{2})$ and $\Ext_{\cE(1)_\star}(\mM_2, \mM_{2})$. Though the color is not essential to viewing these charts, we do find it benefits the readability and take care to choose a color palette  optimized for color-blind individuals \cite{Wong2011}.  

\begin{center}
\begin{figure}[ht]
\includegraphics[width=2in,height=2in]{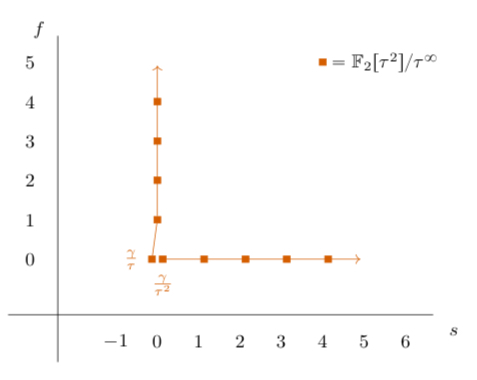}
\caption{$\Ext_{\cE^{-}_{\star}(0)}(\mM_2^\mR, \mM_{2}^{\mR})$} 
\label{fig:Ext_M2NC}
\end{figure}
\end{center}
Now that we have both the positive cone summand $\Ext_{\cE^{+}_{\star}(0)}(\mM_2^\mR, \mM_2^\mR)$ and negative cone summand $\Ext_{\cE^{-}_\star(0)}(\mM_2^\mR, \mM_{2}^{\mR})$ of $\Ext_{\cE_{\star}(n)}(\mM_{2}, \mM_{2})$, we are ready to show that the Adams spectral sequence (\ref{eq:heightZeroAss}) collapses.
\subsubsection{Running the Adams spectral sequence}
\begin{proposition}
The Adams spectral sequence 
\begin{align*}
 E_2^{s, f,w} = \Ext_{\cE(0)_{\star}}^{s,t} (H_{\star} \Sigma^{\rho k} \cB_{-1}(k), H_{\star} H \umZ ) \implies [H \umZ \wedge \cB_{-1} (k), H \umZ \wedge H \umZ]^{H \umZ}
\end{align*}
collapses at the $E_2$-page.\end{proposition}

\begin{proof}
We showed in \cref{E2:zero:split} that the $E_{2}$-page of the Adams spectral sequence has the form
\begin{align*}
E_2^{s, f, w} \cong & \Ext_{\cE(0)_{\star}}^{s,f, w} (H_{\star} \Sigma^{\rho k} \cB_{-1}(k), H_{\star} H \umZ ) \\
\cong & \Ext_{\cE(0)_{\star}}^{s,f,w} (\mM_2, \mM_2) \oplus V,
\end{align*}
where $V$ is an $\mM_{2}$-vector space concentrated in Adams filtration $f=0$. 

First, we note that no nontrivial differentials are possible from $v_{0}$-torsion elements to $v_{0}$-torsion-free elements, because the Adams differential is $v_{0}$-linear. We will continue to use this fact throughout the paper. Thus no differentials are possible from $V$ to $\Ext_{\cE(0)_{\star}}^{s,f,w} (\mM_2, \mM_2)$. Next, note that no differentials are possible from $V$ to $V$ or from $\Ext_{\cE(0)_{\star}}^{s,f,w} (\mM_2, \mM_2)$ to $V$, as the Adams differential increases filtration degree.  

Therefore the only possible nonzero differentials are those from $\Ext_{E(0)_{\star}}^{s,f,w} (\mM_2, \mM_2)$ to itself. However, $\Ext_{E(0)_{\star}}^{s,f,w} (\mM_2, \mM_2)$ is concentrated in even Milnor--Witt degree, where Milnor--Witt degree is defined to be $s-w$, that is stem minus motivic weight. Since the Adams differential has Milnor--Witt degree $-1$, this rules out any nontrivial differentials. 
\end{proof}

Thus we can lift the maps $\theta_k : \Sigma^{\rho k} H_{\star} \cB_{-1} (k) \to H_{\star} H \umZ$ to maps
\[
\Tilde{\theta}_k : H \umZ \wedge \Sigma^{\rho k} \cB_{-1} (k) \to H \umZ \wedge H \umZ
\]
for all $n \in \mN$, and we have proved a $C_{2}$-equivariant analogue of Mahowald's splitting of $H \umZ \wedge H \umZ$.

\begin{theorem}\label{thm: height zero splitting} Up to $2$-completion there is a splitting
\[ 
H \umZ \wedge H \umZ \simeq \bigvee\limits_{k=0}^{\infty} H \umZ \wedge \Sigma^{\rho k}\cB_{-1}(k) 
\]
of $H \umZ$-modules.
\end{theorem}

\begin{cor}\label{cor: height zero simples} Up to $2$-completion there is a splitting
    \[  H \umZ \wedge H \umZ \simeq H \umZ \vee Z,\]
where $Z$ is a sum of suspensions of $H$.
\end{cor}
\begin{proof}
    By definition $\cB_{-1}(0) \simeq S^{0}$, so $H \umZ \wedge \cB_{-1}(0) \simeq H \umZ$. When $k > 0$, $H_{\star}\cB_{-1}(k)$  is a sum of suspensions of $\cE(0)_{\star}$. Since $H_{\star}^{H\umZ}H\cong \cE(0)_{\star}$, so we can choose $Z_{k}$ to be a finite sum of suspensions of $H$ such that $H_{\star}^{H\mZ}Z_{k} \cong H_{\star}\cB_{-1}(k)$. We can then use the relative Adams spectral sequence \[
    \Ext_{\cE(0)_{\star}}\left(H_{\star}\cB_{-1}(k)),  H_\star^{H\umZ} Z_{k} \right) \Longrightarrow [H \umZ \wedge \cB_{-1}(0), Z_{k}]^{H\umZ} 
    \]  
    to lift the isomorphism to an equivalence of $H\umZ$-module spectra. 
    
\end{proof}

\section{Splitting {$ku_\mR \wedge ku_\mR$}}\label{sec:height1split} Let $\cB_0(k)$ denote the $C_2$-equivariant integral Brown--Gitler spectrum defined in \cite{LiPetersenTatum23}. In this section we construct a family of $ku_\mR$-module maps
\[
\Tilde{\theta}_k: \Sigma^{\rho k} ku_\mR \wedge \cB_0(k)\to ku_\mR \wedge ku_\mR
\]
such that 
\[
\bigvee_{k=0}^\infty \Tilde{\theta}_k: \Sigma^{\rho k} ku_\mR \wedge \cB_0(k) \xrightarrow{\simeq} ku_\mR \wedge ku_\mR
\]
is an equivalence (up to 2-completion). 

\subsection{Strategy}\label{sec:heightonestrategy} Our strategy is similar to that of \cref{sec:heightzerostrategy} where we constructed a splitting of $H \umZ \wedge H \umZ.$ We begin by considering the $\mathcal{E}(1)_\star$-comodule isomorphism
$$ \bigoplus_{k = 0}^\infty \theta_k : H_\star \Sigma^{k \rho} B_{0} (k) \to H_\star ku_\mR $$
given by \cref{thm:homology:decomp} in the case where $n = 1$. We write these as a family of maps
$$ \bigoplus_{k = 0}^\infty \theta_k : H_\star^{ku_\mR} \left(ku_\mR \wedge \Sigma^{\rho k} \cB_{0} (k)\right) \to H_\star^{ku_\mR} (ku_\mR \wedge ku_\mR).$$
and then study the
Adams spectral sequence
\begin{equation}\label{eq:AdamsSs}
     E_2^{s, f, w} = \Ext_{\mathcal{E}(1)_\star} \left(H_\star \Sigma^{\rho k} \cB_0(k), H_\star ku_\mR \right) \implies [ku_\mR \wedge \Sigma^{\rho k} \cB_{0} (k),ku_\mR \wedge ku_\mR]^{ku_\mR}.
\end{equation} 
Since $E_{2}^{s,f,w}$ is finite in each degree, the spectral sequence converges \cite[Thm~15.6, Thm~7.1]{boardman1999conditionally}. Moreover, each $\theta_k$ is in $E_2^{0,0,0}$, so constructing maps 
\[
\Tilde{\theta}_k : ku_\mR \wedge \Sigma^{k \rho} \cB_{0} (k)\to ku_\mR \wedge ku_\mR
\]
for all $k \in \mN$ is the same as showing that the $\theta_k$ survive the spectral sequence.

On our way to showing the $\theta_k$ survives the Adams spectral sequence, it will be helpful to first record a simpler splitting of $ku_{\mR}\wedge \cB_{0}(k)$ and $ku_{\mR} \wedge ku_{\mR}$. This will allow us to decompose the Adams spectral sequence (\ref{eq:AdamsSs}) into 
a sum of four separate spectral sequences.

\begin{proposition}\label{prop: other splitting}
    There are $ku_\mR$-module splittings
    \[ ku_\mR \wedge \cB_{0}(k) \simeq C_{k} \vee V_{k} ,\]
    \[ ku_\mR \wedge  ku_\mR\simeq C \vee V, \]
 where $V_{k}$ and $ V$ are sums of suspensions of $H$, and $C_{k}$ and $ C$ contain no $H$-summands.
\end{proposition}

\begin{proof}
In \cref{bg lightning split} we showed that \[H_{\star}\cB_{0}(k) \cong L(\nu_{p}(k!)) \oplus W_{k},\] where $W_{k}$ is a finite sum of suspensions of $\cE(1)_{\star}$. In \cref{prop:relHomology}, we also showed that $H_{\star}^{ku_\mR}H \cong \cE(1)_{\star}$. Take $V_{k}$ to be a sum of suspensions of $H$ such that $H_{\star}^{ku_{\mR}}V_{k} \cong W_{k}$. We can use the Adams spectral sequences
\[ E_{2}^{s,f,w} \cong \Ext_{\cE(1)_{\star}}(H_{\star}^{ku_{\mR}}V_{k}, H_{\star}\cB_{0}(k)) \Longrightarrow [V_{k}, ku_\mR \wedge \cB_{0}(k) ]^{ku_{\mR}} \]
\[ E_{2}^{s,f,w} \cong \Ext_{\cE(1)_{\star}}( H_{\star}\cB_{0}(k), H_{\star}^{ku_{\mR}}V_{k}) \Longrightarrow [ ku_\mR \wedge \cB_{0}(k), V_{k} ]^{ku_{\mR}} \]
to lift the homology splitting $H_{\star}B_{0}\left( k \right) \cong L(\nu_{2}(k!)) \oplus W_{k}$ to a splitting of $ku_\mR$-module spectra by viewing the inclusion $i: H_{\star}^{ku_{\mR}}V_{k} \cong W_k \hookrightarrow H_{\star}\cB_{0}(k)$ as a class in filtration zero in the first $E_{2}$-page, and the projection $j: H_{\star}\cB_{0}(k) \to W_k \cong H_{\star}^{ku_{\mR}}V_{k}$ as a class in filtration zero of the second $E_{2}$-page. Since $\cE(1)_\star$ is a free and injective module (see \cite{Ricka15} and also \cref{equiv:free:inj}), both spectral sequences are entirely concentrated in filtration $f=0$. Therefore, there are no differentials and both the class of the inclusion and the class of the projection lift to maps of $ku_{\mR}$-module spectra
\[V_{k} \to ku_{\mR} \wedge \cB_{0}(k) \to V_{k},\]
and we have a splitting $ku_{\mR} \wedge B_{0}(k) \simeq V_{k} \vee C_{k}$.

To prove $ku_\mR \wedge ku_\mR \simeq C \vee V,$
consider the $\cE(1)_\star$-comodule isomorphism
\[
H_\star ku_\mR \cong \bigoplus_{k = 0}^\infty \Sigma^{\rho k} H_\star \cB_{0}(k) 
\]
of \cref{thm:homology:decomp}. Composing this isomorphism with the isomorphism $H_{\star}\cB_{0}\left( k \right) \cong L(\nu_{2}(k!)) \oplus W_{k}$ of \cref{bg lightning split} yields
\[H_{\star}ku_\mR \cong \bigoplus_{k=0}^{\infty}\Sigma^{\rho k}L(\nu_{p}(k!)) \oplus H_{\star}V,\] 
where $V$ is a finite-type sum of copies of $H$. By the same spectral sequence argument as the first splitting, 
\[ ku_\mR \wedge ku_\mR \simeq ku_\mR \vee V .\]
\end{proof}

As an immediate consequence, we get a decomposition of the relative Adams spectral sequence (\ref{eq:AdamsSs}).

\begin{proposition} \label{prop:splitInto4}
   The Adams spectral sequence  
\[
 E_2^{s, t} = \Ext_{\mathcal{E}(1)_\star} \left(H_\star \Sigma^{\rho k} B_0(k), H_\star ku_\mR \right) \implies [ku_\mR \wedge \Sigma^{\rho k} B_{0} (k),ku_\mR \wedge ku_\mR]^{ku_\mR}
\]decomposes into a sum of four separate spectral sequences, listed below. All but the first is concentrated on the $(f=0)$-line.
\[ \Ext_{\cE(1)_{\star}}(H_{\star}\Sigma^{\rho k}C_{k}, H_{\star}C) \Longrightarrow [\Sigma^{\rho k}C_{k}, C], \]
\[ \Ext_{\cE(1)_{\star}}(H_{\star}\Sigma^{\rho k}C_{k}, H_{\star}V) \Longrightarrow [\Sigma^{\rho k}C_{k}, V], \]
\[ \Ext_{\cE(1)_{\star}}(H_{\star}\Sigma^{\rho k}V_{k}, H_{\star}C) \Longrightarrow [\Sigma^{\rho k}V_{k}, C], \]
\[ \Ext_{\cE(1)_{\star}}(H_{\star}\Sigma^{\rho k}V_{k}, H_{\star}V) \Longrightarrow [\Sigma^{\rho k}V_{k}, V] .\]
\end{proposition}

Since each summand containing $V$ or $V_{k}$ consists solely of copies of suspensions of $\mM_{2}$ on the filtration $(f=0)$-line, we will focus on the spectral sequence 
\[ \Ext_{\cE(1)_{\star}}(H_{\star}\Sigma^{\rho k}C_{k}, H_{\star}C) \Longrightarrow [\Sigma^{\rho k}C_{k}, C] .\]

By construction, 
\begin{align*}
    H_{\star}^{ku_{\mR}} C_{k} & \cong L(\nu_{2}(k!)), \\
    H_{\star}^{ku_{\mR}} C & \cong \bigoplus_{m=0}^{\infty}\Sigma^{\rho m}L(\nu_{2}(m!)).
\end{align*}
Thus in order to calculate $\Ext_{\cE(1)_{\star}}(H_{\star}\Sigma^{\rho k}C_{k}, H_{\star}C)$, we inductively compute 
\[
\Sigma^{\rho(m-k)}\Ext_{\mathcal{E}(1)_\star} (L(\nu_2(k!)), L(\nu_2(m!)))
\]
using the long exact sequences induced by applying  $\Ext_{\cE(1)_\star} (\text{---}, \, L(m))$ and \newline $\Ext_{\cE(1)_\star} (L(m), \, \text{---})$ to the short exact sequence
\begin{align} \label{eq:ses}
    0 \to \Sigma^\rho L(k - 1) \to L(k) \to \cE(1)// \cE(0)_\star \to 0.
\end{align}
We will phrase our long exact sequence computations in terms of the spectral sequence associated to each long exact sequence.

Since $L(0) \cong \mM_{2}$, the base case for this inductive computation is $\Ext_{\cE(1)_{\star}}(\mM_{2}, \mM_{2})$. This $\Ext$ term is also the $E_{2}$-page of the Adams spectral sequence for $\pi_{\star}ku_{\mathbb{R}}$ and is computed in \cite{GuillouHillIsaksenRavenel2020}. In order to build inductively on their result, we describe $\Ext_{\cE(1)_{\star}}(\mM_{2}, \mM_{2})$.

Similarly to the height zero case, the square-zero extension $\mM_{2} \cong \mM_{2}^{\mR} \oplus NC$ induces a decomposition 
\[ \Ext_{\cE(1)_{\star}}(\mM_{2}, \mM_{2}) \cong \Ext_{\cE_{\star}^{\mR}(1)}(\mM_{2}^{\mR}, \mM_{2}^{\mR}) \oplus \Ext_{\cE_{\star}^{\mR}(1)}(NC, \mM_{2}^{\mR}) \]
\cite[Proposition~2.2]{GuillouHillIsaksenRavenel2020}. We will use the notation 
\[ \Ext_{\cE^{+}_{\star}(1)}(\mM_{2}, \mM_{2}) := \Ext_{\cE^{\mR}_{\star}(1)}(\mM_{2}^{\mR}, \mM_{2}^{\mR})\]
\[ \Ext_{\cE^{-}_{\star}(1)}(\mM_{2}, \mM_{2}) := \Ext_{\cE^{\mR}_{\star}(1)}(NC, \mM_{2}^{\mR}).\]

In \cite[Theorem~3.1]{Hill11}, Hill computed $\Ext_{\cE^{+}_{\star}(1)}(\mM_{2}, \mM_{2}) .$ In the notation of \cite[Proposition 6.3]{GuillouHillIsaksenRavenel2020},
\begin{equation}\label{eqn: Ext_mR}
    \Ext_{\cE^{+}_{\star}(1)}(\mM_{2}, \mM_{2})  \cong \mF_2 [\rho, \tau^4, v_0, \tau^2 v_0, v_1]/(\rho v_0, \rho^3 v_1, (\tau^2 v_0)^2 + \tau^4 v_0^2)
,
\end{equation}

where the stem $s,$ filtration $f,$ and motivic weight $w$ of the generators are given in \cref{tab:ExtM2M2gens}. The negative cone summand, $\Ext_{\cE_\star^{-} (1)} (\mM_2, \mM_2):= \Ext_{\cE_{\star}^{\mR} (1)} (NC, \mM_2^{\mR})$, is a module over $\Ext_{\cE_\star^{+} (1)} (\mM_2, \mM_2)$. We list the generators of $\Ext_{\cE_\star^{-} (1)} (\mM_2, \mM_2)$ below the horizontal line in \cref{tab:ExtM2M2gens}.  There are relations of the form $v_{1}\frac{\gamma}{\tau^{4i + 3}} = v_{0}\frac{\gamma}{\rho^{2}\tau^{4i + 2}}$ for all $i$. The other relations will be clear from the names of the generators (for example, $\tau^{2}v_{0}\frac{\gamma}{\tau^{3}} = \frac{\gamma}{\tau} v_{0}$).
\begin{table}[ht]
    \centering
    \begin{tabular}{l l l}
          $(s,\, f,\, w)$ & Element & \\
            \hline
             $\,$ & $\,$ \\
             $(-1, 0, -1)$ & $\rho$ \\
             $(0, 1, 0)$ & $v_0$ \\
            $(2, 1, 1)$ & $v_1$ \\
            $(0, 1, -2)$ & $\tau^2 v_0$ \\
            $(0, 0, -4)$ & $\tau^4$ \\
           $\,$ & $\,$ \\
           \hline
           \, & \, \\
           $(0,0, 4j+2)$ & $\frac{\gamma}
           {\tau^{4j + 1}}$ & $0 \le j$\\
           $\left(2,0,4(j+1) + 1\right)$ & $\frac{\gamma}
           {\tau^{4j + 2}\rho^{2}}$ & $0 \le j$ \\
           $(0,0,4j + 4)$ & $\frac{\gamma}
           {\tau^{4j + 3}}$ & $0 \le j$ \\
           $(i,0,  4j  + i + 1)$ & $\frac{\gamma}
           {\tau^{4j}\rho^{i}}$ & $0 < j, 0 \le i$ \\
           \label{ExtM2gens}
    \end{tabular}
    \caption{Generators for $\Ext_{\cE (1)_\star} (\mM_2, \mM_2)$. The generators above the horizontal line are the generators of the ring $\Ext_{\cE^{+}_\star (1)} (\mM_2, \mM_2)$. The elements below the horizontal line generate $\Ext_{\cE^{-}_\star (1)} (\mM_2, \mM_2)$ as an $\Ext_{\cE^{+}_\star (1)} (\mM_2, \mM_2)$-module.}
    \label{tab:ExtM2M2gens}
\end{table}

In charts, $\Ext_{\cE(1)_\star} (\mM_2, \mM_2)$ is given in Figure \ref{fig:ExtM2M2}. The horizontal axis is the stem ($s$), and the vertical axis is the Adams filtration $(f)$. Note that the motivic weight ($w$) is suppressed in this depiction. We will use this grading convention in each of the following charts. Vertical lines denote $v_{0}$-multiplication, and horizontal lines denote $\rho$-multiplication. For readability, we follow the convention of \cite{GuillouHillIsaksenRavenel2020} and split the charts into a Part A and Part B for each positive and negative cone summand of each $\Ext$-group. 

If a generator of a $v_{1}$-tower in stem $2n$ or $2n-1$ has motivic weight congruent to $n$ mod $4$, we assign it to Part A of \cref{fig:ExtM2M2}. If a generator of a $v_{1}$-tower in stem $2n$ or $2n-1$ has motivic weight congruent to $n-2$ mod $4$, we assign it to Part B of \cref{fig:ExtM2M2}. Likewise, $v_{1}$-torsion generators are assigned so that in Part A, any generator of a $\rho$-tower in stem $2n$ or stem $2n-1$ is in motivic weight congruent to $n$ or $n+3$ mod $4$. In Part B, any generator of a $\rho$-tower in stem $2n$ or stem $2n-1$ is in motivic weight congruent to $n+2$ or $n+1$ mod $4$. We will call a class the ``base'' of a $\rho$-divisible tower if the rest of the tower can be obtained by $\rho$-division, and we assign these towers to Part A and Part B using the same rules for the bases. Within each chart, every possible $v_1$-extension occurs. Note that for degree reasons, no $v_{1}$-extensions are possible between the charts. The only extensions between the charts are action by $v_{0}\tau^{2}$. These extensions should be clear from the names of the generators. This organizes the charts in such a way that no differentials in subsequent $\Ext$ and Adams spectral sequence computations are possible between the charts. 

In \cref{fig:ExtM2M2} and our remaining charts, we also make use of color. Specifically, we use blue to depict classes in Parts A of the positive and negative cone summands of representation-graded shifts of $\Ext_{\cE(1)_\star}(\mM_2, \mM_2)$. Similarly, we use green to depict classes in Parts B of the positive and negative cone summands of representation-graded shifts of $\Ext_{\cE(1)_\star}(\mM_2, \mM_2)$. We find this increases the readability of subsequent figures, particularly those involving representation-graded shifts of both $\Ext_{\cE(0)_\star}(\mM_2, \mM_2)$ and $\Ext_{\cE(1)_\star}(\mM_2, \mM_2)$. 

\begin{center}
\begin{figure}[!ht]
\includegraphics[width=5in, height=4in]{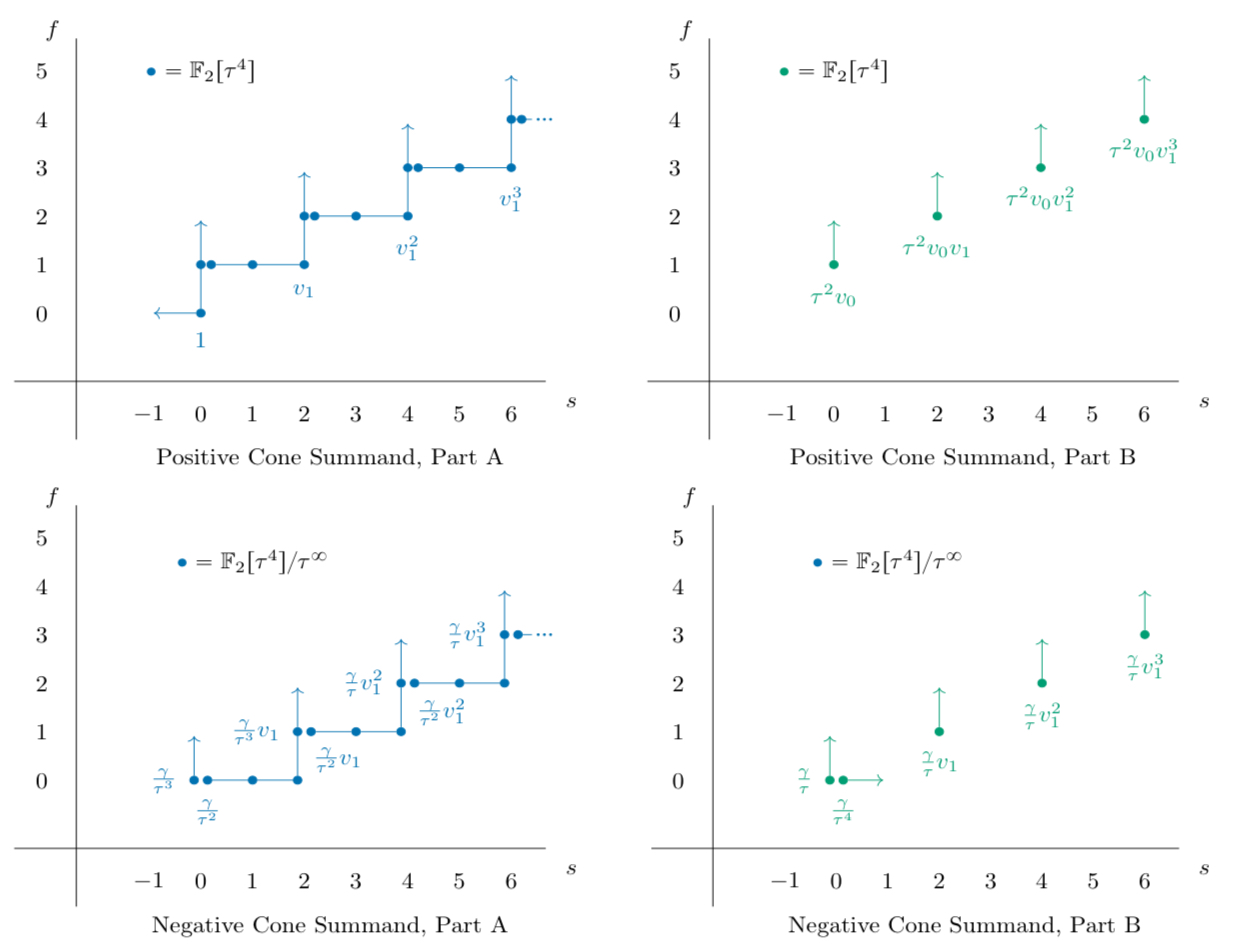}
\caption{$\Ext_{\mathcal{E}_\star(1)} (\mM_2, \mM_2)$} 
\label{fig:ExtM2M2}
\end{figure}
\end{center}

Since $L(k)$ and $\cE(1)_{\star}$ are free over $\mM_{2}$, there is a natural $\Ext_{\cE^{+}_\star (1)}(\mM_{2}, \mM_{2})$-module structure on $\Ext_{\cE(1)_{\star}}(L(k), L(m))$ \cite[p.211-212]{cartan1999homological}.
In particular, we will use the fact that the differential $d$ is a map of $\Ext_{\cE^{\mR}(1)_{\star}}^{+}(\mM_{2}, \mM_{2})$-modules, and thus is $v_{0}$, $v_{1}$, $\rho$, $\tau^{4}$, and $v_{0}\tau^{2}$-linear.

Moreover, we will also need to consider $\Ext_{\cE(0)_\star}^{s,f,w}(\mM_2, \mM_2)$ as an $\Ext_{\cE^+_\star(1)}^{f,s,w}(\mM_2, \mM_2)$-module in our calculations. To that end, the following proposition is helpful. 
\begin{proposition} \label{prop:ExtE0asE1}
    As an $\Ext_{\cE^{+}_\star (1)}(\mM_2, \mM_2)$-module,
\begin{align*}
    \Ext_{\cE^{+}_\star (0)}(\mM_2, \mM_2)  \cong & \Ext_{\cE^{+}_\star (1)}^{f,s,w}(\mM_2, \mM_2)/ (v_1) \{1, \tau^2\} \\
    & \oplus \Ext_{\cE^{+}_\star (1)}^{f,s,w}(\mM_2, \mM_2) / (v_1) \left\{\frac{\gamma}{\tau^{4i}\rho^{j}},\frac{\gamma}{\tau^{4i-2}\rho^{j}}, \frac{\gamma}{\tau^{4i-1}},\frac{\gamma}{\tau^{4i-3}}\right\}
,
\end{align*}
where $i \ge 1, j\ge 0$.
\end{proposition}
The module structure is the natural one given by the change-of-rings isomorphism 
\[
\Ext_{\cE(0)_\star}^{s,f,w}(\mM_2, \mM_2) \cong \Ext_{\cE(1)_\star}^{s,f,w}(\mM_2, \cE(1)//\cE(0)_{\star} ).
\]
\cref{fig:ExtE0asE1Module} depicts this structure, according to our chart conventions for $\Ext_{\cE(1)_{\star}^{+}}(\mM_{2}, \mM_{2})$-modules described above. 

\begin{center}
\begin{figure}[ht]
\includegraphics[width=4in,height=4in]{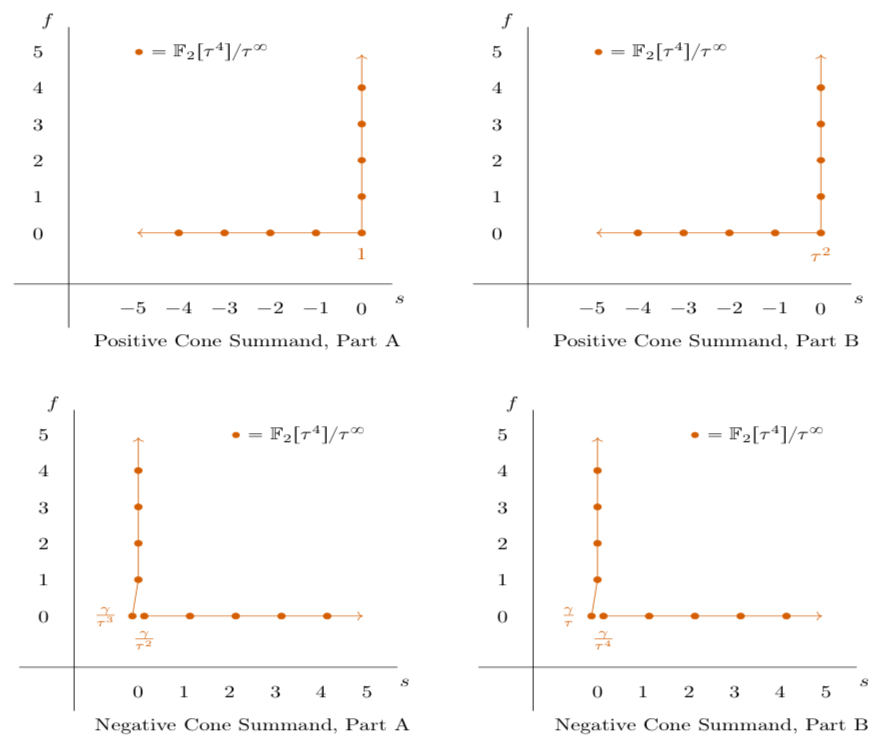}
\caption{$\Ext_{\cE(0)_\star}^{s,f,w}(\mM_2, \mM_2) $ as an $\Ext_{\cE^\mR_\star(1)}(\mM_2, \mM_2)$-module} 
\label{fig:ExtE0asE1Module}
\end{figure}
\end{center}

Computing inductively, we find that $\Ext_{\cE(1)_{\star}}^{s,f,w}(\mM_{2}, L(m) )$ has the structure that the reader familiar with the non-equivariant case might hope for: up to $v_{1}$-extensions, it consists of a shifted copy of $\Ext_{\cE(1)_{\star}}^{s,f,w}(\mM_{2}, \mM_2 )$, along with a sum of shifted copies of $\Ext_{\cE(0)_{\star}}^{s,f,w}(\mM_{2}, \mM_{2} )$. The equivariant $v_{1}$-extensions also parallel the nonequivariant case.
\begin{proposition}\label{k:0}
As an $\Ext_{\cE(1)_{\star}}^{s,f,w}(\mM_{2}, \mM_{2} )$-module, 
\[\Ext_{\cE(1)_{\star}}^{s,f,w}(\mM_{2}, L(m) ) \cong \Ext_{\cE(0)_\star}(\mM_2, \mM_2)\{x_{0},\ldots, x_{m-1}\} \oplus  \Ext_{\cE(1)_\bigstar}(\mM_2, \mM_2) \{x_m\},\] 
where $|x_{i}| = (2i,\, 0,\, i)$ and there are extensions $v_{1}x_{i} = v_{0}x_{i+1}$ for each $i$. 
\end{proposition}
\begin{proof} We proceed by induction on $m$. For the base case $m=0$, we have $L(0) = \mM_{2}$ and so indeed $\Ext_{\cE(1)_{\star}}^{s,f,w}(\mM_{2}, L(0)) \cong \Ext_{\cE(1)_{\star}}^{s,f,w}(\mM_{2}, \mM_{2})\{x_{0}\}$ with $|x_{0}| = (0,0,0)$. To begin the induction, suppose that $\Ext_{\cE(1)_{\star}}^{s,f,w}(\mM_{2}, L(m') )$ satisfies the above description when $m' < m$ and consider the long exact sequence 
\begin{align*}\label{eqn: les for IH on m}
   \cdots \rightarrow \Ext_{\cE(1)_{\star}}^{s,f,w}(\mM_{2}, L(m) ) \rightarrow  \Ext_{\cE(1)_{\star}}^{s,f,w}(\mM_{2}, \cE(1)//\cE(0)_{\star} ) \\
   \xrightarrow{d} \Ext_{\cE(1)_{\star}}^{s-1,f+1,w}(\mM_{2}, \Sigma^{\rho}L(m-1) ) \rightarrow \cdots.
\end{align*}
induced by the short exact sequence 
\[
 0 \to \Sigma^\rho L(m - 1) \to L(m) \to \cE(1)// \cE(0)_\star \to 0
\]
of (\ref{eq:ses}).

First, we use change-of-rings to write 
\[ 
\Ext_{\cE(1)_{\star}}^{s,f,w}(\mM_{2}, \cE(1)//\cE(0)_{\star} ) \cong \Ext_{\cE(0)_{\star}}^{s,f,w}(\mM_{2}, \mM_{2} ).
\]
Then the induction hypothesis together with the description of $\Ext_{\cE(0)_{\star}}^{s,f,w}(\mM_{2}, \mM_{2} )$ in \cref{prop:ExtE0asE1} implies that the differential
\[
d: \Ext_{\cE(0)_{\star}}^{s,f,w}(\mM_{2},\cE(1)//\cE(0)_{\star} ) \rightarrow \Ext_{\cE(1)_{\star}}^{s-1,f+1,w}(\mM_{2}, \Sigma^{\rho}L(m-1) )
\] must be zero for degree reasons. For example, this can be observed in \cref{fig:ExtM2L3-E1}. Note that in \cref{fig:ExtM2L3-E1}, the classes have been relabeled as described at the end of the proof.

Therefore, 
\begin{align*}
    \Ext_{\cE(1)_{\star}}^{s,f,w}(\mM_{2}, L(m) ) & \cong \Ext_{\cE(0)_{\star}}^{s,f,w}(\mM_{2},\cE(0)_{\star} ) \oplus \Ext_{\cE(1)_{\star}}^{s,f,w}(\mM_{2}, \Sigma^{\rho}L(m-1) ) \\
    & \cong \Ext_{\cE(0)_{\star}}^{s,f,w}(\mM_{2},\cE(0)_{\star} ) \oplus \Sigma^{\rho} \Ext_{\cE(1)_{\star}}^{s,f,w}(\mM_{2}, L(m-1) ),
\end{align*}
up to extensions. 

Let $x$ denote the generator of $\Ext_{\cE(0)_{\star}}^{s,f,w}(\mM_{2},\mM_{2} )$ as an $\Ext_{\cE(1)_\star}(\mM_2, \mM_2)$-module. By comparison with the underlying calculation (\cref{nonequiv: Ext computation}), there must be an extension $v_1 x = v_0 (\Sigma^\rho x_0).$ There are no other $\Ext_{\cE(1)_{\star}}(\mM_{2}, \mM_{2})$-module extensions for degree reasons.

Relabeling $x$ by $x_{0}$ and $\Sigma^{\rho}x_{i}$ by $x_{i+1}$ finishes the proof.
\end{proof}

\begin{center}
\begin{figure}[ht]
\includegraphics[width=5in,height=4in]{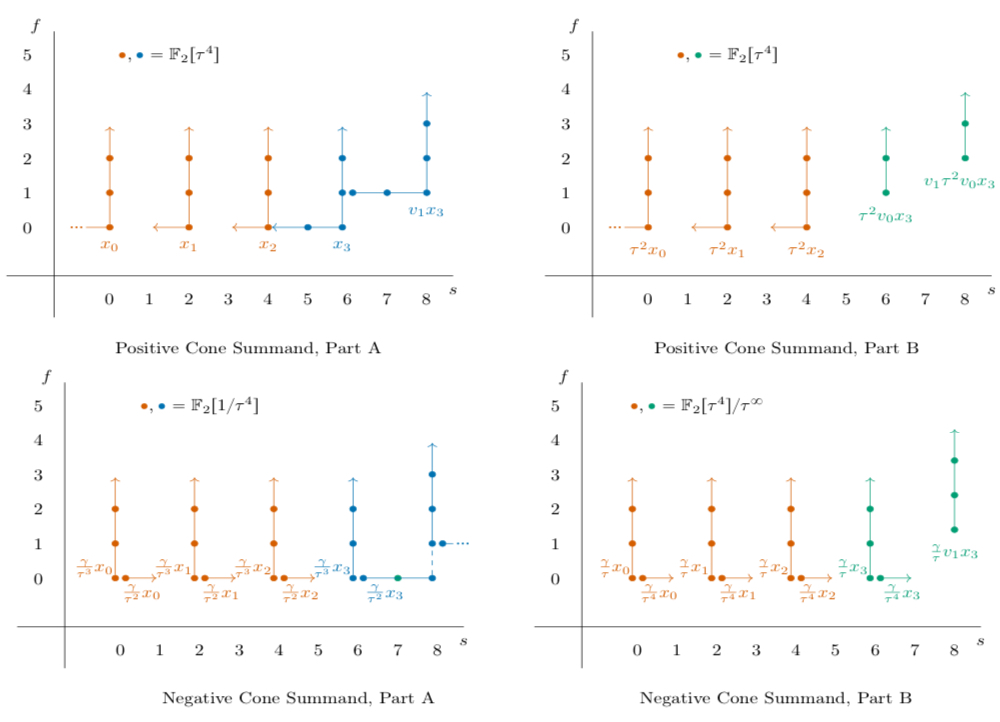}
\caption{$\Ext_{\mathcal{E}(1)_\star} (\mM_2, L(3))$} \label{fig:ExtM2L3-E1}
\end{figure} 
\end{center}
We now introduce two lemmas that will be helpful for computing $\Ext_{\cE(1)_{\star}}(L(k), L(m))$ when $k >0.$  
\begin{lemma}\label{lem:change of rings}
    There is a `wrong-side' change-of-rings isomorphism 
    \[\Ext_{\cE(1)_{\star}}^{s,f,w}( \cE(1)//\cE(0)_{\star} ,\, - \,) \cong \Sigma^{-\rho - 1} \Ext_{\cE(0)_{\star}}^{s,f,w}(\mM_{2}, \, - \,).\]
\end{lemma}
\begin{proof} 
   By the equivalence of categories between left $\cE(n)_\star$-comodules and right $\cE(n)$-\newline modules (\cref{prop:equiv of cat}),
   \[
   \Ext_{\cE(1)_{\star}}^{s,f,w}( \cE(1)//\cE(0)_{\star} , \, - \,) \cong \Ext_{\cE(1)}^{s,f,w}( \cE(1)//\cE(0)_{\star} , \, - \, ).
   \]

    As an $\cE(1)$-module, $\cE(1)//\cE(0)_{\star} \cong \Sigma^{\rho + 1}\cE(1)//\cE(0)$. By ordinary change-of-rings, 
    \[ \Ext_{\cE(1)}^{s,f,w}( \Sigma^{\rho + 1}\cE(1)//\cE(0) , \, - \, ) \cong \Sigma^{-\rho - 1} \Ext_{\cE(0)}^{s,f,w}( \mM_{2} , \, - \,).\]
    Applying the equivalence of categories again concludes the proof.
\end{proof}

\begin{lemma}\label{homs lemma}
Suppose $m > k$. Then if $s<0$, and $x \in\Hom_{\cE(1)_{\star}}^{2s,w}(L(k), L(m))$ for any $w$, then $\rho$ divides $x$. 
\end{lemma}
This is a straightforward computation and is best checked by drawing the relevant lightning flash modules.

We are now ready to compute $\Ext_{\cE(1)_{\star}}(L(k), L(m))$ when $k > 0.$ We proceed by fixing $m \geq 0$ and inducting on $k.$ Similarly to the nonequivariant setting (\cref{nonequiv: Ext computation}), the resulting $\Ext$ groups have different forms depending on whether $k$ is greater than $m.$ Thus we first describe $\Ext_{\cE(1)_{\star}}(L(k), L(m))$ when $k \le m$ before moving on to the case where $k > m.$
\begin{proposition}\label{m:0}
    Suppose $k \le m$. Then 
    \[\Ext_{\cE(1)_{\star}}(L(k), L(m)) \cong \Ext_{\cE(1)_{\star}}(\mM_{2}, L(m-k)) \oplus V,\] where $V$ consists of $v_{0}$ and $v_{1}$-torsion concentrated on the $(f=0)$-line, specifically: \begin{enumerate}
        \item $\rho$-towers in odd stem $s$ and filtration $f=0$, and
        \item infinitely $\rho$-divisible towers in filtration $f=0$.
    \end{enumerate}
\end{proposition}

\begin{center}
\begin{figure}[ht]
\includegraphics[width=5in,height=4in]{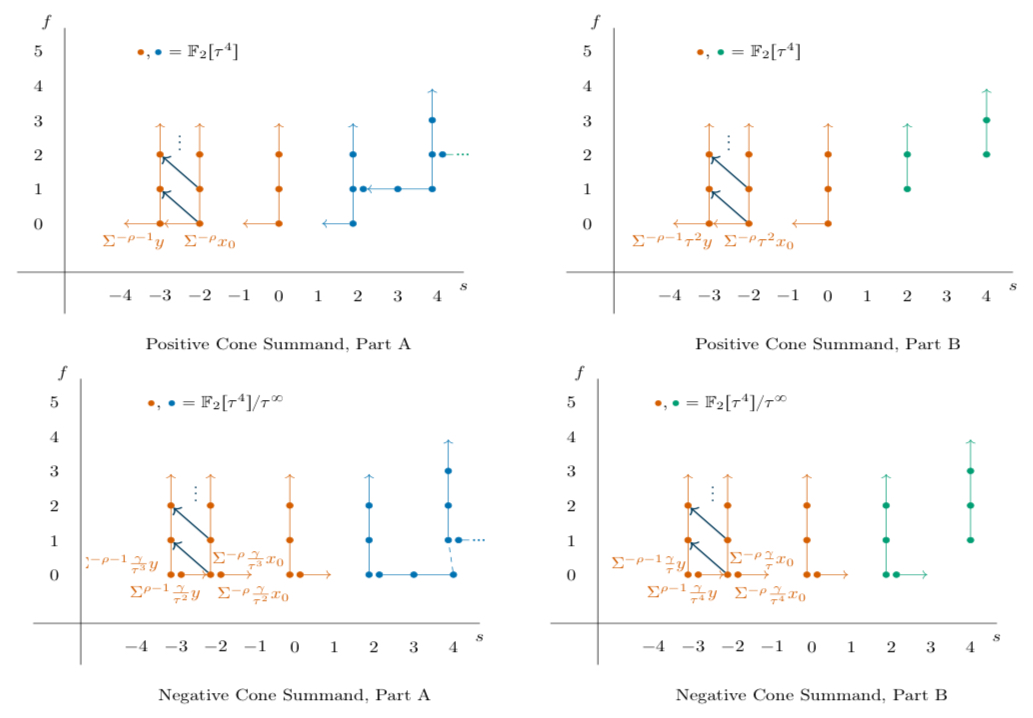}
\caption{$E_1$-page computing $\Ext_{\cE(1)_\star}(L(1), L(2))$}
\label{fig:E1ExtL1L2}
\end{figure}
\end{center}

\begin{proof}
    The base case of the induction, when $k=0$, holds by \cref{k:0}. Suppose that the claim holds for $k-1$ and consider the long exact sequence 
\begin{align}\label{eqn:les of k<m}\begin{split}
   \cdots \rightarrow \Ext_{\cE(1)_{\star}}^{s,f,w}(L(k), L(m) ) \rightarrow  \Ext_{\cE(1)_{\star}}^{s,f,w}(\Sigma^{\rho}L(k-1) , L(m)) \\
   \xrightarrow{d} \Ext_{\cE(1)_{\star}}^{s-1,f+1,w}( \cE(1)//\cE(0)_{\star} , L(m) ) \rightarrow \cdots.
   \end{split}
\end{align}
induced by the short exact sequence 
\[
 0 \to \Sigma^\rho L(k - 1) \to L(k) \to \cE(1)// \cE(0)_\star \to 0.
\]

By the induction assumption,
\[ \Ext_{\cE(1)_{\star}}^{s,f,w}(\Sigma^{\rho}L(k-1) , L(m)) \cong  \Ext^{s,f,w}_{\cE(1)_{\star}}(\mM_{2}, L(m-k+1)) \oplus V.\]
The `wrong-side' change of rings isomorphism (\cref{lem:change of rings}) gives
\[ \Ext_{\cE(1)_{\star}}^{s,f,w}( \cE(1)//\cE(0)_{\star} , L(m) ) \cong \Sigma^{-\rho - 1}\Ext_{\cE(0)_{\star}}^{s,f,w}( \mM_{2}, L(m) ) .\]
So we can rewrite the spectral sequence associated to the long exact sequence (\ref{eqn:les of k<m}) as
\begin{align}\label{eqn: spec seq for k > m}
    \Sigma^{-\rho - 1}\Ext_{\cE(0)_{\star}}^{s,f,w}( \mM_{2}, L(m) ) \oplus \Ext_{\cE(1)_{\star}}^{s,f,w}(\Sigma^{\rho}L(k-1) , L(m))
    \implies  \Ext_{\cE(1)_{\star}}(L(k), L(m)).
\end{align}

\begin{center}
\begin{figure}[ht]
\includegraphics[width=5in,height=4in]{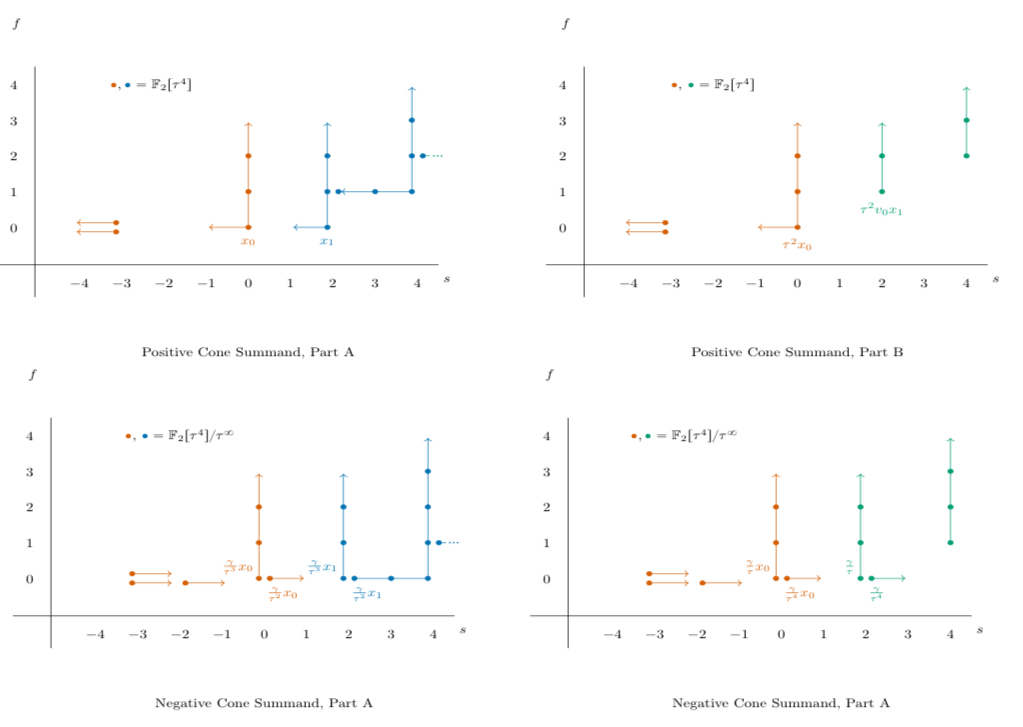}
\caption{$\Ext_{\cE(1)_\star}(L(1), L(2))$}
\label{fig:ExtL1L2}
\end{figure}
\end{center}

Note that as an $\cE(0)_{\star}$-comodule,
\[
L(m) \cong \mM_{2} \oplus \Sigma^{\rho + 1}\cE(0)_{\star} \oplus \cdots \oplus \Sigma^{m\rho+1}\cE(0)_{\star}.
\] 
This implies \begin{align}\label{eqn:Ext0(0,m)}\Ext_{\cE(0)_{\star}}\left(\mM_{2}, L(m) \right)\{y \} \cong \Ext_{\cE(0)_{\star}}\left(\mM_{2}, \mM_{2} \right) \oplus W,\end{align}
where $|y| = (0,0,0)$ and $W$ is a sum of odd-stem suspensions of $\mM_{2}$. 

Observe that the $v_{0}$-towers of $\Ext_{\cE(1)_{\star}}(\Sigma^{\rho}L(k-1), L(m))$ have the following generators in stem $-2$: $\Sigma^{-\rho}\tau^{2i}x_{0}$ and $\Sigma^{-\rho}\frac{\gamma}{\tau^{2i + 1}}{x_{0}}$ for $i \ge 0$. Likewise, the $v_{0}$-towers in stem $-3$ in $\Sigma^{-\rho - 1}\Ext_{\cE(0)_{\star}}^{*,*,*}( \mM_{2}, L(m))$ are generated by $\Sigma^{-\rho - 1}\tau^{2i}y$ and $\Sigma^{-\rho - 1}\frac{\gamma}{\tau^{2i + 1}}y$ for $i \ge 0$ (see \cref{fig:E1ExtL1L2}, for example). For each $\Sigma^{-\rho}\tau^{2i}x_{0}$, the only potential nonzero differential from $\Sigma^{-\rho}x_{0}$ goes to $\Sigma^{-\rho - 1}\tau^{2i}v_{0}y$. Likewise, the only potential nonzero differential from $\Sigma^{-\rho}\frac{\gamma}{\tau^{2i + 1}}{x_{0}}$ goes to $\Sigma^{-\rho - 1}\frac{\gamma}{\tau^{2i + 1}}v_{0}y$. \cref{homs lemma} implies that $\Ext^{-2,0,*}( L(k), L(m) )$ must be zero, so indeed
\begin{align*}
    & d(\Sigma^{-\rho}\tau^{2i}x_{0}) = \Sigma^{-\rho - 1}\tau^{2i}v_{0}y \\
    & d(\Sigma^{-\rho}\frac{\gamma}{\tau^{2i + 1}}{x_{0}}) = \Sigma^{-\rho - 1}v_{0}\frac{\gamma}{\tau^{2i + 1}}y
\end{align*} 
for all $i \ge 0$.  

Since the differential $d$ is $v_{0}$-linear, this implies 
\begin{align*}
    d(\Sigma^{-\rho}v_{0}^{j}\tau^{2i}x_{0}) & = \Sigma^{-\rho - 1}\tau^{2i}v_{0}^{j+1}y \\
    d(\Sigma^{-\rho}v_{0}^{j}\frac{\gamma}{\tau^{2i + 1}}{x_{0}}) & = \Sigma^{-\rho - 1}v_{0}^{j+1}\frac{\gamma}{\tau^{2i + 1}}y
\end{align*} and for all $i,j \ge 0$. For degree reasons, no other nontrivial differentials are possible. There is no room for other $\Ext_{\cE^{\mR}(1)_{\star}}(\mM_{2}, \mM_{2})$-module extensions. Thus
\[\Ext_{\cE(1)_{\star}}(L(k), L(m)) \cong \Ext_{\cE(1)_{\star}}(\mM_{2}, L(m-k)) \oplus V,\]
as illustrated in Figures \ref{fig:E1ExtL1L2} and \ref{fig:ExtL1L2} which depict the case where $m =2$ and one inducts from $k=0$ to $k = 1.$
\end{proof}

Having computed $\Ext_{\cE(1)_{\star}}(L(k), L(m))$ for fixed $m \geq 0$ by induction on $k \le m,$ we are now ready to continue the induction for $k > m.$ 

\begin{lemma}\label{lem:k>m}
The positive cone $\Ext_{\cE^{+}_\star(1)}^{*,*,*}(L(k), L(m))$ when $k > m$ consists of: 
\begin{enumerate}
\item a triangle formation consisting of \label{classes:positive triangle}
\[
\Ext_{\cE^{+}_{\star}(1)}\left(\mM_{2}, \mM_{2}\right) \left\{y_{0}, \ldots, y_{m-k-1} \right\}
\]
with relations $v_{1}y_{i} = v_{0}y_{i+1},$ $v_{0}y_{0} = 0,$ and $v_{1}y_{m-k-1} = 0$ and 
\[
\Ext_{\cE^{+}_{\star}(1)}\left(\mM_{2}, \mM_{2}\right) \left\{ \tau^{2}y_{0}, \ldots, \tau^{2}y_{m-k-1} \right\}
\]
with relations $v_{1} \tau^{2}y_{i} = v_{0}\tau^{2}y_{i+1},$ $v_{0}^{2}\tau^{2}y_{0} = 0,$ and $v_{1}\tau^2 y_{m-k-1} = 0$ where \newline $|y_{i}| = \big(-2(k-m - i)-1, 0, -(k-m - i)\big),$ 

     \item infinite $\rho$-towers generated in odd stem

      \item a copy of $v_{1}\Ext_{\cE^{+}_{\star}(1)}(\mM_{2}, \mM_{2})$, with generator denoted $x$ and $|x| = (0, k-m, 0),$\label{component 3 of pos cone}

    \item $\rho$-pairs: \[\mF_{2}[\rho, v_{1}] \left\{ v_1 b \, | \, v_{1}^{m-k-1}b = \rho x, \, \rho^{2}b = 0\right\}, \] where $|b| = (2(m-k)-1, \, 0, \, m-k-1)$. 
    \end{enumerate}

    The negative cone $\Ext_{\cE^{-}_\star(1)}^{*,*,*}(L(k), L(m))$ when $k > m$ consists of:

\begin{enumerate} 
\item a triangle formation consisting of \label{classes:negative triangle}
\[ 
\Ext_{\cE^{+}_{\star}(1)}\left(\mM_{2}, \mM_{2}\right) \left\{ \frac{\gamma}{\tau^{4j + 3}}y_0, \cdots, \frac{\gamma}{\tau^{4j + 3}}y_{m-k-1} \right\} 
\]
with relations $v_{1}\frac{\gamma}{\tau^{4j + 3}}y_{i} = v_{0}\frac{\gamma}{\tau^{4j + 3}}y_{i+1},$ $v_{1}\frac{\gamma}{\tau^{4j + 3}}y_{m-k-1} = 0,$ and $v_{0}\frac{\gamma}{\tau^{4j + 3}}y_{0} = 0,$ and 
\[
\Ext_{\cE^{+}_{\star}(1)}\left(\mM_{2}, \mM_{2}\right) \left\{ \frac{\gamma}{\tau^{4j + 1}}y_0, \cdots, \frac{\gamma}{\tau^{4j + 1}}y_{m-k-1} \right\}
\]
with relations $v_{1}\frac{\gamma}{\tau^{4j + 1}} = v_{0}\frac{\gamma}{\tau^{4j + 1}}y_{i+1},$ $v_{1}\frac{\gamma}{\tau^{4j + 1}}y_{m-k-1} = 0,$ and $v_{0}\frac{\gamma}{\tau^{4j + 1}}y_{0} = 0$ where $|y_{i}| = \big(-2(k-m - i)-1, 0, -(k-m - i)\big),$ 
\item infinite $\rho$-divisible towers in filtration $f = 0$, as well as copies of $\mF_{2}[\tau^{2}]/\tau^{\infty}$ in odd stem,
\item a copy of the $\Ext_{\cE^{+}_{\star}(1)}(\mM_{2}, \mM_{2})$-submodule of $\Ext_{\cE^{-}_{\star}(1)}(\mM_{2}, \mM_{2})$ generated by \newline $\mF_{2}\frac{[\tau^{4}]}{\tau^{\infty}} \left\{ \frac{\gamma}{\rho^{2}\tau^{2}} \right\}$, with generator denoted by $x'$ and shifted so that $|x'| = (0, k-m-1, 4),$ 
\item $\rho$-pairs, with generator denoted by $c$ and $|c| = (2(m-k) +1, 0, m-k + 1)$ for $m-k > 1$
\[
\frac{\mF_2[\tau^4]}{\tau^\infty}[\rho, v_{1}]\{c | \rho^{2}c = 0,\ v_1^nc = x'\}.
\]
\end{enumerate}
\end{lemma}

Before proving \cref{lem:k>m}, we introduce some charts illustrating $\Ext_{\cE(1)_\star} (L(k), L(m))$ in the cases where $k = 1$ and $m = 0$ (\cref{fig:ExtL1M2}) and where $k =2$ and $m = 0$ (\cref{fig:ExtL2M2}). In both figures, one can see the triangle formation begin to emerge in stems $s \leq -3$. The shape of the triangle formation becomes more apparent as the difference between $k$ and $m$ grows. 

Examples of infinite $\rho$-towers can be found starting in filtration $f = 0$ and stem $s = -3$ in the positive cone summands. As multiplication by $\rho$ is represented by a horizontal line to the left, these $\rho$ towers appear as arrows on the zero line pointing to the left. Examples of infinite $\rho$-divisible towers can be found starting in filtration $f = 0$ and stem $s = -3$ in the negative cone summands. As divisibility by $\rho$ is represented by a horizontal line to the right, these $\rho$-divisible towers appear as arrows on the zero line pointing to the right. 

To identify the copies of $v_{1}\Ext_{\cE^{+}_{\star}(1)}(\mM_{2}, \mM_{2})$ in the positive cone summands of Figures \ref{fig:ExtL1M2} and \ref{fig:ExtL2M2}, it may be helpful to compare with the chart for $\Ext_{\cE^{+}_{\star}(1)}(\mM_{2}, \mM_{2})$ in \cref{fig:ExtM2M2}. Specifically, the generator of $v_{1}\Ext_{\cE^{+}_{\star}(1)}(\mM_{2}, \mM_{2})$ in \cref{fig:ExtL1M2} can be found in stem $s = 0$ and filtration $f = 1.$ Similarly, the generator of $v_{1}\Ext_{\cE^{+}_{\star}(1)}(\mM_{2}, \mM_{2})$ in \cref{fig:ExtL2M2} can be found in stem $s = 0$ and filtration $f = 2$. 

To identify the copies of the $\Ext_{\cE^{+}_{\star}(1)}(\mM_{2}, \mM_{2})$-submodule of $\Ext_{\cE^{+}_{\star}(1)}(\mM_{2}, \mM_{2})$ generated by $\mF_{2}\frac{[\tau^{4}]}{\tau^{\infty}} \left\{ \frac{\gamma}{\rho^{2}\tau^{2}} \right\}$, with generator denoted by $x'$ and shifted so that $|x'| = (0, k-m-1, 4)$ in the negative cone summands of Figures \ref{fig:ExtL1M2} and \ref{fig:ExtL2M2}, it may be helpful to compare with the chart for $\Ext_{\cE^{-}_{\star}(1)}(\mM_{2}, \mM_{2})$ in \cref{fig:ExtM2M2}. Specifically, the generator of the $\Ext_{\cE^{+}_{\star}(1)}(\mM_{2}, \mM_{2})$-submodule in \cref{fig:ExtL1M2} can be found in stem $s = 0$ and filtration $f = 0$. Similarly, the generator of the $\Ext_{\cE^{+}_{\star}(1)}(\mM_{2}, \mM_{2})$-submodule in \cref{fig:ExtL2M2} can be found in stem $s = 0$ and filtration $f = 1.$ 

These submodules generated by $x'$ in the negative cone summands appear very similar to $v_{1}\Ext_{\cE^{+}_{\star}(1)}(\mM_{2}, \mM_{2})$ in the positive cone summands (compare with \cref{fig:ExtM2M2}). The only differences are that the negative cone summands have infinitely-$\tau^{4}$-divisible towers rather than $\tau^{4}$-towers, and that the negative cone summand is shifted down by $1$ in filtration $f$, and shifted down by $\rho$ in stem in comparison with the positive cone summand. 

The $\rho$-pairs only appear when $k - m \geq 2.$ As such, $\rho$-pairs can be seen in  \cref{fig:ExtL2M2}. Specifically, there is a $\rho$-pair in the positive cone summand consisting of two blue dots in filtration $f = 1$ and stems $s = -3$ and $s = -4$ linked by a $\rho$-multiplication. Another $\rho$-pair can be seen in the negative cone summand in filtration $f=0$ and stems $s = -3$ and $s = -4$. Specifically, the $\rho$-pair is the two green dots in filtration $f = 1$ and stems $s = -4,$ and $s = -3$ linked by divisibility by $\rho.$

While proceeding through the inductive proof of \cref{lem:k>m}, it is helpful to keep similar charts of the $\Ext$-groups $\Ext_{\cE(1)_\star} (L(k), L(m))$ in mind.

\begin{center}
\begin{figure}[ht]
\includegraphics[width=5in,height=4in]{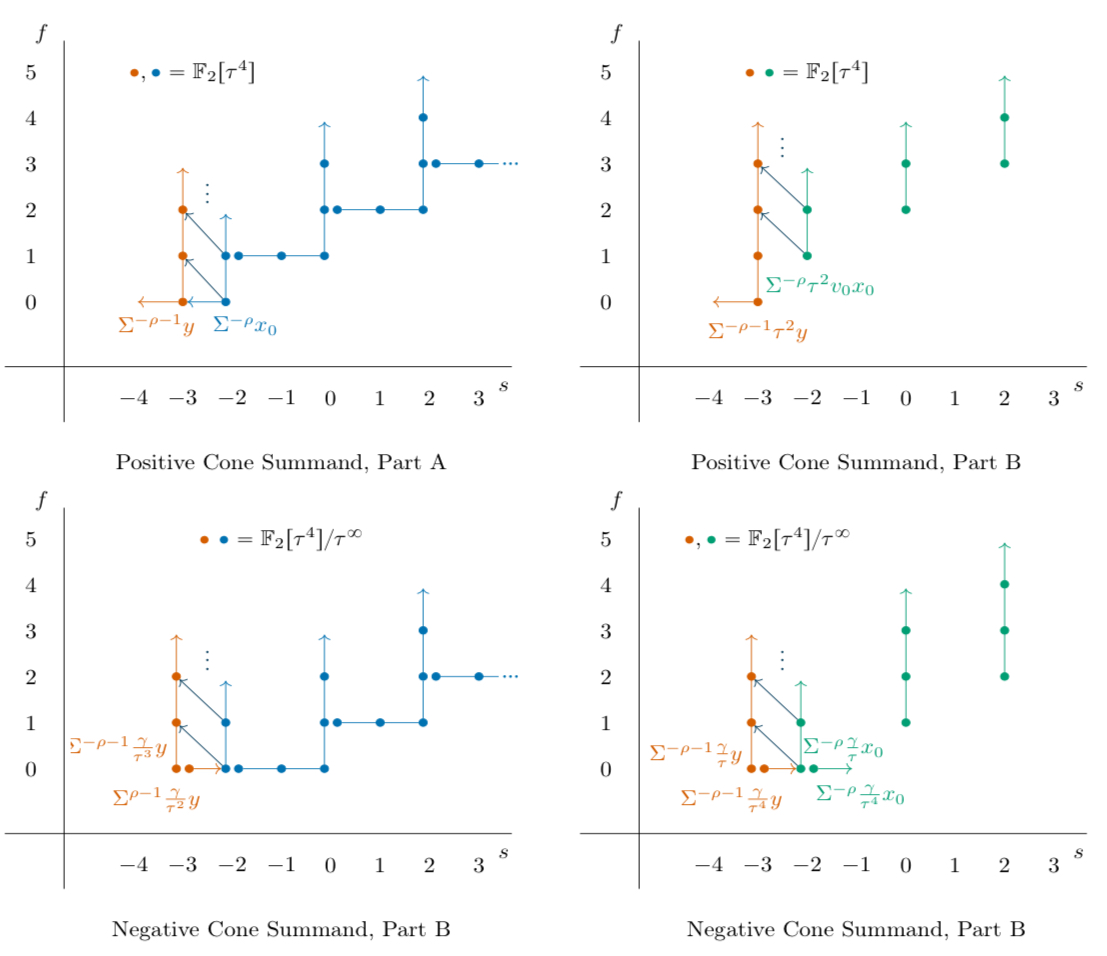}
\caption{$E_1$-page computing $\Ext_{E(1)_\star} (L(1), \mM_2)$}\label{fig:ExtL1M2-E1}
\end{figure}
\end{center}

\begin{center}
\begin{figure}[ht]
\includegraphics[width=5in,height=4in]{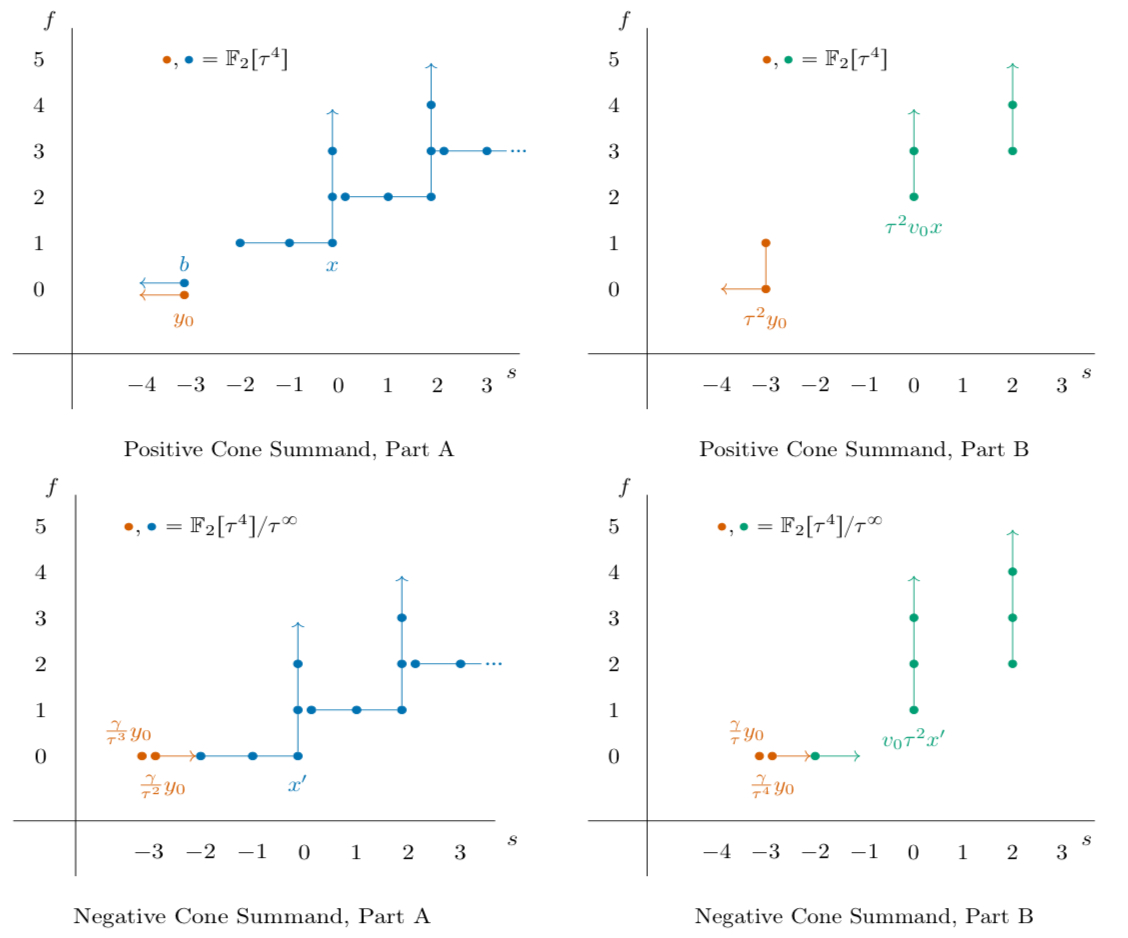}
\caption{$\Ext_{E(1)_\star} (L(1), \mM_2)$} \label{fig:ExtL1M2}
\end{figure}
\end{center}

\begin{figure}[ht]
    \centering
\includegraphics[width=5in,height=4in]{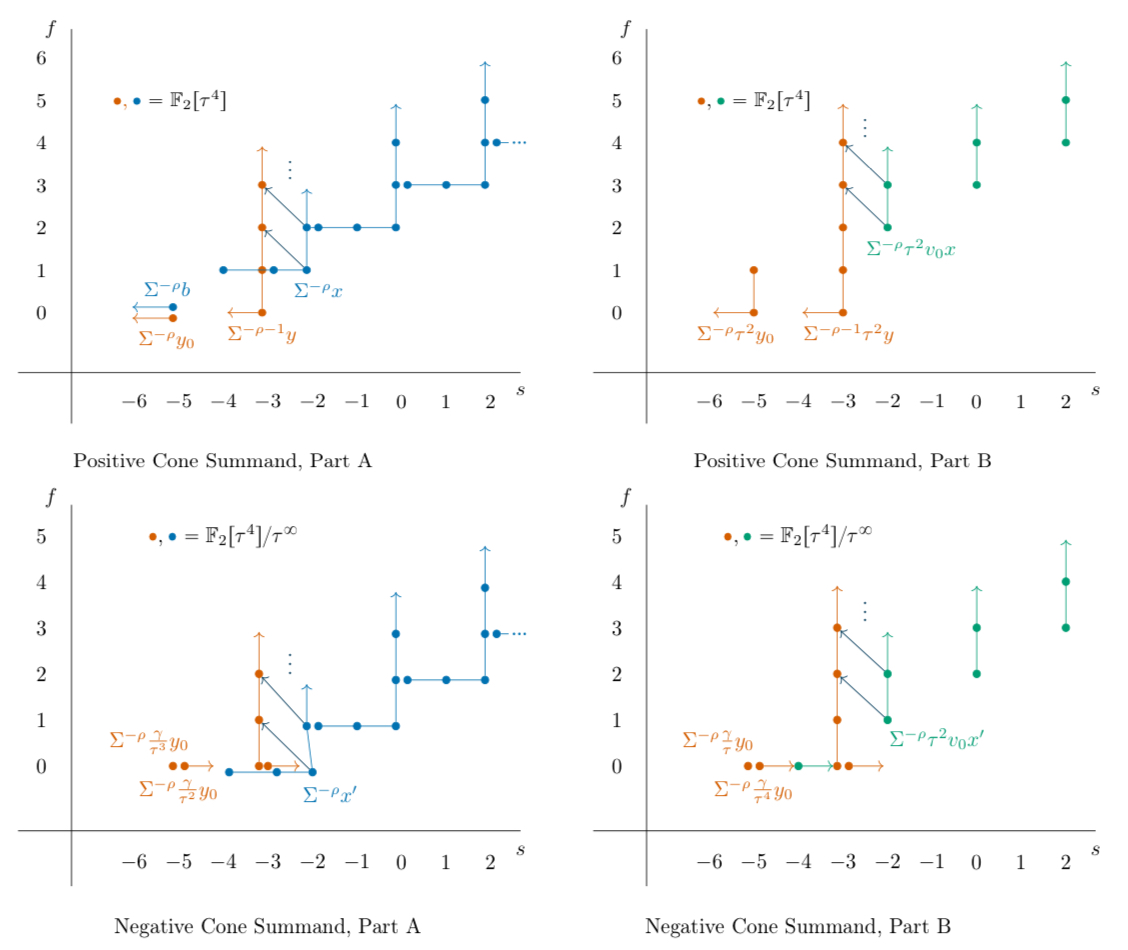}
\caption{$E_1$-page computing $\Ext_{E(1)_\star}(L(2), \mathbb{M}_2)$} \label{fig:E1ExtL2M2}
\end{figure}

\begin{proof}[Proof of \cref{lem:k>m}]
We will compute $\Ext_{\cE(1)_\star}(L(k),L(m))$ for $k > m$ by fixing $m$ and inducting on $k$. We begin with our base case, $\Ext_{\cE(1)_\star}(L(m+1),L(m))$. The reader is encouraged to refer to \cref{fig:ExtL1M2-E1}, which displays the computation of the base case when $m=0$, that is, the computation of $\Ext_{\cE(1)_{\star}}\left(L(1), \mM_{2}\right)$. 

We use the spectral sequence
\begin{align*}
    \Sigma^{-\rho - 1}\Ext_{\cE(0)_{\star}}^{s,f,w}( \mM_{2}, L(m) ) \oplus \Ext_{\cE(1)_{\star}}^{s,f,w}(\Sigma^{\rho}L(m) , L(m))
    \implies  \Ext_{\cE(1)_{\star}}(L(m+1), L(m))
\end{align*}
given in (\ref{eqn: spec seq for k > m}). 

Recall also that
\begin{align*}
\Ext_{\cE(0)_{\star}}\left(\mM_{2}, L(m) \right) \cong \Ext_{\cE(0)_{\star}}\left(\mM_{2}, \mM_{2} \right) \{y \}\oplus W,\end{align*}
where $|y| = (0,0,0,)$ and $W$ is a sum of odd-stem suspensions of $\mM_{2}$. We will not depict the copies of $W$ in our charts, as they are just infinite $\rho$-towers in odd stem and infinitely divisible $\rho$-towers. 

First, observe that the only possible differentials are of the form
\begin{equation} \label{eq:diffPC}
\begin{split}
    & d (\Sigma^{-\rho} v_{0}^{i} \tau^{4j}x_{0} ) = \Sigma^{-\rho - 1} v_{0}^{i+1}\tau^{4i}y, \qquad \qquad i,j \geq 0 \\
    & d(\Sigma^{-\rho} v_0^{i} \tau^{4j - 2}x_{0}) = \Sigma^{-\rho - 1} v_0^{i+1} \tau^{4j - 2}y \qquad \, \, i \ge 1, \, j \geq 1
\end{split}
\end{equation}
in the positive cone summand and  
\begin{align} \label{eq:diffNC}
\begin{split}
    & d \left(\Sigma^{-\rho} v_{0}^{i} \frac{\gamma}{\tau^{3 + 4j}}x_{0} \right) = \Sigma^{-\rho -1}v_{0}^{i+1} \frac{\gamma}{\tau^{3 + 4j}}y \qquad \qquad i\ge 0, \, j \geq 1\\
    & d \left( \Sigma^{-\rho} v_{0}^{i} \frac{\gamma}{\tau^{1 + 4j}}x_{0} \right) = \Sigma^{-\rho - 1} v_{0}^{i+1}\frac{\gamma}{\tau^{1 + 4j}}y \qquad \qquad i \ge 1, \, j \geq 1 
\end{split}
\end{align} 
in the negative cone summand.

Due to the $\Ext_{\cE^{+}_{\star}(1)}(\mM_{2}, \mM_{2})$-module structure, it will suffice to determine $d(\Sigma^{-\rho}x_{0} )$ and $d(\Sigma^{-\rho}\frac{\gamma}{\tau^{3 + 4j}}x_{0} )$.

Comparison with our computation of $\Ext_{\cE(1)_{\star}}\left(L(m+1), L(m+1)\right)$ in \cref{m:0} implies these differentials must be  
\[
d(\Sigma^{-\rho}x_{0} )= \Sigma^{-\rho - 1} v_{0}y. 
\]
\[d\left(\Sigma^{-\rho} v_{0}^{i} \frac{\gamma}{\tau^{3 + 4j}}x_{0} \right) = \Sigma^{-\rho -1}v_{0}^{i+1} \frac{\gamma}{\tau^{3 + 4j}}y.\]
Specifically, suppose towards a contradiction that $d(\Sigma^{-\rho}x_{0} ) = 0$. Then we will have an infinite $v_{0}$-tower in stem $s=-3$ of $\Ext_{\cE(1)_{\star}}\left(L(k), L(m)\right)$ for $k=m+1$ and $k=m$. Combining this with the long exact sequence of (\ref{eqn:les of k<m}) would imply $\Ext_{\cE(1)_{\star}}\left(L(m+1), L(m+1)\right)$ must also have an infinite $v_{0}$-tower in odd stem. But we already have already proven in \cref{m:0} that no such tower exists.

\begin{figure}[ht]
    \centering
\includegraphics[width=5in,height=4in]{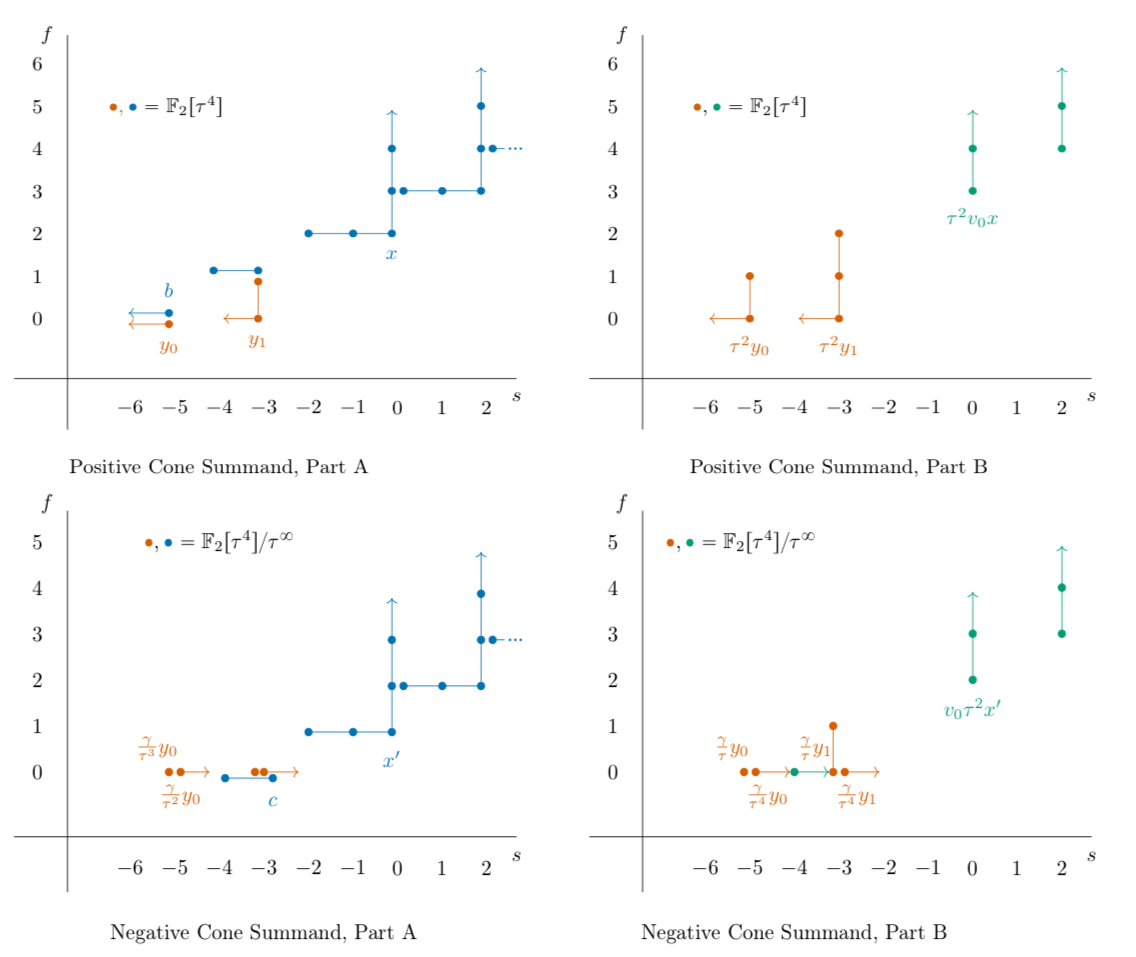}
\caption{$\Ext_{E(1)_\star}(L(2), \mathbb{M}_2)$} \label{fig:ExtL2M2}
\end{figure}

The argument showing 
\[
d \left(\Sigma^{-\rho} v_{0}^{i} \frac{\gamma}{\tau^{3 + 4j}}x_{0} \right) = \Sigma^{-\rho -1}v_{0}^{i+1} \frac{\gamma}{\tau^{3 + 4j}}y
\]
is similar. Hence by $\Ext_{\cE(1)_{\star}}(\mM_{2}, \mM_{2})$-linearity there are nontrivial differentials exactly as described in (\ref{eq:diffPC}) and (\ref{eq:diffNC}). Moreover, there is no room for $\Ext_{\cE(1)_{\star}}(\mM_{2}, \mM_{2})$-module extensions. 

Relabeling $\Sigma^{-\rho -1}y$ to $y_{0}$, $\Sigma^{-\rho}v_{1}$ to $x$, and $\Sigma^{-\rho}\rho$ to $b$, and $\frac{\gamma}{\rho^{2}\tau^{2}}x_{0}$ to $x'$ completes the computation of the base case $\Ext_{\cE(1)_{\star}}\left(L(m+1), L(m)\right)$. 

Now we are ready to proceed with the induction. To aid in understand the general argument, we suggest working through the computation of $\Ext_{\cE(1)_\star}(L(2), \mM_2)$ as illustrated in \cref{fig:ExtL2M2}. We also find it illuminating to refer to \cref{fig:ExtL2M2} while proceeding with the general induction. The computation is very similar to the base case, so we provide fewer details. 

Suppose the description claimed in the lemma holds for all $k'$ where $m+1 \le k' < k$. Consider the spectral sequence
\begin{align*}
    \Sigma^{-\rho - 1}\Ext_{\cE(0)_{\star}}^{s,f,w}( \mM_{2}, L(m) ) \oplus \Ext_{\cE(1)_{\star}}^{s,f,w}(\Sigma^{\rho}L(k-1) , L(m))
    \implies  \Ext_{\cE(1)_{\star}}(L(k), L(m))
\end{align*}
given in (\ref{eqn: spec seq for k > m}).

Using the $\Ext_{\cE(1)_{\star}}(\mM_{2}, \mM_{2})$-module structure, all potential nontrivial differentials are determined by $d(\Sigma^{-\rho}x)$ and $d\left(\Sigma^{-\rho}x'\right)$. For degree reasons, if nontrivial,
\begin{align} \label{eq:diffInduction}
    \begin{split}
        & d(\Sigma^{-\rho}x)  = v_{0}^{k-m}\Sigma^{-\rho - 1}y \\
        & d\left(x' \right) = \Sigma^{-\rho - 1}v_{0}^{k-m}\frac{\gamma}{\tau^3}y.
    \end{split}
\end{align} 
To see these differentials must indeed be nontrivial, suppose towards a contradiction that any one of these differentials is trivial. Then, just as in the computation of the base case, we will get an $v_{0}$-tower in odd stem in $\Ext_{\cE(1)_{\star}}(L(k), L(m))$. Induction on $m$ would then yield a $v_{0}$-tower in odd stem in $\Ext_{\cE(1)_{\star}}(L(m), L(m))$, which by \cref{m:0} cannot exist. Hence $d(\Sigma^{-\rho}x)$ and $d\left(\Sigma^{-\rho}x'\right)$ must be nontrivial as given in (\ref{eq:diffInduction}).

To finish the proof, relabel $\Sigma^{-\rho}v_{1}x$ as $x$, $\Sigma^{-\rho}b$ as $b$, $\Sigma^{-\rho - 1}y$ as $y_{0}$, $\Sigma^{-\rho - 1}y_{i}$ as $y_{i+1}$, and $\Sigma^{\rho}c$ as $c$ (except in the case $k=m+2$, where $c$ is the relabling of $\Sigma^{\rho}\rho x'$). Finally, note that the extensions $v_{1}y_{i} = v_{0}y_{i+1}$ and $v_{1}\tau^{2}y_{i} = v_{0}\tau^{2}y_{i+1}$ must occur by comparison to the underlying \cref{nonequiv: Ext computation}. No other $\Ext_{\cE(1)_{\star}}(\mM_{2}, \mM_{2})$-extensions are possible for degree reasons.
\end{proof}

\subsection{Analyzing Adams differentials}
The goal of this section is to finish proving the Adams spectral sequence (\ref{eq:AdamsSs}) collapses. 
\begin{theorem}\label{prop: spec seq collapse}
    For all $k \ge 0$, the Adams spectral sequence
   \begin{align*}
E_2^{s, f, w}  = Ext_{\mathcal{E}(1)_\star} \left(H_\star \Sigma^{\rho k} \cB_0(k), H_\star ku_\mR\right) 
      \implies \left[ku_\mR \wedge \Sigma^{\rho k} \cB_{0} (k), ku_\mR \wedge ku_\mR\right]^{ku_\mR}\notag
\end{align*}
collapses.
\end{theorem}

In the proof, we will rely on differentials in the relative Adams spectral sequence being $v_0,$ $v_1,$ $\tau^4,$ and $v_0 \tau^2$-linear so we provide some discussion of this fact here.  

Let $E_{r}(X,Y)$ denote the $r^{th}$-page of the relative Adams spectral sequence so 
\[E_{2}(X,Y) = \Ext_{\cE(1)_{\star}}(H_{\star}^{ku_{\mR}}X, H_{\star}^{ku_{\mR}}Y) \Longrightarrow [X,Y]^{ku_{\mR}} .\] 
With this notation, 
\begin{align*}
    E_{2}(ku_{\mR},ku_{\mR}) & = \Ext_{\cE(1)_{\star}}(\mM_{2}, \mM_{2}) \Longrightarrow \pi_{\star}ku_{\mR}, \\
    E_{2}(C_{k},C) = &\bigoplus\limits_{k=0}^{\infty}\Ext_{\cE(1)_{\star}}(L(\nu_{2}(k!), L(m))  \Longrightarrow [C_{k}, C]^{ku_{\mR}}.
\end{align*}
There is a pairing of relative Adams spectral sequences 
\[\lambda: E_{r}(ku_{\mR}, ku_{\mR}) \otimes E_{r}(C_{k}, C) \longrightarrow E_{r}(C_{k}, C) \]
induced by $ku_{\mR} \wedge C_{k} \to C_{k}$ and $ku_{\mR} \wedge C \to C$, such that 
\[d_{r}\left(\lambda(x \otimes y)\right) = \lambda\left(d_{r}(x)\otimes y\right) + \lambda\left(x\otimes d_{r}(y)\right).\]
Furthermore, the pairing on the $E_{2}$-page is the same as the natural $\Ext_{\cE(1)_{\star}}(\mM_{2}, \mM_{2})$-module 
\[\Ext_{\cE(1)_{\star}}(\mM_{2}, \mM_{2}) \otimes \Ext_{\cE(1)_{\star}}(L(k), L(m)) \longrightarrow \Ext_{\cE(1)_{\star}}(L(k), L(m))  \]
structure used above.

Note that the Adams spectral sequence \[ E_{2}(ku_{\mR},ku_{\mR}) = \Ext_{\cE(1)_{\star}}(\mM_{2}, \mM_{2}) \Longrightarrow \pi_{\star}ku_{\mR}  \]
collapses at the $E_{2}$-page, so for any class $v_{0}^{i}v_{1}^{j}\tau^{2j} \in \Ext_{\cE^\mR_\star(1)_{\star}}(\mM_{2}, \mM_{2})$ and $y \in E_{r}(C_{k}, C)$, we can simply write $d_{r}(v_{0}^{i}v_{1}^{j}\tau^{2j}y) = v_{0}^{i}v_{1}^{j}\tau^{2j}d_{r}(y)$. Therefore, all differentials in $E_{r}(C_{k}, C)$ are linear over $\Ext_{\cE(1)_{\star}}(\mM_{2}, \mM_{2})$, and in particular, all differentials are $v_{0}$, $v_{1}$, $\rho$, $\tau^{4}$, and $v_{0}\tau^{2}$-linear.

\begin{proof}[Proof of \cref{prop: spec seq collapse}]
In \cref{prop:splitInto4}, we showed that the Adams spectral sequence splits into a sum of four different Adams spectral sequences, and that all of the summands collapse at the $E_2$-page (and are concentrated on the $(f = 0)$-line) except  
\[ 
\Ext_{\mathcal{E}(1)_\star}^{s,f,w} (H_\star \Sigma^{\rho k}C_{k}, H_\star C) \Longrightarrow [\Sigma^{\rho k}C_{k}, C] .
\]
We also observed
\[ 
\Ext_{\mathcal{E}(1)_\star}^{s,f,w} (H_\star\Sigma^{\rho k} C_{k}, H_\star C) \cong  \bigoplus\limits_{m=0}^\infty \Sigma^{\rho (m-k)} \Ext_{\mathcal{E}(1)_\star} (L (\nu_2(k!)), L(\nu_2(m!))) .
\] 
To show the Adams spectral sequence collapses, we will begin by considering Adams differentials with source a $v_{1}$-torsion class.

\subsubsection{Differentials originating in $v_{1}$-torsion classes} The $v_{1}$-torsion classes in \[
\bigoplus\limits_{m=0}^\infty \Sigma^{\rho (m-k)} \Ext_{\mathcal{E}(1)_\star} (L (\nu_2(k!)), L(\nu_2(m!)))
\]
consist of: 
\begin{enumerate}
    \item $v_0$-towers in the triangle formation (\cref{lem:k>m} (1)),
    \item $\rho$-towers based in odd stem and filtration $f = 0$ (\cref{lem:k>m} (2)),
    \item $\rho$-divisible towers in filtration $f=0$ (\cref{lem:k>m} (2)). 
\end{enumerate} 

By $v_1$-linearity, the only potential targets of differentials with source a $v_1$-torsion class are other $v_1$-torsion classes. All $v_1$-torsion classes in filtration $f > 0$ are contained in the $v_0$-towers in the triangle formation and are $\rho$-torsion. Hence the $\rho$ and $\rho$-divisible towers cannot support differentials. Further, the $v_0$-towers in the triangle formation are all in odd stem, so they cannot support differentials as Adams differentials decrease stem by one. 

Before analyzing the possible differentials originating in $v_{1}$-torsion free classes, it will be helpful to list the generators of $Ext_{\cE(1)_{\star}}(L(k), L(m))$ as a module over the positive cone summand $\Ext_{\cE^{+}_\star}(1)((L(k), L(m))$. 
 
\subsubsection{Listing Generators.} Since differentials with source a $v_1$-torsion free class can have target a $v_1$-torsion class, we consider both $v_1$-torsion and $v_1$-torsion free classes. However, we omit $\rho$ and $\rho$-divisible towers from the remaining analysis since, as they are concentrated in filtration $f=0,$ they cannot be the target of any differentials.

\cref{tab:Extgens k<m pos} lists the generators of the positive cone summand of $\Ext_{\cE (1)_\star} (L(k), L(m))$ as a module over the positive cone summand $\Ext_{\cE^{+}_{\star}(1)}(1)((L(k), L(m))$ when $k \le m$. \cref{tab:Extgens k<m neg} lists the generators of the negative cone summand. It may be helpful to compare these tables with the chart for $\Ext_{\cE (1)_\star} (L(1), L(2))$ in \cref{fig:E1ExtL1L2}). 

\cref{tab:Extgens k>m pos} lists the generators of the positive cone summand of $\Ext_{\cE (1)_\star} (L(k), L(m))$ as a module over the positive cone summand $\Ext_{\cE^{-}_{\star}(1)}(1)((L(k), L(m))$ when $k > m.$ \cref{tab:Extgens k>m neg} lists the generators of the negative cone summand. In both tables, the $v_{1}$-torsion free generators are listed above the horizontal line, and $v_{1}$-torsion generators are below. It may also be helpful to compare these tables with the chart given for $\Ext_{\cE (1)_\star} (L(2), \mM_2)$ in \cref{fig:ExtL2M2}).

\begin{table}[h!] 
    \centering
    \begin{tabular}{l |l |l}
            Generators & $(s,\, f,\, w)$ & \\
            \hline
            $x_{i}$ & $(2i, 0, i)$ & $0 \le i \le m-k$ \\
            $\tau^{2}x_{i},$ & $(2i, 0, i-2)$ & $0 \le i < m-k$ \\
    \end{tabular}
    \caption{Generators of $\Ext_{\cE^{+}_{\star}(1)} (L(k), L(m))$ for $k \le m$, omitting $\rho$ and $\rho$-divisible towers}
    \label{tab:Extgens k<m pos}
\end{table}

\begin{table}[h!] 
    \centering
    \begin{tabular}{l |l |l}
            Generators & $(s,\, f,\, w)$ & \\
            \hline
            $\frac{\gamma}{\tau^{2j+1}}x_{i}$ & $(2i, 0, i + 2j + 2)$ & $0 \le i \le m-k$, $j \ge 0$ \\
            $\frac{\gamma}{\rho^{2}\tau^{2j}}x_{m-k}$ & $(2(m-k+1), 0, m-k + 2j + 3)$ & $j \ge1 $ \\
    \end{tabular}
    \caption{Generators of $\Ext_{\cE^{-}_{\star}(1)} (L(k), L(m))$ for $k \le m$, omitting $\rho$ and $\rho$-divisible towers}
    \label{tab:Extgens k<m neg}
\end{table}

\begin{table}[h!] 
    \centering
    \begin{tabular}{l |l |l}
            Generators & $(s,\, f,\, w)$ & \\
            \hline
            $x$ & $(0, k-m, 0)$ &  \\
            $b$ & $(2(m-k)-1, 1, m-k-1)$ & \\
            \hline
            $y_{i}$ &{$\left(-2(k-m-i) -1, 0, -(k-m-i)\right)$} &{$0 \le i < k-m$} \\
            $\tau^{2}y_{i}$ & {$\left(-2(k-m-i) -1, 0, -(k-m-i)   - 2\right)$} & {$0 \le i < k-m$}\\
    \end{tabular}
    \caption{Generators of $\Ext_{\cE^{+}_{\star}(1)} (L(k), L(m))$ for $k > m$, omitting $\rho$-towers and $\rho$-divisible towers} 
    \label{tab:Extgens k>m pos}
\end{table}

\begin{table}[h!] 
    \centering
    \begin{tabular}{l |l |l}
            Generators & $(s,\, f,\, w)$ & \\
            \hline
            $x'$  & $(0, k-m-1, 4)$ & {}\\
            $\frac{\gamma}{\tau^{4j}}c$
            & $(2(m-k)+1, 0, m-k + 4j+ 1)$ 
            & {} \\
            \hline
            $\frac{\gamma}{\tau^{2j + 1}}y_{i}$ &{$\left(-2(k-m-i) -1, 0, -(k-m-i) + 2j + 2\right)$} &{$0 \le i < k-m$} \\
            
    \end{tabular}
    \caption{Generators of $\Ext_{\cE^{-}_{\star}(1)} (L(k), L(m))$ for $k > m$, omitting $\rho$ and $\rho$-divisible towers}
    \label{tab:Extgens k>m neg}
\end{table}

\subsubsection{Differentials from $v_{1}$-torsion free classes to $v_{1}$-torsion free classes.}
Since the Adams differential decreases stem by $1$ but preserves motivic weight, it will follow immediately from the following proposition that there are no differentials from $v_{1}$-torsion free classes to $v_{1}$-torsion free classes. 

\begin{proposition}
    Let $x$ be a $v_{1}$-torsion free class in $E_{2}^{s,f,w}$. If $x$ is in stem $2n$, then the motivic weight of $x$ is congruent to $n$ mod $2$. If $x$ is in stem $2n-1$, then the motivic weight of $x$ is congruent to $n-1$ mod $2$. 
\end{proposition}
\begin{proof}
    Observe from the tables that every $v_{1}$-torsion free generator in
    \[
    \Ext_{\cE(1)_{\star}}(L(\nu_{2}(k!)), L(\nu_{2}(m!))
    \]
    satisfies the statement. Furthermore, the statement is preserved under action by an element in $\Ext_{\cE^{\mR}_\star(1)}(\mM_{2}, \mM_{2})$ (see \cref{tab:ExtM2M2gens}). Finally, suspension by the $C_2$-regular representation also preserves the statement. So indeed every $v_{1}$-torsion free element of 
    \[
    E_{2}^{s,f,w} \cong \bigoplus\limits_{k=0}^{\infty}\Ext_{\cE(1)_{\star}}(L(\nu_{2}(k!)), L(\nu_{2}(m!))
    \]
    has a motivic weight fitting the desired description.
\end{proof}

\subsubsection{Differentials from $v_{1}$-torsion free classes to $v_{1}$-torsion classes}
We will now show there are no differentials with source a $v_1$-torsion free class and target a $v_{1}$-torsion class. The $v_1$-torsion classes in filtration $f > 0$ are precisely the classes in the $v_0$-towers of the triangle formation (\cref{lem:k>m} (1)) with $f >0$. In particular, these classes occur only in odd stem $s \leq -3$. Since 
\[
\bigoplus\limits_{k \le m}\Sigma^{\rho(m-k)}\Ext_{\cE(1)}( L(\nu_{2}(k!) ), L(\nu_{2}(m!) ) )
\]
is contained entirely in stem $s \ge 0$, it only remains to show that there are no differentials originating in $v_1$-torsion free classes of even stem in 
\[
\bigoplus\limits_{k > m}\Sigma^{\rho(m-k)}\Ext_{\cE(1)}( L(\nu_{2}(k!) ), L(\nu_{2}(m!) ) )
.\]

Fix $n \ge 0$. We will put bounds on the Adams filtration and motivic weight of generators in stem $-2n$. Then we will put bounds on the Adams filtration and motivic weight of potential $v_1$-torsion targets, which are in stem $-2n-1$. This will allow us to rule out the potential for any nonzero Adams differentials. 

The only even-stem $v_1$-torsion free generator in the positive cone summand of 
\[
\Sigma^{\rho(m-k)}\Ext_{\cE(1)}( L(\nu_{2}(k!) ), L(\nu_{2}(m!) ) )
\]
is $\Sigma^{\rho(m - k)} x$, which has stem $2(m-k)$ and filtration $\nu_{2}(k!) - \nu_{2}(m!)$ (see \cref{tab:Extgens k>m pos}). Similarly, the only even-stem $v_1$-torsion free generator in the negative cone summand is $\Sigma^{\rho(m - k)} x'$ (see \cref{tab:Extgens k>m neg}), which has stem $2(m-k)$ and filtration $\nu_{2}(k!) - \nu_{2}(m!) - 1$. So in \[
\bigoplus\limits_{k > m}\Sigma^{\rho(m-k)}\Ext_{\cE(1)}( L(\nu_{2}(k!) ), L(\nu_{2}(m'!) ) ),\]
there are exactly two $v_{1}$-torsion free generators in degree $-2n$: one positive cone generator $\Sigma^{-\rho n}x$, which has filtration $\nu_{2}(k!) - \nu_{2}(k-n)!$, and one negative cone generator $\Sigma^{-\rho n}x'$, which has filtration $\nu_{2}(k!) - \nu_{2}(k-n)!-1$.

One can check in the same manner that at stem $-2n-1$, any $v_{1}$-torsion element in the positive cone component of 
\[
\bigoplus\limits_{m>k}\Sigma^{(m-k)\rho}\Ext_{\cE(1)_{\star}}\left(L(\nu_{2}(k!)), L(\nu_{2}(m!)) \right)
\]
has Adams filtration at most $\nu_{2}(k!) - \nu_{2}\left((k-n+1)! \right)$. Meanwhile, the negative cone summands are shifted down by $1$ in filtration: any $v_{1}$-torsion element in the negative cone component has Adams filtration at most $\nu_{2}(k!) - \nu_{2}\left((k-n+1)! \right)-1$. 

\begin{comment}
\textcolor{blue}{Proving these bounds proceeds similarly to the nonequivariant argument given in \cref{nonequiv: collapse}. Specifically, observe first that in $ \Ext_{\cE(1)_{\star}}\left(L(\nu_{2}(k!), L(\nu_{2}(m!)\right)$, the highest filtration of any $v_{1}$-torsion free class in stem $-3-2i$ is $\nu_{2}(k!)-\nu_{2}(m!)-i$ in the positive cone component. (This is recorded in \cref{lem:k>m}, but it is helpful to compare with \cref{fig:clExtL3Fp}). So in $\Sigma^{\rho(k-m)}\Ext_{\cE(1)_{\star}}\left(L(\nu_{2}(k!), L(\nu_{2}(m!)\right)$, the highest filtration of any $v_{1}$-torsion free class in stem $-3-2i - 2(k-m)$ is $\nu_{2}(k!)-\nu_{2}(m!)-i$ in the positive cone component. So for any $0 \le i \le m < k$ such that 
\[-3 -2i -2(k-m) = -2n -1,\]
the highest filtration class will be $\nu_{2}(k!) - \nu_{2}\left((k-n+1+i)!\right) - i$ in the positive cone. So at any stem $-2n-1$, the highest filtration class will be in filtration $\nu_{2}(k!) - \nu_{2}((k-n+1)!)$ in the positive cone. The upper bound on the Adams filtration of $v_1$-torsion classes in the negative cone component follows from a similar argument. }
\end{comment}

These bounds on Adams filtration rule out the possibility of any remaining nontrivial differentials, except possibly from the negative cone summand to the positive cone summand. However, we also rule these out for degree reasons, specifically by considering motivic weight.

Any generator in stem $-2n$ of the negative cone summand has motivic weight at least $-n + 4$, while any $v_{1}$-torsion class in stem $-2n-1$ of the the positive cone summand with filtration $f > 0$ has  motivic weight at most $-n$. 

\begin{comment}
This can be seen as follows: $\Sigma^{\rho(m-k)}\frac{\gamma}{\tau^{4j}}x'$ has motivic weight $-n+4(j+1)$. So the minimum motivic weight of any negative cone component generator in stem $-2n$ is indeed $-n+4$. Now consider a class $\Sigma^{\rho(m-k)}v_{0}^{i}y_{j}$ in $\Sigma^{\rho(k-m)}\Ext_{\cE(1)_{\star}}\left(L(\nu_{2}(k!), L(\nu_{2}(m!)\right)$. The stem of this class is $2(\nu_{2}(m!)-\nu_{2}(k!) + j + m-k)-1$, and the motivic weight of this class is $\nu_{2}(m!)-\nu_{2}(k!) + j + m -k$. So if a class $\Sigma^{\rho(m-k)}v_{0}^{i}y_{j}$ has stem $-2n-1$, then the motivic weight that class is $-n$. Thus the motivic weight of $\Sigma^{\rho(m-k)}\tau^{2r}v_{0}^{i}y_{j}$ is at most $-n$. 
So indeed no differentials from $v_{1}$-torsion free to $v_{1}$-torsion classes are possible. 
\end{comment}
\end{proof}
Thus the class $\theta_{k}:H_{\star}\Sigma^{\rho k}\cB_{0}(k)\rightarrow H_{\star}\Sigma^{\rho k}ku_{\mR}$ survives the Adams spectral sequence, and we have proved our main theorem.

\begin{theorem} \label{thm:mainTheorem}
    Up to $2$-completion, there is a splitting of $ku_{\mR}$-modules
    \[ku_{\mR} \wedge ku_{\mR}  \simeq \overset{\infty}{\underset{k=0}{\bigvee}} ku_{\mR} \wedge \Sigma^{\rho k}\cB_{0}(k).\]
\end{theorem}

\subsection{Description in terms of $C_2$-equivariant Adams covers}\label{sec:adams covers}
We now give a description of the splitting  of \cref{thm:mainTheorem} in terms of $C_2$-equivariant Adams covers. Let $ku_{\mR}^{\langle n \rangle}$ denote the $n$-th Adams cover of $ku_{\mR}$, that is, the $n$-th term in a minimal Adams resolution of $ku_{\mR}$ over $H$. (An Adams resolution is said to be minimal if applying $H\mF_{2}^{\star}(-)$ induces a minimal resolution of $\cA$-modules. (see \cite[p.2-3]{lellmann1984operations} for discussion). 
\begin{proposition}
    Let $ku_{\mR}^{\langle n \rangle}$ denote the $n$-th Adams cover of $ku_{\mR}$. Then $H_{\star}^{ku_{\mR}}ku_{\mR}^{\langle n \rangle} \cong L(n)$.
\end{proposition}
\begin{proof}
    Consider a minimal Adams resolution of $ku_{\mR}$ over $H$. 
    \[
    \begin{tikzcd}
        ku_{\mR}^{\langle 0 \rangle }\arrow[d, "i_{0}"] & ku_{\mR}^{\langle 1 \rangle }\arrow[l]\arrow[d, "i_{1}"] & ku_{\mR}^{\langle 2 \rangle}\arrow[d]\arrow[l] & \cdots \arrow[l]\\
        K_{0}\arrow[ur, dashed] & K_{1}\arrow[ur, dashed]  & K_{2}\arrow[ur, dashed] \\
    \end{tikzcd}
    \]

Recall that $ku_{\mR}^{\langle n \rangle}$ is defined inductively: $ku_{\mR}^{\langle 0 \rangle}:= ku_{\mR}$ and $ku_{\mR}^{\langle n \rangle}$ is the fiber of 
\[
ku_{\mR}^{\langle n-1 \rangle} \to K_{n-1}
\]
for all $n>0$. First, note that since $ku_{\mR}$ and $H$ are both $ku_{\mR}$-modules, applying $H_{\star}^{ku_{\mR}}(-)$ yields a minimal resolution of $\mM_{2}$ by $\cE(1)_{\star}$-modules. 

We will proceed by induction. Note that $K_{0} := H$, and $i_{0}: ku_{\mR} \to H$ is the unit map in $H^{*}ku_{\mR}$. 
Applying $H_{\star}^{ku_{\mR}}(-)$ to 
\[ ku_{\mR}^{\langle 1 \rangle} \to ku_{\mR} \to H\] yields a long exact sequence
\[\cdots \to H_{\star}^{ku_{\mR}}ku_{\mR}^{\langle 1 \rangle} \to \mM_{2} \to \cE(1)_{\star} \to \cdots.  \]
By inspection, $H_{\star}^{ku_{\mR}}ku_{\mR}^{\langle 1 \rangle} \cong L(1)$. 

Now suppose that $H_{\star}ku_{\mR}^{\langle k \rangle} \cong L(k)$ for all $k < n$. Then define $K_{n}:= \bigoplus\limits_{i=0}^{n}\Sigma^{i \rho + 1} H$. Recall that $L(n) \cong \cE(1)\{x_{1}, x_{2}, \ldots x_{n}|\ Q_{1}x_{i+1} = Q_{0}x_{i}\ 1 \le i < n \}$. Define a map 
\[i_{n_{\star}}: L(n) \to H_{\star}^{ku_{\mR}} K_{n}\] 
by sending $x_{i} \mapsto \Sigma^{i \rho + 1}1 \in \Sigma^{i \rho + 1}\cE(1)_{\star}$. This map can be realized by $i_{n}:ku_{\mR}^{\langle n \rangle} \to  K_{n}$.
\end{proof}

Note that $H_{\star}^{ku_{\mR}}ku_{\mR}^{\langle n \rangle} \cong L(\nu_{2}(n!))$, and that $ku_{\mR}^{\langle n \rangle}$ must be a $ku_{\mR}$-module by construction. We will make use of the uniqueness of these Adams covers as $k_{\mR}$-modules.
\begin{proposition}\label{prop:uniqueness of adams covers}
    Let $X$ be a $ku_{\mR}$-module such that $H_{\star}^{ku_{\mR}}X \cong L(n)$. Then 
    \[ X \simeq ku_{\mR}^{\langle n \rangle}.\]
\end{proposition}
\begin{proof}
    Consider the $ku_{\mR}$-relative Adams spectral sequence
\[E_2^{s,f,w} = \Ext_{\cE(1)_{\star}}(H_{\star}^{ku_{\mR}}X, H_{\star}^{ku_{\mR}}ku_{\mR}^{\langle n \rangle}) \cong \Ext_{\cE(1)_{\star}}\left(L(n), L(n)\right)\Longrightarrow [X, ku_{\mR}^{\langle n \rangle} ]^{ku_{\mR}}.\]
The proof of our main theorem (\cref{thm:mainTheorem}) shows that this spectral sequence collapses at the $E_{2}$-page, and so we can lift the isomorphism to an equivalence of $ku_{\mR}$-module spectra.
\end{proof}

Now we can prove the following theorem.

\begin{theorem}
    Up to $2$-completion,    \[ ku_{\mR} \wedge \cB_{0}(k) \simeq ku_{\mR}^{\langle \nu_{2}(k!) \rangle} \vee V_{k},\]
where $V_{k}$ is a sum of suspensions of $H$.
\end{theorem}

\begin{cor} \label{thm:C2AdamsCovers}
    Up to $2$-completion,    \[ ku_{\mR} \wedge ku_{\mR} \simeq \bigvee\limits_{k=0}^{\infty} \Sigma^{\rho k}ku_{\mR}^{\langle \nu_{2}(k!) \rangle} \vee V,\]
where $V$ is a sum of suspensions of $H$.
\end{cor}

\begin{proof}
\cref{prop: other splitting} shows $ku_{\mR} \wedge \cB_{0}(k) \cong C_{k} \vee V_{k}$, where $V_k$ is a sum of suspensions of $H$ and $C_k$ contains no $H$-summands. Specifically, $H_{\star}^{ku_{\mR}}C_{k} \cong L(\nu_{2}(n!))$. Thus we can use the Adams spectral sequence
\begin{align*} \label{eq:AdamsSpecSeqC2Covers}
    \Ext_{\cE(1)}(H_{\star}^{ku_{\mR}}ku_{\mR}^{\langle n \rangle}, H_{\star}C_{k}) \Longrightarrow [ku_{\mR}^{\langle n \rangle}, C_{k}]^{ku_{\mR}}
\end{align*}
to lift the isomorphism to an equivalence of $ku_{\mR}$-module spectra. Since the $E_2$-page of this spectral sequence is simply a summand of the $E_{2}$-page of the Adams spectral sequence (\ref{eq:AdamsSs}) which appeared in the proof of \cref{thm:mainTheorem}, the same arguments imply this spectral sequence also collapses.
\end{proof}
    
\subsection{{$ku_\mR$-operations and cooperations}}

\subsubsection{The $ku_\mR$-cooperations algebra.} Another consequence of our proof of \cref{thm:mainTheorem} is a computation of the $ku_\mR$-cooperations algebra.

\begin{theorem} \label{thm:cooperationsAlg}
    The $ku_\mR$-cooperations algebra ${ku_\mR}_\star ku_\mR$ splits as
    \[
    {ku_\mR}_\star ku_\mR \cong \overset{\infty}{\underset{k=0}{\bigoplus}}  {ku_\mR}_{\star - k \rho} \cB_0(k), 
    \]
    where as a ${ku_\mR}_\star$-module 
    \begin{align*}
        {ku_\mR}_{\star}\cB_{0}(k)
        \cong \bigoplus\limits_{k=0}^{\infty} H \umZ_{\star}\{x_{0}, x_{1}, \ldots, x_{\nu_2(2k!)-1} \} \oplus {ku_\mathbb{R}}_{\star}\{x_{\nu_2(2k!)} \} \oplus V_{k},
    \end{align*}
    with extensions $v_1 x_{i - 1} = \rho x_i$, and where $| x_i | = \rho i$ and $V_k$ is a sum of suspensions of $H_\star$. 
\end{theorem}

While we do not specify the degree of the summands appearing in $V_k$ in the statement of \cref{thm:cooperationsAlg} due to the amount of bookkeeping that would be involved, our computational methods do allow one to precisely deduce them.

\begin{proof}
    Consider the $ku_{\mR}$-Adams spectral sequence
\[\Ext_{\cE(1)_{\star}}(\mM_{2}, H_{\star}\cB_{0}(k))\Longrightarrow {ku_\mR}_\star \cB_{0}(k). \]
This is a summand of the spectral sequence of \cref{eq:AdamsSs} so this spectral sequence collapses. Further, \cref{k:0} gives a complete description of the $E_{2}$-page. Observe there are no hidden extensions for degree reasons.  
\end{proof}

The methods we use to compute the $ku_\mR$-cooperations algebra also yield an analogous $\mR$-motivic calculation. Let $kgl$ be the lift of $ku_\mR$ along the Betti realization functor from the $\mR$-motivic stable homotopy category to the category of genuine $C_2$-spectra. Because there is no interaction between the negative cone and positive cone summands in our spectral sequence computation of ${ku_\mR}_\star ku_{\mR}$, we obtain the following $\mR$-motivic calculation by simply omitting the negative cone summands in our computation of ${ku_\mR}_\star ku_\mR.$

\begin{cor} \label{cor:kglcooperations}
    The $kgl$-cooperations algebra $kgl_{*,*} kgl$ is given by 
    \[
        kgl_{*,*} kgl \cong \bigoplus\limits_{k=0}^\infty BPGL \langle 0 \rangle_{*,*} \{ x_0, x_1, \cdots, x_{\nu_2 (2k!) - 1} \} \oplus kgl_{*, *} \{ x_{ \nu_2 (2k!)} \} \oplus W,
    \]
    with extensions $v_1 x_{i - 1} = \rho x_i,$ and where $\vert x_i \vert = (2i, i)$ and $W$ is a sum of suspensions of $\mM_2^\mR.$
\end{cor}

At this time, it is not know whether motivic analogues of integral Brown--Gitler spectra exist. The construction we use in the $C_{2}$-equivariant setting relies on a Thom spectrum model, which is not known to exist in the $\mR$-motivic setting, so our methods do not currently provide a motivic analogue of \cref{thm:mainTheorem}. However, we do have all the necessary ingredients to construct the following $\mR$-motivic analogue of \cref{thm:C2AdamsCovers}.

\begin{cor}
Let $H\mF_{2}^{\mR}$ denote the mod $2$ motivic Eilenberg Maclane spectrum over $\mR$. Then \[kgl \wedge kgl \simeq \bigvee\limits_{k=0}^{\infty} \Sigma^{2k,k}kgl^{\langle \nu_{2}(k!) \rangle} \vee V,\]
where $kgl^{\langle \nu_{2}(k!) \rangle}$ is the $\nu_{2}(k!)^{th}$-Adams cover in a minimal $H\mF_{2}^{\mR}$-resolution of $kgl$, and $V$ is a sum of suspensions of $H\mF_{2}^{\mR}$.
\end{cor}
\begin{proof}
    This is analogous to the equivariant proof, but we will modify our notation and strategy slightly as we do not have motivic integral Brown--Gitler spectra. Take $V$ to be the sum of suspensions of $H \mF_2^\mR$ realizing the positive cone portion of the summand $W$ in \cref{cor:kglcooperations}, and let 
    \[
    H_{*,*}^{kgl}:=\pi_{*,*}\left(H\mF_{2}^{\mR} \underset{kgl}{\wedge}-\right).
    \]
    
    On the level of homology, $H_{*,*}^{kgl}V$ splits off from $H_{*,*}^{kgl}\left(kgl \wedge kgl\right) \cong H_{*,*}kgl$. So just as in the proof of \cref{prop: other splitting}, we can use the $kgl$-relative $H\mF_{2}^{\mR}$-motivic Adams spectral sequences
    \[ E_{2}^{s,f,w} \cong \Ext_{\cE^{\mR}_\star (1)}\left(H_{*,*}^{kgl}V_{k}, H_{*,*}kgl \right) \Longrightarrow [V, kgl \wedge kgl ]^{kgl} \]
\[ E_{2}^{s,f,w} \cong \Ext_{\cE_\star^{\mR}(1)}\left(H_{*,*}kgl, H_{*,*}^{kgl}V_{k} \right) \Longrightarrow [ kgl \wedge kgl, V ]^{kgl} \]
to lift the homology splitting to 
\[kgl \wedge kgl \simeq C \vee V,\]
where $H_{*,*}C$ contains no $H \mF_2^\mR$-summands. By omitting the negative cone in our computation of the homology of $ku_{\mR}$, we get 
\[
H_{*,*}C \cong \Sigma^{2k,k}L(\nu_{2}(k!)),
\]
where $L(\nu_{2}(k!))$ is the $\mR$-motivic lightning flash module. Likewise, the motivic Adams covers are of the form \[H_{*,*}^{kgl}kgl^{\langle m \rangle} \cong L(m).\] Thus we can consider the $kgl$-relative Adams spectral sequence
 \[ E_{2}^{s,f,w} \cong \Ext_{\cE^{\mR}(1)}\left(\Sigma^{2k,k}L(\nu_{2}(k!)), H_{*,*}kgl \right) \Longrightarrow [\Sigma^{2k,k}kgl^{\langle \nu_{2}(k!) \rangle}, kgl \wedge kgl ]^{kgl} .\]

The collapse of this spectral sequence follows from omitting the negative cone summands in the proof of the main theorem (\cref{thm:mainTheorem}).
\end{proof}

\subsubsection{$ku_\mR$-Operations} Our computational methods further yield an inductive description of the operations algebra for $ku_\mR$, that is $[ku_\mR, ku_\mR]$.

\begin{theorem} \label{thm:operationsAlgebra}
     The operations algebra $[ku_\mR, ku_\mR]$ splits as
     \[ [ku_\mR, ku_\mR] \cong \bigoplus\limits_{k=0}^{\infty}[\Sigma^{\rho k}\cB_{0}(k), ku_\mR] .\]
     The Adams spectral sequence 
     \[
     \Ext_{\cE(1)_{\star}}(H_{\star}\cB_{0}(k), H_{\star}ku_\mR) \Longrightarrow [\Sigma^{\rho k}\cB_{0}(k), ku_\mR]
     \] 
     collapses at the $E_{2}$-page, and its $E_{2}$-page is described in \cref{lem:k>m}.
\end{theorem}

\begin{proof}
By \cref{thm:mainTheorem}, \[ [ku_{\mR}, ku_{\mR}] \cong \left[ku_\mR \wedge \bigvee\limits_{k=0}^{\infty} \Sigma^{\rho k}\cB_{0}(k), ku_\mR \right]^{ku_{\mR}} \cong \bigoplus\limits_{k=0}^{\infty}\left[ \Sigma^{\rho k}\cB_{0}(k), ku_\mR \right].\]

So we can use the Adams spectral sequence 
\[ \Ext_{A_{\star}}(H_{\star}\cB_{0}(k), H_{\star}ku_{\mR})  \cong \Ext_{\cE(1)_{\star}}(H_{\star}\cB_{0}(k), \mM_{2}) \Longrightarrow  [\cB_{0}(k), ku_\mR] \] to compute this summand.
Recall that \[\Ext_{\cE(1)_{\star}}(H_{\star}\cB_{0}\left( k \right), H_{\star}ku_\mR) \cong \Ext_{\cE(1)_{\star}}(\Sigma^{\rho k}L(\nu(k!)), H_{\star}ku_\mR) \oplus V_{k},\]
where $V_{k}$ is a sum of suspensions of $\mM_{2}$.
\cref{lem:k>m} gives a description of  \newline $\Ext_{\cE(1)_{\star}}(L(\nu(k!), \mM_{2})$ for all $k \geq 0$. By the same arguments as the proof of the main theorem (\cref{thm:mainTheorem}), no differentials are possible.
\end{proof}

Similarly to the $ku_\mR$-cooperations algebra calculation, there are also no interactions between negative and positive cone summands in our calculation of the $ku_\mR$-operations algebra. Thus these methods also yield an analogous computation of the $kgl$-operations algebra by simply omitting all negative cone summands. 

\subsubsection{The $E_{1}$-page of the $ku_{\mR}$-based Adams spectral sequence}\label{sec:E1kuR}
The purpose of this section is to demonstrate how the splitting of \cref{thm:mainTheorem} allows us to easily compute the $E_{1}$-page of the $ku_{\mR}$-based Adams spectral sequence for the sphere. We also give some information towards computing the $E_2$-page.

The $E_{1}$-page of the $ku_{\mR}$-Adams spectral sequence has the form
\[ E_{1}^{n,*,*}\cong \pi_{\star}\left( ku_{\mR}\wedge \overline{ku_{\mR}}^{\wedge n}\right) \Longrightarrow \pi_{\star}\mathbb{S}_{(2)} ,\]
where $\overline{ku_{\mR}}$ denotes the cofiber of the unit $\mathbb{S} \to ku_{\mR}$. 

It follows immediately from \cref{thm:mainTheorem} and \cref{thm:C2AdamsCovers} that 
\[
ku_{\mR} \wedge \overline{ku_{\mR}} \simeq \bigvee\limits_{k=1}^{\infty} ku_{\mR} \wedge \Sigma^{\rho k} \cB_{0}(k) \simeq \bigvee\limits_{k=1}^{\infty} \Sigma^{\rho k}ku_{\mR}^{\langle \nu_{2}(k!)   \rangle } \vee V,
\]
where $V$ is the sum of suspensions of $H$ found in \cref{thm:C2AdamsCovers}. We can iterate this splitting to describe the entire $E_{1}$-page of the $ku_{\mR}$-Adams spectral sequence in terms of finite $ku_{\mR}$-modules. In particular, the splitting implies
 $$ku_{\mR}\wedge \overline{ku_{\mR}}^{\wedge n} \simeq \bigvee\limits_{i_{1},i_{2}, \ldots, i_{n} \ge 1} ku_{\mR} \wedge \cB_{0}(i_{1}) \wedge \cB_{0}(i_{2}) \wedge \cdots \wedge \cB_{0}(i_{n}) .$$

In this section, we will give a formula for the homotopy groups of each summand 
\[
ku_{\mR} \wedge \cB_{0}(i_{1}) \wedge \cB_{0}(i_{2}) \wedge \cdots \wedge \cB_{0}(i_{n}),
\]
up to suspensions of $\mM_{2}$. Locating the suspensions of $\mM_{2}$ is simply a matter of extensive bookkeeping, so we reserve it for forthcoming work on the $ku_\mR$-Adams spectral sequence. 

\begin{proposition}
Let $I = (i_{1}, i_{2}, \ldots i_{n})$. Then \[ H_{\star}\left(\cB_{0}(i_{1}) \wedge \cB_{0}(i_{2}) \wedge \cdots \wedge \cB_{0}(i_{n}) \right) \cong L(I) \oplus F_{I},\]
where $L(I)$ is the lightning flash module $L(\nu_{2}(i_{1}!) + \nu_{2}(i_{2}!) + \cdots \nu_{2}(i_{n}!))$ and $F_{I}$ is a free $\cE(1)_{\star}$-comodule.
\end{proposition}

\begin{proof}
    In \cref{bg lightning split} we showed that $H_{\star}\cB_{0}(k) \cong L(\nu_{2}(k!)) \oplus F_{k}$, where $F_{k}$ is a sum of suspensions of $\cE(1)_{\star}$. In the proof of \cref{bg lightning split}, we observed that 
    \begin{align*}
        \cM_{*}\left(L(m) /(\rho,\tau), Q_{0}\right) & \cong \mF_{2}\{1\} \\
        \cM_{*}\left( L(m)/(\rho,\tau), Q_{1}\right) & \cong \mF_{2}\{x_{m} \} .
    \end{align*}

    By the K\"{u}nneth theorem for Margolis homology,
    \begin{align*}
         \cM_{*}  \left(L(m) /(\rho,\tau), Q_{0}\right) \otimes \cM_{*}\left( L(n)/(\rho,\tau), Q_{0}\right) & \cong \mF_{2}\{1\}, \\
         \cM_{*}  \left(L(m)/(\rho,\tau), Q_{1}\right) \otimes \cM_{*}\left(  L(\nu_{2}(n))/(\rho,\tau), Q_{1}\right) & \cong \mF_{2}\{x_{m} \otimes x_{n}\} \\
         & \cong \mF_{2}\{x_{m+n}\}.
    \end{align*}
Since  
    \begin{equation}\label{eqn:iffy}
        L(m)/(\rho,\tau) \otimes L(n)/(\rho,\tau) \cong \left( L (m) \otimes L(n)\right)/(\rho,\tau),
    \end{equation}
     we can apply the Whitehead Theorem for Margolis homology (\cref{prop:EquivWhitehead}) and conclude that $L(m) \otimes L(n) \cong L(m + n)$, up to free $\cE(1)_{\star}$ summands. Thus again by the Whitehead theorem, 
    \[
    H_{\star}\cB_{0}(i_{1}) \otimes H_{\star}\cB_{0}(i_{2}) \cong L\left((\nu_{2}(i_{1}!) + \nu_{2}(i_{2}!)\right),
    \]
    up to free $\cE(1)_{\star}$-summands. Iterating this finishes the proof.
\end{proof}

We also describe each summand in terms of an Adams cover. 

\begin{proposition}
    Let $I = (i_{1}, i_{2}, \ldots i_{n})$. Then \[ ku_{\mR}\wedge \cB_{0}(i_{1}) \wedge \cB_{0}(i_{2}) \wedge \cdots \wedge \cB_{0}(i_{n})  \simeq ku_{\mR}^{\langle |I| \rangle} \vee V_{I},\]
where $V_{I}$ is a sum of suspensions of $H$, and $|I| := \nu_{2}(i_{1}!) + \nu_{2}(i_{2}!) + \cdots + \nu_{2}(i_{n}!)$-Adams cover of $ku_{\mR}$. Furthermore, 
\[ \pi_{\star}\left( ku_{\mR}\wedge \cB_{0}(i_{1}) \wedge \cB_{0}(i_{2}) \wedge \cdots \wedge \cB_{0}(i_{n})  \right) \cong  \bigoplus\limits_{k=0}^{\infty} H \umZ_{\star}\{x_{0}, x_{1}, \ldots, x_{m-1} \} \oplus {ku_\mathbb{R}}_{\star}\{x_{m} \} \oplus W_{I},\]
where $m = \nu_{2}(i_{1}!) + \nu_{2}(i_{2}!) + \cdots + \nu_{2}(i_{n}!)$ and $W_{I}$ is a sum of suspensions of $\mM_{2}$.
\end{proposition}
\begin{proof}
First, it follows from the same argument as that of \cref{prop: other splitting} that we can split $ku_{\mR}\wedge \cB_{0}(i_{1}) \wedge \cB_{0}(i_{2}) \wedge \cdots \wedge \cB_{0}(i_{n})  \simeq C_{I} \vee V_{I}$, where $V_{I}$ is a sum of suspensions of $H$ and $H_{\star}^{ku_{\mR}}C_{I} \cong L(m)$. By \cref{prop:uniqueness of adams covers}, $C_{I} \simeq ku_{\mR}^{\langle |I| \rangle}$.  The computation of the homotopy groups is the same as that of \cref{thm:cooperationsAlg}.
\end{proof}

We expect these Adams cover descriptions will be useful for computing differentials to get to the $E_{2}$-page (see for example \cite{BeaudryBehrensBattacharyaCulverXu20}).

\printbibliography

@article{Wong2011,
    author = {Wong, B.},
    title = {Points of view: Color blindness},
    journal = {Nature Methods},
    year = {2011},
    DOI = {https://doi.org/10.1112/topo.12084},
}

@article {BOSS19,
    AUTHOR = {Behrens, M. and Ormsby, K. and Stapleton, N. and Stojanoska,
              V.},
     TITLE = {On the ring of cooperations for 2-primary connective
              topological modular forms},
   JOURNAL = {J. Topol.},
  FJOURNAL = {Journal of Topology},
    VOLUME = {12},
      YEAR = {2019},
    NUMBER = {2},
     PAGES = {577--657},
      ISSN = {1753-8416,1753-8424},
   MRCLASS = {55N34 (11F11 11F23 11F33)},
  MRNUMBER = {4072175},
MRREVIEWER = {Marian\ F.\ Anton},
       DOI = {10.1112/topo.12094},
       URL = {https://doi.org/10.1112/topo.12094},
}

@article {CDGM1988,
    AUTHOR = {Cohen, Fred R. and Davis, Donald M. and Goerss, Paul G. and
              Mahowald, Mark E.},
     TITLE = {Integral {B}rown-{G}itler spectra},
   JOURNAL = {Proc. Amer. Math. Soc.},
  FJOURNAL = {Proceedings of the American Mathematical Society},
    VOLUME = {103},
      YEAR = {1988},
    NUMBER = {4},
     PAGES = {1299--1304},
      ISSN = {0002-9939,1088-6826},
   MRCLASS = {55P42 (55P35)},
  MRNUMBER = {955026},
MRREVIEWER = {Richard\ John\ Steiner},
       DOI = {10.2307/2047129},
       URL = {https://doi.org/10.2307/2047129},
}

@article {cohen1979geometry,
    AUTHOR = {Cohen, Ralph L.},
     TITLE = {The geometry of {$\Omega \sp{2}S\sp{3}$} and braid
              orientations},
   JOURNAL = {Invent. Math.},
  FJOURNAL = {Inventiones Mathematicae},
    VOLUME = {54},
      YEAR = {1979},
    NUMBER = {1},
     PAGES = {53--67},
      ISSN = {0020-9910,1432-1297},
   MRCLASS = {55P35 (55N22)},
  MRNUMBER = {549545},
MRREVIEWER = {V.\ P.\ Snaith},
       DOI = {10.1007/BF01391176},
       URL = {https://doi.org/10.1007/BF01391176},
}

@article {Culverp219,
    AUTHOR = {Culver, Dominic Leon},
     TITLE = {On {$\rm BP\langle 2\rangle$}-cooperations},
   JOURNAL = {Algebr. Geom. Topol.},
  FJOURNAL = {Algebraic {\&} Geometric Topology},
    VOLUME = {19},
      YEAR = {2019},
    NUMBER = {2},
     PAGES = {807--862},
      ISSN = {1472-2747,1472-2739},
   MRCLASS = {55S10 (55T15)},
  MRNUMBER = {3924178},
MRREVIEWER = {Nguyen Sum},
       DOI = {10.2140/agt.2019.19.807},
       URL = {https://doi.org/10.2140/agt.2019.19.807},
}

@article {CulverOddp20,
    AUTHOR = {Culver, D.},
     TITLE = {The {$BP\langle2\rangle$}-cooperations algebra at odd primes},
   JOURNAL = {J. Pure Appl. Algebra},
  FJOURNAL = {Journal of Pure and Applied Algebra},
    VOLUME = {224},
      YEAR = {2020},
    NUMBER = {5},
     PAGES = {106239, 28},
      ISSN = {0022-4049,1873-1376},
   MRCLASS = {55T15 (55Q10 55S10 55S25)},
  MRNUMBER = {4046239},
MRREVIEWER = {Nguyen Sum},
       DOI = {10.1016/j.jpaa.2019.106239},
       URL = {https://doi.org/10.1016/j.jpaa.2019.106239},
}

@article {GuillouHillIsaksenRavenel2020,
    AUTHOR = {Guillou, Bertrand J. and Hill, Michael A. and Isaksen, Daniel
              C. and Ravenel, Douglas Conner},
     TITLE = {The cohomology of {$C_2$}-equivariant {$\mathcal{A}(1)$} and the
              homotopy of {${\rm ko}_{C_2}$}},
   JOURNAL = {Tunis. J. Math.},
  FJOURNAL = {Tunisian Journal of Mathematics},
    VOLUME = {2},
      YEAR = {2020},
    NUMBER = {3},
     PAGES = {567--632},
      ISSN = {2576-7658,2576-7666},
   MRCLASS = {14F42 (55Q91 55T15)},
  MRNUMBER = {4041284},
MRREVIEWER = {Bj\o rn\ Ian\ Dundas},
       DOI = {10.2140/tunis.2020.2.567},
       URL = {https://doi.org/10.2140/tunis.2020.2.567},
}

@article {Hill11,
    AUTHOR = {Hill, Michael A.},
     TITLE = {Ext and the motivic {S}teenrod algebra over {$\mathbb{R}$}},
   JOURNAL = {J. Pure Appl. Algebra},
  FJOURNAL = {Journal of Pure and Applied Algebra},
    VOLUME = {215},
      YEAR = {2011},
    NUMBER = {5},
     PAGES = {715--727},
      ISSN = {0022-4049,1873-1376},
   MRCLASS = {55T15 (14F42 19D45 55P43 55S10)},
  MRNUMBER = {2747214},
MRREVIEWER = {Kyle\ M.\ Ormsby},
       DOI = {10.1016/j.jpaa.2010.06.017},
       URL = {https://doi.org/10.1016/j.jpaa.2010.06.017},
}

@article {goerss1986some,
    AUTHOR = {Goerss, Paul G. and Jones, John D. S. and Mahowald, Mark E.},
     TITLE = {Some generalized {B}rown-{G}itler spectra},
   JOURNAL = {Trans. Amer. Math. Soc.},
  FJOURNAL = {Transactions of the American Mathematical Society},
    VOLUME = {294},
      YEAR = {1986},
    NUMBER = {1},
     PAGES = {113--132},
      ISSN = {0002-9947,1088-6850},
   MRCLASS = {55N20 (55S10 55T15)},
  MRNUMBER = {819938},
MRREVIEWER = {Edgar\ H.\ Brown, Jr.},
       DOI = {10.2307/2000121},
       URL = {https://doi.org/10.2307/2000121},
}

@article {HuKriz2001,
    AUTHOR = {Hu, Po and Kriz, Igor},
     TITLE = {Real-oriented homotopy theory and an analogue of the
              {A}dams-{N}ovikov spectral sequence},
   JOURNAL = {Topology},
  FJOURNAL = {Topology. An International Journal of Mathematics},
    VOLUME = {40},
      YEAR = {2001},
    NUMBER = {2},
     PAGES = {317--399},
      ISSN = {0040-9383},
   MRCLASS = {55T15 (19L47 55N22 55N91 55P42 55P91)},
  MRNUMBER = {1808224},
MRREVIEWER = {J.\ P. C. Greenlees},
       DOI = {10.1016/S0040-9383(99)00065-8},
       URL = {https://doi.org/10.1016/S0040-9383(99)00065-8},
}

@article {HellerOrmsby16,
    AUTHOR = {Heller, J. and Ormsby, K.},
     TITLE = {Galois equivariance and stable motivic homotopy theory},
   JOURNAL = {Trans. Amer. Math. Soc.},
  FJOURNAL = {Transactions of the American Mathematical Society},
    VOLUME = {368},
      YEAR = {2016},
    NUMBER = {11},
     PAGES = {8047--8077},
      ISSN = {0002-9947,1088-6850},
   MRCLASS = {14F42 (11E81 19E15 55P91)},
  MRNUMBER = {3546793},
MRREVIEWER = {Daniel\ C.\ Isaksen},
       DOI = {10.1090/tran6647},
       URL = {https://doi.org/10.1090/tran6647},
}

@article {HW2020,
	AUTHOR = {Hahn, Jeremy and Wilson, Dylan},
	TITLE = {Eilenberg--{M}ac {L}ane spectra as equivariant {T}hom spectra},
	JOURNAL = {Geom. Topol.},
	FJOURNAL = {Geometry {\&} Topology},
	VOLUME = {24},
	YEAR = {2020},
	NUMBER = {6},
	PAGES = {2709--2748},
	ISSN = {1465-3060},
	MRCLASS = {55P43 (55P91)},
	MRNUMBER = {4194302},
	MRREVIEWER = {Birgit Richter},
	DOI = {10.2140/gt.2020.24.2709},
	URL = {https://doi.org/10.2140/gt.2020.24.2709},
}

@article {kane1981operations,
    AUTHOR = {Kane, Richard M.},
     TITLE = {Operations in connective {$K$}-theory},
   JOURNAL = {Mem. Amer. Math. Soc.},
  FJOURNAL = {Memoirs of the American Mathematical Society},
    VOLUME = {34},
      YEAR = {1981},
    NUMBER = {254},
     PAGES = {vi+102},
      ISSN = {0065-9266,1947-6221},
   MRCLASS = {55S25 (55S10)},
  MRNUMBER = {634210},
MRREVIEWER = {Frederick\ Cohen},
       DOI = {10.1090/memo/0254},
       URL = {https://doi.org/10.1090/memo/0254},
}

@article {klippenstein1988brown,
    AUTHOR = {Klippenstein, John},
     TITLE = {Brown-{G}itler spectra at {${\rm BP}\langle 2\rangle$}},
   JOURNAL = {Topology Appl.},
  FJOURNAL = {Topology and its Applications},
    VOLUME = {29},
      YEAR = {1988},
    NUMBER = {1},
     PAGES = {79--91},
      ISSN = {0166-8641,1879-3207},
   MRCLASS = {55P42 (55N20 55S25)},
  MRNUMBER = {944072},
MRREVIEWER = {W.\ Stephen\ Wilson},
       DOI = {10.1016/0166-8641(88)90060-0},
       URL = {https://doi.org/10.1016/0166-8641(88)90060-0},
}

@article {Mahowald1977,
    AUTHOR = {Mahowald, Mark},
     TITLE = {A new infinite family in {${}\sb{2}\pi_{*}{}^s$}},
   JOURNAL = {Topology},
  FJOURNAL = {Topology. An International Journal of Mathematics},
    VOLUME = {16},
      YEAR = {1977},
    NUMBER = {3},
     PAGES = {249--256},
      ISSN = {0040-9383},
   MRCLASS = {55E45},
  MRNUMBER = {445498},
MRREVIEWER = {Donald\ W.\ Kahn},
       DOI = {10.1016/0040-9383(77)90005-2},
       URL = {https://doi.org/10.1016/0040-9383(77)90005-2},
}

@article {mahowald1981bo,
    AUTHOR = {Mahowald, Mark},
     TITLE = {{$b{\rm o}$}-resolutions},
   JOURNAL = {Pacific J. Math.},
  FJOURNAL = {Pacific Journal of Mathematics},
    VOLUME = {92},
      YEAR = {1981},
    NUMBER = {2},
     PAGES = {365--383},
      ISSN = {0030-8730,1945-5844},
   MRCLASS = {55Q45 (55P42)},
  MRNUMBER = {618072},
MRREVIEWER = {R.\ J.\ Milgram},
       URL = {http://projecteuclid.org/euclid.pjm/1102736799},
}

@article {BGL2022,
	AUTHOR = {Bhattacharya, Prasit and Guillou, Bertrand and Li, Ang},
	TITLE = {An {${\mathbb{R}}$}-motivic {$\upsilon_1$}-self-map of periodicity
		1},
	JOURNAL = {Homology Homotopy Appl.},
	FJOURNAL = {Homology, Homotopy and Applications},
	VOLUME = {24},
	YEAR = {2022},
	NUMBER = {1},
	PAGES = {299--324},
	ISSN = {1532-0073},
	MRCLASS = {14F42 (55Q51 55Q91)},
	MRNUMBER = {4410466},
	DOI = {10.4310/hha.2022.v24.n1.a15},
	URL = {https://doi.org/10.4310/hha.2022.v24.n1.a15},
}

@article{HHR2011,
  title={On the 3-primary Arf-Kervaire Invariant Problem},
  author={Hill, M. A. and Hopkins, M. J. and Ravenel, D. C.},
  journal={unpublished note},
  year={2011},
  URL = {https://www.sas.rochester.edu/mth/sites/doug-ravenel/mypapers/odd.pdf}
}

@article {HK2018,
	AUTHOR = {Heard, Drew and Krause, Achim},
	TITLE = {Vanishing lines for modules over the motivic {S}teenrod
		algebra},
	JOURNAL = {New York J. Math.},
	FJOURNAL = {New York Journal of Mathematics},
	VOLUME = {24},
	YEAR = {2018},
	PAGES = {183--199},
	MRCLASS = {55S10 (14F42)},
	MRNUMBER = {3778499},
	MRREVIEWER = {Piotr Kraso\'{n}},
	URL = {http://nyjm.albany.edu:8000/j/2018/24_183.html},
}

@article {Ormsby2011,
	AUTHOR = {Ormsby, Kyle M.},
	TITLE = {Motivic invariants of {$p$}-adic fields},
	JOURNAL = {J. K-Theory},
	FJOURNAL = {Journal of K-Theory. K-Theory and its Applications in Algebra,
	Geometry, Analysis {\&} Topology},
	VOLUME = {7},
	YEAR = {2011},
	NUMBER = {3},
	PAGES = {597--618},
	ISSN = {1865-2433},
	MRCLASS = {14F42 (19D50)},
	MRNUMBER = {2811717},
	MRREVIEWER = {Matthias Wendt},
	DOI = {10.1017/is011004017jkt153},
	URL = {https://doi.org/10.1017/is011004017jkt153},
}

@article {Ricka15,
    AUTHOR = {Ricka, Nicolas},
     TITLE = {Subalgebras of the {$\mathbb{Z}/2$}-equivariant {S}teenrod
              algebra},
   JOURNAL = {Homology Homotopy Appl.},
  FJOURNAL = {Homology, Homotopy and Applications},
    VOLUME = {17},
      YEAR = {2015},
    NUMBER = {1},
     PAGES = {281--305},
      ISSN = {1532-0073,1532-0081},
   MRCLASS = {55S10 (55S91)},
  MRNUMBER = {3350083},
MRREVIEWER = {Haynes\ R.\ Miller},
       DOI = {10.4310/HHA.2015.v17.n1.a14},
       URL = {https://doi.org/10.4310/HHA.2015.v17.n1.a14},
}

@article {SanW2il022,
    AUTHOR = {Sankar, Krishanu and Wilson, Dylan},
     TITLE = {On the {$C_p$}-equivariant dual {S}teenrod algebra},
   JOURNAL = {Proc. Amer. Math. Soc.},
  FJOURNAL = {Proceedings of the American Mathematical Society},
    VOLUME = {150},
      YEAR = {2022},
    NUMBER = {8},
     PAGES = {3635--3647},
      ISSN = {0002-9939,1088-6826},
   MRCLASS = {55S91 (55N91 55P91 55S10)},
  MRNUMBER = {4439482},
       DOI = {10.1090/proc/15846},
       URL = {https://doi.org/10.1090/proc/15846},
}

@article {LiPetersenTatum23,
    AUTHOR = {Li, Guchuan and Petersen, Sarah and Tatum, Elizabeth},
     TITLE = {A {T}hom spectrum model for {$C_2$}-integral {B}rown-{G}itler spectra},
   JOURNAL = {Proc. Amer. Math. Soc.},
  FJOURNAL = {Proceedings of the American Mathematical Society},
    VOLUME = {153},
      YEAR = {2025},
    NUMBER = {12},
     PAGES = {5421--5435},
      ISSN = {0002-9939,1088-6826},
   MRCLASS = {55P42 (55P91 55P92)},
  MRNUMBER = {4989653},
       DOI = {10.1090/proc/17339},
       URL = {https://doi.org/10.1090/proc/17339},
}

@book{margolis2011spectra,
  title={Spectra and the Steenrod algebra: modules over the Steenrod algebra and the stable homotopy category},
  author={Margolis, Harvey Robert},
  year={2011},
  publisher={Elsevier}
}

@article {GheorgheIsaksenRicka18,
    AUTHOR = {Gheorghe, Bogdan and Isaksen, Daniel C. and Ricka, Nicolas},
     TITLE = {The {P}icard group of motivic {$\mathcal{A}_{\mathbb{C}}(1)$}},
   JOURNAL = {J. Homotopy Relat. Struct.},
  FJOURNAL = {Journal of Homotopy and Related Structures},
    VOLUME = {13},
      YEAR = {2018},
    NUMBER = {4},
     PAGES = {847--865},
      ISSN = {2193-8407,1512-2891},
   MRCLASS = {14F42 (14C22 20C05)},
  MRNUMBER = {3870774},
MRREVIEWER = {Benjamin\ Collas},
       DOI = {10.1007/s40062-018-0200-z},
       URL = {https://doi.org/10.1007/s40062-018-0200-z},
}

@book {BrzezinskiWisbauer03,
    AUTHOR = {Brzezinski, Tomasz and Wisbauer, Robert},
     TITLE = {Corings and comodules},
    SERIES = {London Mathematical Society Lecture Note Series},
    VOLUME = {309},
 PUBLISHER = {Cambridge University Press, Cambridge},
      YEAR = {2003},
     PAGES = {xii+476},
      ISBN = {0-521-53931-5},
   MRCLASS = {16W30},
  MRNUMBER = {2012570},
MRREVIEWER = {Fabio\ Gavarini},
       DOI = {10.1017/CBO9780511546495},
       URL = {https://doi.org/10.1017/CBO9780511546495},
}

@misc{GuillouIsaksen24,
      title={{$C_2$}-Equivariant Stable Stems}, 
      author={Bertrand J. Guillou and Daniel C. Isaksen},
      year={2024},
      eprint={2404.14627},
      archivePrefix={arXiv},
      primaryClass={math.AT},
      url={https://arxiv.org/abs/2404.14627}, 
}

@article {DavisMahowald89,
    AUTHOR = {Davis, Donald M. and Mahowald, Mark},
     TITLE = {The image of the stable {$J$}-homomorphism},
   JOURNAL = {Topology},
  FJOURNAL = {Topology. An International Journal of Mathematics},
    VOLUME = {28},
      YEAR = {1989},
    NUMBER = {1},
     PAGES = {39--58},
      ISSN = {0040-9383},
   MRCLASS = {55Q50 (55T15)},
  MRNUMBER = {991098},
MRREVIEWER = {Donald\ W.\ Kahn},
       DOI = {10.1016/0040-9383(89)90031-1},
       URL = {https://doi.org/10.1016/0040-9383(89)90031-1},
}

@article {Gonzalez00,
    AUTHOR = {Gonz\'alez, Jes\'us},
     TITLE = {A vanishing line in the {${\rm BP}\langle 1\rangle$}-{A}dams
              spectral sequence},
   JOURNAL = {Topology},
  FJOURNAL = {Topology. An International Journal of Mathematics},
    VOLUME = {39},
      YEAR = {2000},
    NUMBER = {6},
     PAGES = {1137--1153},
      ISSN = {0040-9383},
   MRCLASS = {55T15 (18G25 19L41 55N20 55P42)},
  MRNUMBER = {1783851},
MRREVIEWER = {J.\ P. C. Greenlees},
       DOI = {10.1016/S0040-9383(99)00002-6},
       URL = {https://doi.org/10.1016/S0040-9383(99)00002-6},
}

@incollection {Davis87,
    AUTHOR = {Davis, Donald M.},
     TITLE = {The {$b{\rm o}$}-{A}dams spectral sequence: some calculations
              and a proof of its vanishing line},
 BOOKTITLE = {Algebraic topology ({S}eattle, {W}ash., 1985)},
    SERIES = {Lecture Notes in Math.},
    VOLUME = {1286},
     PAGES = {267--285},
 PUBLISHER = {Springer, Berlin},
      YEAR = {1987},
      ISBN = {3-540-18481-3},
   MRCLASS = {55T25 (55P42 55Q40)},
  MRNUMBER = {922930},
MRREVIEWER = {Vincent\ Giambalvo},
       DOI = {10.1007/BFb0078745},
       URL = {https://doi.org/10.1007/BFb0078745},
}

@article {BeaudryBehrensBattacharyaCulverXu20,
    AUTHOR = {Beaudry, A. and Behrens, M. and Bhattacharya, P. and Culver,
              D. and Xu, Z.},
     TITLE = {On the {$E_2$}-term of the bo-{A}dams spectral sequence},
   JOURNAL = {J. Topol.},
  FJOURNAL = {Journal of Topology},
    VOLUME = {13},
      YEAR = {2020},
    NUMBER = {1},
     PAGES = {356--415},
      ISSN = {1753-8416,1753-8424},
   MRCLASS = {55Q45 (55T15)},
  MRNUMBER = {4138742},
MRREVIEWER = {Donald\ M.\ Davis},
       DOI = {10.1112/topo.12136},
       URL = {https://doi.org/10.1112/topo.12136},
}

@article{bhattacharya2024steenrod,
  title={On the Steenrod module structure of $\mathbb{R}$-motivic Spanier-Whitehead duals},
  author={Bhattacharya, Prasit and Guillou, Bertrand and Li, Ang},
  journal={Proceedings of the American Mathematical Society, Series B},
  volume={11},
  number={48},
  pages={555--569},
  year={2024}
}

@article {BrownGitler73,
    AUTHOR = {Brown, Jr., Edgar H. and Gitler, Samuel},
     TITLE = {A spectrum whose cohomology is a certain cyclic module over
              the {S}teenrod algebra},
   JOURNAL = {Topology},
  FJOURNAL = {Topology. An International Journal of Mathematics},
    VOLUME = {12},
      YEAR = {1973},
     PAGES = {283--295},
      ISSN = {0040-9383},
   MRCLASS = {55B20},
  MRNUMBER = {391071},
MRREVIEWER = {J.\ Siegel},
       DOI = {10.1016/0040-9383(73)90014-1},
       URL = {https://doi.org/10.1016/0040-9383(73)90014-1},
}

@article {BehrensWilson18,
    AUTHOR = {Behrens, Mark and Wilson, Dylan},
     TITLE = {A {$C_2$}-equivariant analog of {M}ahowald's {T}hom spectrum
              theorem},
   JOURNAL = {Proc. Amer. Math. Soc.},
  FJOURNAL = {Proceedings of the American Mathematical Society},
    VOLUME = {146},
      YEAR = {2018},
    NUMBER = {11},
     PAGES = {5003--5012},
      ISSN = {0002-9939,1088-6826},
   MRCLASS = {55P91 (55S91)},
  MRNUMBER = {3856165},
MRREVIEWER = {Samik\ Basu},
       DOI = {10.1090/proc/14175},
       URL = {https://doi.org/10.1090/proc/14175},
}

@article {CarrickHillRavenel24,
    AUTHOR = {Carrick, Christian and Hill, Michael A. and Ravenel, Douglas
              C.},
     TITLE = {The homological slice spectral sequence in motivic and real
              bordism},
   JOURNAL = {Adv. Math.},
  FJOURNAL = {Advances in Mathematics},
    VOLUME = {458},
      YEAR = {2024},
     PAGES = {Paper No. 109955, 74},
      ISSN = {0001-8708,1090-2082},
   MRCLASS = {55T25 (14F42 55N22 55P91 57T05)},
  MRNUMBER = {4798807},
       DOI = {10.1016/j.aim.2024.109955},
       URL = {https://doi.org/10.1016/j.aim.2024.109955},
}

@book {Adams74,
    AUTHOR = {Adams, J. F.},
     TITLE = {Stable homotopy and generalised homology},
    SERIES = {Chicago Lectures in Mathematics},
 PUBLISHER = {University of Chicago Press, Chicago, Ill.-London},
      YEAR = {1974},
     PAGES = {x+373},
   MRCLASS = {55B20 (55E10)},
  MRNUMBER = {402720},
MRREVIEWER = {P.\ S.\ Landweber},
}

@article{baker2001adams,
  title={On the Adams spectral sequence for R--modules},
  author={Baker, Andrew and Lazarev, Andrey},
  journal={Algebraic {\&} Geometric Topology},
  volume={1},
  number={1},
  pages={173--199},
  year={2001},
  publisher={Mathematical Sciences Publishers}
}

@article{cohen1985immersion,
  title={The immersion conjecture for differentiable manifolds},
  author={Cohen, Ralph L},
  journal={Annals of Mathematics},
  volume={122},
  number={2},
  pages={237--328},
  year={1985},
  publisher={JSTOR}
}

@article{boardman1999conditionally,
  title={Conditionally convergent spectral sequences},
  author={Boardman, J Michael},
  journal={Contemporary Mathematics},
  volume={239},
  pages={49--84},
  year={1999},
  publisher={Providence, RI: American Mathematical Society}
}

@article{lellmann1984operations,
  title={Operations and Cooperations in Odd-Primary Connective K-Theory},
  author={Lellmann, Wolfgang},
  journal={Journal of the London Mathematical Society},
  volume={2},
  number={3},
  pages={562--576},
  year={1984},
  publisher={Wiley Online Library}
}

@misc{burklund2023quiversadamsspectralsequence,
      title={Quivers and the Adams spectral sequence}, 
      author={Robert Burklund and Piotr Pstrągowski},
      year={2023},
      eprint={2305.08231},
      archivePrefix={arXiv},
      primaryClass={math.AT},
      url={https://arxiv.org/abs/2305.08231}, 
}

@article {May20,
    AUTHOR = {May, Clover},
     TITLE = {A structure theorem for {$RO(C_2)$}-graded {B}redon
              cohomology},
   JOURNAL = {Algebr. Geom. Topol.},
  FJOURNAL = {Algebraic {\&} Geometric Topology},
    VOLUME = {20},
      YEAR = {2020},
    NUMBER = {4},
     PAGES = {1691--1728},
      ISSN = {1472-2747,1472-2739},
   MRCLASS = {55N91},
  MRNUMBER = {4127082},
MRREVIEWER = {Nansen\ Petrosyan},
       DOI = {10.2140/agt.2020.20.1691},
       URL = {https://doi.org/10.2140/agt.2020.20.1691},
}

@misc{KuKrizSomberZou23,
      title={The $\mathbb{Z}/p$-equivariant dual Steenrod algebra for an odd prime $p$}, 
      author={Po Hu and Igor Kriz and Petr Somberg and Foling Zou},
      year={2023},
      eprint={2205.13427},
      archivePrefix={arXiv},
      primaryClass={math.AT},
      url={https://arxiv.org/abs/2205.13427}, 
}

@article {Lewis95,
    AUTHOR = {Lewis, Jr., L. Gaunce},
     TITLE = {Change of universe functors in equivariant stable homotopy
              theory},
   JOURNAL = {Fund. Math.},
  FJOURNAL = {Fundamenta Mathematicae},
    VOLUME = {148},
      YEAR = {1995},
    NUMBER = {2},
     PAGES = {117--158},
      ISSN = {0016-2736,1730-6329},
   MRCLASS = {55P91 (55P20 55Q91)},
  MRNUMBER = {1360142},
MRREVIEWER = {J.\ M.\ Boardman},
       DOI = {10.4064/fm-148-2-117-158},
       URL = {https://doi.org/10.4064/fm-148-2-117-158},
}

@Article{HahnSengerWilson23,
 Author = {Hahn, Jeremy and Senger, Andrew and Wilson, Dylan},
 Title = {Odd primary analogs of real orientations},
 FJournal = {Geometry {\&} Topology},
 Journal = {Geom. Topol.},
 ISSN = {1465-3060},
 Volume = {27},
 Number = {1},
 Pages = {87--129},
 Year = {2023},
 Language = {English},
 DOI = {10.2140/gt.2023.27.87},
 zbMATH = {7688325},
 Zbl = {1523.55011}
}

@article{klippenstein_splitting,
  title={The $bu$ Cohomology of {$BP \langle n \rangle$}},
  author={Klippenstein, John},
  journal={Journal of Pure and Applied Algebra},
  volume={57},
  number={2},
  pages={127--140},
  year={1989},
  publisher={Elsevier}
}

@article{wilson1975omega,
  title={The {$\Omega$}-spectrum for Brown-Peterson cohomology part II},
  author={Wilson, W Stephen},
  journal={American Journal of Mathematics},
  volume={97},
  number={1},
  pages={101--123},
  year={1975},
  publisher={JSTOR}
}

@incollection {Boardman82,
    AUTHOR = {Boardman, J. M.},
     TITLE = {The eightfold way to {BP}-operations or {$E\sb\ast E$}\ and
              all that},
 BOOKTITLE = {Current trends in algebraic topology, {P}art 1 ({L}ondon,
              {O}nt., 1981)},
    SERIES = {CMS Conf. Proc.},
    VOLUME = {2},
     PAGES = {187--226},
 PUBLISHER = {Amer. Math. Soc., Providence, RI},
      YEAR = {1982},
      ISBN = {0-8218-6001-1},
   MRCLASS = {55N20 (55S05)},
  MRNUMBER = {686116},
MRREVIEWER = {J.\ F.\ Adams},
}

@book{cartan1999homological,
  title={Homological algebra},
  author={Cartan, Henri and Eilenberg, Samuel},
  volume={19},
  year={1999},
  publisher={Princeton university press}
}

@article {Mahowald82,
    AUTHOR = {Mahowald, Mark},
     TITLE = {The image of {$J$}\ in the {$EHP$}\ sequence},
   JOURNAL = {Ann. of Math. (2)},
  FJOURNAL = {Annals of Mathematics. Second Series},
    VOLUME = {116},
      YEAR = {1982},
    NUMBER = {1},
     PAGES = {65--112},
      ISSN = {0003-486X},
   MRCLASS = {55Q40},
  MRNUMBER = {662118},
MRREVIEWER = {Donald\ M.\ Davis},
       DOI = {10.2307/2007048},
       URL = {https://doi.org/10.2307/2007048},
}

@misc{MorrisPetersenTatum2025,
      title={Splittings of truncated motivic Brown--Peterson cooperations algebras}, 
      author={Jackson Morris and Sarah Petersen and Elizabeth Tatum},
      year={2025},
      eprint={2509.19542},
      archivePrefix={arXiv},
      primaryClass={math.AT},
      url={https://arxiv.org/abs/2509.19542}, 
}

@article{carlsson1980stable,
  title={On the stable splitting of {$bo \wedge bo$} and torsion operations in connective K-theory},
  author={Carlsson, Gunnar},
  journal={Pacific Journal of Mathematics},
  volume={87},
  number={2},
  pages={283--297},
  year={1980},
  publisher={Mathematical Sciences Publishers}
}
\end{document}